\documentclass[12 pt]{amsart}

\usepackage{hyperref}
\usepackage{etex}
\usepackage[shortlabels]{enumitem}
\usepackage{amsmath}
\usepackage{amsxtra}
\usepackage{amscd}
\usepackage{amsthm}
\usepackage{adjustbox}
\usepackage{amsfonts}
\usepackage{amssymb}
\usepackage{eucal}
\usepackage[all]{xy}
\usepackage{graphicx}
\usepackage{tikz-cd}
\usepackage{mathrsfs}
\usepackage{subfiles}
\usepackage{mathpazo}
\usepackage[colorinlistoftodos, textsize=tiny]{todonotes}
\usepackage{morefloats}
\usepackage{pdfpages}
\usepackage{thm-restate}
\usepackage[percent]{overpic}
\usepackage[utf8]{inputenc}
\usepackage{epigraph}
\usepackage{csquotes}
\usepackage[margin=1in]{geometry}
\usepackage{adjustbox}
\usepackage{microtype}
\usepackage{stmaryrd}
\usepackage{tikz}
\usetikzlibrary{shapes,arrows.meta}
\usepackage{tikz-3dplot} 
\usetikzlibrary{fadings}
\usetikzlibrary{decorations.pathmorphing, decorations.markings}

\usepackage{verbatim}
\usepackage{stmaryrd}
\usepackage{scalerel}
\usepackage{stackengine}
\stackMath
\newcommand\reallywidehat[1]{%
\savestack{\tmpbox}{\stretchto{%
  \scaleto{%
    \scalerel*[\widthof{\ensuremath{#1}}]{\kern-.6pt\bigwedge\kern-.6pt}%
    {\rule[-\textheight/2]{1ex}{\textheight}}
  }{\textheight}%
}{0.5ex}}%
\stackon[1pt]{#1}{\tmpbox}%
}
\usepackage{mathtools}

\graphicspath{ {images/} }

\RequirePackage{color}
\definecolor{myred}{rgb}{0.75,0,0}
\definecolor{mygreen}{rgb}{0,0.5,0}
\definecolor{myblue}{rgb}{0,0,0.65}

\usepackage{color}

\usepackage{hyperref}
\hypersetup{citecolor=blue}
\usepackage{tikz}
\usetikzlibrary{matrix,arrows,decorations.pathmorphing}

\theoremstyle{plain}
\newtheorem{theorem}[subsubsection]{Theorem}

\newtheorem{proposition}[subsubsection]{Proposition}
\newtheorem{proposition-definition}[subsubsection]{Proposition-Definition}

\newtheorem{corollary}[subsubsection]{Corollary}

\theoremstyle{definition}
\newtheorem{definition}[subsubsection]{Definition}
\newtheorem{remark}[subsubsection]{Remark}
\newtheorem{example}[subsubsection]{Example}

\newtheorem{question}[subsubsection]{Question}
\newtheorem{conjecture}[subsubsection]{Conjecture} 

\newtheorem{slogan}[subsubsection]{Slogan}

\newtheorem*{claim*}{Claim}
\theoremstyle{remark}

\numberwithin{equation}{section}
\newcommand\nc{\newcommand}
\nc\on{\operatorname}
\nc\renc{\renewcommand}

\nc\mf\mathfrak
\nc\mc\mathcal
\nc\mb\mathbb
\nc\msf\mathsf
\nc\mscr\mathscr

\title{Motives, mapping class groups, and monodromy}
\author{Daniel Litt}
\date{\today}

\begin{document}

\begin{abstract}
We survey some recent developments at the interface of algebraic geometry, surface topology, and the theory of ordinary differential equations. Motivated by ``non-abelian" analogues of standard conjectures on the cohomology of algebraic varieties, we study mapping class group actions on character varieties and their algebro-geometric avatar---isomonodromy differential equations---from the point of view of both complex and arithmetic geometry. We then collect some open questions and conjectures on these topics. These notes are an extended version of my talk at the  April 2024 Current Developments in Mathematics conference at Harvard.
\end{abstract}

\maketitle
\setcounter{tocdepth}{1}

\tableofcontents

\section{Introduction}\label{section:introduction}
The goal of these notes is to explain a number of relationships between the topology of algebraic varieties, representation theory and dynamics of mapping class groups, and the theory of algebraic differential equations. I have tried to survey a few different classical and modern perspectives on these subjects, and to put together some conjectures and questions that might give the reader a sense of some of the directions the area is heading.

The basic question motivating this work is: how is the geometry of an algebraic variety $X$ reflected in the structure  of its fundamental group, $\pi_1(X)$, and in the representation theory of $\pi_1(X)$? We will see shortly that this question is the modern descendent of some very classical questions about ordinary differential equations, braid groups and mapping class groups, hypergeometric functions, etc. As is traditional in algebraic geometry, we view it as a special case of a much more general question about \emph{families} of algebraic varieties: given (say) a smooth proper morphism $$f: X\to S,$$ and points $x\in X, s=f(x)\in S$, how is the geometry of $f$ reflected in the exact sequence $$\pi_1(X_s, x)\to \pi_1(X, x)\to \pi_1(S, s)\to 1,$$ in the induced outer action of $\pi_1(S, s)$ on $\pi_1(X_s)$, and in the induced action of $\pi_1(S,s)$ on conjugacy classes of representations of $\pi_1(X_s)$? 

There are a number of basic examples of morphisms $f$ as above that the reader would do well to keep in mind: namely the maps $$\mathscr{C}_{g,n}\to \mathscr{M}_{g,n}$$ from the universal $n$-pointed curve of genus $g$ to the Deligne-Mumford moduli space of smooth $n$-pointed curves of genus $g$. We will return to this fundamental example throughout this survey, as in this case the questions we consider are closely related to important classical questions in surface topology, through the natural identification of $\pi_1(\mathscr{M}_{g,n})$ with the (pure) mapping class group of an $n$-pointed surface of genus $g$, and to important classical questions in the theory of ordinary differential equations, when $g=0$.

\subsection{A reader's guide}
We hope the material covered here will appeal to mathematicians with interests in algebraic and arithmetic geometry, dynamics, or surface topology. We have tried to write \autoref{section:matrices}, which discusses a classical and elementary question about dynamics of $2\times 2$ matrices (which arises when one specializes the more general questions considered here to the case of the variety $X=\mathbb{CP}^1\setminus\{x_1, \cdots, x_n\}$), with a broad mathematical audience in mind. \autoref{section:canonical} discusses the generalization of this question to surfaces of arbitrary genus: namely, the analysis of finite orbits of the mapping class group action on the character variety of a $n$-punctured surface of genus $g$. While the methods of proof are somewhat technical---relying as they do on non-abelian Hodge theory and input from the Langlands program---we hope that the questions considered, and their answers, will still be of broad interest. This section also introduces the connection to certain algebraic differential equations, the so-called \emph{isomonodromy} differential equations, examples of which include the Painlev\'e VI equation and the Schlesinger system.

\autoref{section:p-curvature} and \autoref{section:non-abelian-conjectures} will primarily be of interest to algebraic and arithmetic geometers. In \autoref{section:p-curvature} we give a conjectural (arithmetic) answer to \emph{all} questions about finite orbits of the actions of fundamental groups of algebraic varieties on character varieties, and algebraic solutions to isomonodromy equations, and sketch a proof for ``Picard-Fuchs" initial conditions. In \autoref{section:non-abelian-conjectures}, we give some philosophical motivation for these questions by analogy to standard conjectures on algebraic cycles (the Hodge conjecture, Tate conjecture, and so on) and enumerate a number of questions on the arithmetic and algebraic geometry of character varieties, suggested by this analogy.

Finally, in \autoref{section:mcg}, we return to questions about mapping class groups and their representations, and explain the connection with a number of basic open questions about vector bundles on algebraic curves. This last section should be of interest to both complex algebraic geometers and surface topologists, and we have done our best to make it accessible to readers from either of these backgrounds.

The last two sections, \autoref{section:non-abelian-conjectures} and \autoref{section:mcg}, are filled with questions and conjectures. We hope that these will give the reader a sense of where the subject is heading.  
\subsection{Acknowledgments}
Everything new in this paper is joint work, with various subsets of $\{\text{Josh Lam, Aaron Landesman, Will Sawin}\}$. I am extremely grateful to them for the many, many ideas they have contributed to this work. I would also like to  acknowledge the enormous intellectual debt the work here owes to H\'el\`ene Esnault, Michael Groechenig, Nick Katz, Mark Kisin, and Carlos Simpson. I am also very grateful to Josh Lam, Aaron Landesman, and Salim Tayou for many useful comments. In particular many of the subjects covered here (particularly those in \autoref{section:non-abelian-conjectures}) are discussed from a somewhat different point of view in \cite{esnault-local-systems}, which we enthusiastically recommend. This work was supported by the NSERC Discovery Grant, ``Anabelian methods in arithmetic and algebraic geometry" and by a Sloan Research Fellowship.
\section{Some questions about $n$-tuples of matrices}\label{section:matrices} To bring things down to earth, we start with an example that will demonstrate many of the features of the general situation discussed in \autoref{section:introduction}. Take $X$ to be the simplest algebraic variety with interesting fundamental group, i.e.~ $$X=\mathbb{CP}^1\setminus \{x_1, \cdots, x_n\},$$ where $x_1, \cdots, x_n$ are distinct points. 

\begin{figure}
\begin{tikzpicture}[scale=3, decoration={
    markings,
    mark=at position 0.5 with {\arrow{>}}}
    ]
    \shade[ball color = blue!40, opacity = 0.4] (0,0) circle (1cm);
    
    \def\n{5}
    \def\loopsize{0.2} 
    \def\loopwidth{0.05} 
   
    \foreach \i in {1,...,\n}{
        \pgfmathsetmacro{\angle}{360/\n * \i}
        
        \coordinate (Puncture) at (\angle:.8cm);
        
        \fill[black] (Puncture) circle (0.5mm);
        \node at (\angle:.68cm) {$x_{\i}$};
        
        \pgfmathsetmacro{\loopI}{\angle+360/\n/3}
          \pgfmathsetmacro{\loopII}{\angle-360/\n/3}
                    \pgfmathsetmacro{\looplabel}{\angle-360/\n/2.5}

        
        \draw plot [smooth, tension=1] coordinates { (0,0) (\loopI:.6cm) (\angle:.9cm) (\loopII:.6cm) (0,0)};
        \node at (\looplabel:.7cm) {$\gamma_{\i}$};
    }
\end{tikzpicture}
\caption{Generators of the fundamental group of a $5$-times punctured curve of genus zero, satisfying the relations $\prod \gamma_i=\on{id}$.}\label{figure:pi1-genus-0}
\end{figure}
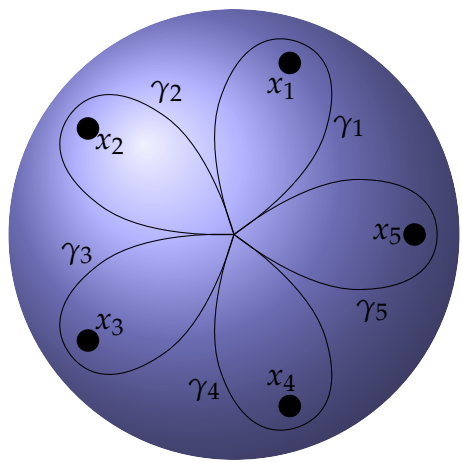

The fundamental group of $X$ has the presentation $$\pi_1(X)=\langle \gamma_1, \cdots, \gamma_n \mid \prod_i \gamma_i=1\rangle,$$ with $\gamma_i$ a loop around $x_i$ (as in \autoref{figure:pi1-genus-0}), and hence a representation $$\pi_1(X)\to \on{GL}_r(\mathbb{C})$$ is the same\footnote{Explicitly, set $A_i$ to be the image of $\gamma_i$ in $\on{GL}_r(\mathbb{C})$.} as an $n$-tuple of $r\times r$ invertible matrices $(A_1, \cdots, A_n)$ such that \begin{equation}
\label{equation:product}
 	\prod_{i=1}^n A_i=\text{id}.
\end{equation}

Set $$Y(0, n, r)=\left\{(A_1, \cdots, A_n)\in \on{GL}_r(\mathbb{C})^n \text{ such that } \prod_{i=1}^n A_i=\text{id}\right\}/\on{GL}_r(\mathbb{C})$$ where $\on{GL}_r(\mathbb{C})$ acts on an $n$-tuple $(A_1, \cdots, A_n)$ by simultaneous conjugation, i.e.~ $$B\cdot (A_1, \cdots, A_n)=(BA_1B^{-1}, \cdots, BA_nB^{-1}).$$ That is, $Y(0, n, r)$ is the set of conjugacy classes of $n$-tuples of $r\times r$ invertible complex matrices whose product is the identity matrix, or equivalently the set of conjugacy classes of $r$-dimensional representations of $\pi_1(X)$.\footnote{Here the index $0$ in the notation $Y(0,n,r)$ indicates that this is the genus $0$ version of a more general problem, which we will encounter later in the notes.}

Given conjugacy classes $C_1, \cdots, C_n\subset \on{GL}_r(\mathbb{C})$ we write $\underline{C}$ for the tuple $(C_1, \cdots, C_n)$, and set 
$$Y(\underline{C})=\left\{(A_1, \cdots, A_n) \middle| \prod_{i=1}^n A_i=\text{id}\text{ and } A_i\in C_i \text{ for all }i\right\}\Big/\text{simultaneous conjugation}$$
 That is, $Y(\underline{C})\subset Y(0,n,r)$ is the set of (simultaneous) conjugacy classes of $n$-tuples of matrices $(A_1, \cdots, A_n)$ satisfying the equation \eqref{equation:product},
 subject to the constraint that $A_i$ lies in $C_i$ for each $i$, or equivalently the set of conjugacy classes of representations $$\rho: \pi_1(X)\to \on{GL}_r(\mathbb{C})$$ such that $\rho(\gamma_i)\in C_i$ for each $i$.
 
 We denote by $Y(0,n,r)^{\text{irr}},$ (resp.~$Y(\underline{C})^{\text{irr}}$) the subsets of $Y(0,n,r)$ (resp.~$Y(\underline{C})$) corresponding to irreducible representations of $\pi_1(X)$. Note that these sets are naturally topological spaces (and indeed, they have a lot more structure, as we will see later). For example, $Y(\underline{C})$ inherits a natural (quotient) topology: it is a quotient of the subset of $C_1\times \cdots\times C_n$ consisting of tuples satisfying \eqref{equation:product}.

Three questions immediately present themselves:
\begin{enumerate}
	\item (Existence) For which $\underline{C}$ is $Y(\underline{C})^{\text{irr}}$ non-empty? This question is a form of the Deligne-Simpson problem, surveyed nicely in \cite{kostov-DS}, and has a number of variants; for example, one might ask that the $A_i$ all lie in the unitary group, or in some other subgroup of $\on{GL}_r(\mathbb{C})$.
	\item (Uniqueness) For which $\underline{C}$ is $Y(\underline{C})^{\text{irr}}$ a singleton? That is, when is a solution to \eqref{equation:product} determined uniquely (up to simultaneous conjugation) by the conjugacy classes $C_i$ of the matrices $A_i$? For reasons that will soon become clear, the question of classifying tuples $\underline{C}$ such that $Y(\underline{C})$ is a singleton is typically referred to as the classification of \emph{rigid local systems}, and was studied by Katz in his book of the same name, \cite{katz-rigid}. 
	\item (Geometry and dynamics) When $Y(\underline{C})$ is not a singleton, what does it look like? In particular, $Y(\underline{C})$ has a huge group of symmetries: given one solution $(A_1, \cdots, A_n)$ to \eqref{equation:product}, one may produce others via the  operations $$\sigma_{i}: (A_1, \cdots, A_n)\mapsto (A_1, \cdots, A_iA_{i+1}A_i^{-1}, A_i, \cdots, A_n),$$ where $i$ ranges from $1$ to $n-1$. The group $B_n:=\langle \sigma_1, \cdots, \sigma_{n-1}\rangle$ acts on $Y(0,n,r)$, permuting the $C_i$, hence an index $n!$ subgroup preserves each $Y(\underline{C})$. The study of the dynamics of this action goes back to work of Markoff \cite{markoff1879formes, markoff1880formes} and Fricke and Klein \cite{fricke-klein} in the 19th century; we will be interested in the most basic questions about this action, e.g.~classifying finite orbits, invariant subvarieties, and so on.

Note that the question of classifying finite orbits generalizes (2) above: if $Y(\underline{C})^{\text{irr}}$ is a singleton, it necessarily has finite orbit under the action of the group $B_n$.
\end{enumerate}
The existence question (1) and uniqueness question (2) above are reasonably well-understood, as we shortly explain. Question (3) has a long history and we are quite far from any kind of answer, even in the case of $2\times 2$ matrices. Our purpose in this section is to summarize some recent progress on this very special case.
\subsection{Existence and uniqueness}\label{subsection:existence-uniqueness}
We will not say much about the Deligne-Simpson problem---i.e.~the problem of determining whether or not $Y(\underline{C})^{\text{irr}}$ is non-empty---except to note that what we know about it (in particular, a complete solution in case the conjugacy classes satisfy a suitable genericity condition), due to work of Simpson \cite{simpson-products}, Kostov \cite{kostov-ds-soln}, Crawley-Boevey \cite{cb-matrices, cb-ds, cb-ds-icm}, and others, is closely tied to algebraic geometry. For example, the existence of points of $Y(\underline{C})$ corresponding to irreducible unitary representations of $\pi_1(X)$ is equivalent to the existence of stable parabolic bundles on $\mathbb{CP}^1$ with prescribed local data, and this point of view leads to a complete solution in this case, in terms of the quantum Schubert calculus of the Grassmannian; the case of rank $2$ was worked out by Biswas \cite{biswas-parabolic}, and the general case by Agnihotri-Woodward \cite{aw-unitary-quantum-schubert} and Belkale \cite{belkale-unitary}.

Regarding uniqueness, we briefly summarize Katz's classification \cite[\S6]{katz-rigid} of tuples of conjugacy classes $\underline{C}$ for which $Y(\underline{C})^{\text{irr}}$ is a singleton, as his method will have some relevance later. We say a tuple of conjugacy classes $(C_1, \cdots, C_n)\subset \on{GL}_r(\mathbb{C})^n$ is \emph{rigid} if $Y(\underline{C})^{\text{irr}}$ is a singleton; we will also refer to the corresponding conjugacy class of irreducible representations $$\rho: \pi_1(X)\to \on{GL}_r(\mathbb{C})$$ with $\rho(\gamma_i)\in C_i$ as a \emph{rigid representation}, and the corresponding local system on $X$ as a \emph{rigid local system}.

For each $\lambda\in \mathbb{C}^\times\setminus\{1\}$, Katz produces a functor $$\on{MC}_\lambda: \on{Rep}(\pi_1(X))\to \on{Rep}(\pi_1(X))$$ with the following properties: 
\begin{itemize}
	\item If $\rho$ is a rigid irreducible $\pi_1(X)$-representation, $\on{MC}_\lambda(\rho)$ is rigid and irreducible.
	\item The functors $\on{MC}_\lambda, \on{MC}_{\lambda^{-1}}$ are quasi-inverse.
	\item If $\rho$ is a rigid irreducible $\pi_1(X)$-representation of rank at least $2$, there exists a rank one representation $$\chi: \pi_1(X)\to \mathbb{C}^\times$$ and a scalar $\lambda \in \mathbb{C}^\times\setminus \{1\}$ such that $$\on{rk} \on{MC}_\lambda(\rho\otimes \chi)< \on{rk} \rho.$$
\end{itemize}
If a representation $\rho$ corresponds to a tuple of matrices $(A_1, \cdots, A_n)$, we will also write $\on{MC}_\lambda(A_1, \cdots, A_n)$ for a tuple of matrices corresponding to $\on{MC}_\lambda(\rho)$. Note that $\on{MC}_\lambda$ does not in general preserve the dimension of a representation.

The upshot of this construction is as follows: given a rigid irreducible $\pi_1(X)$-representation $\rho$, one may find a sequence of rank one representations $\chi_i: \pi_1(X)\to \mathbb{C}^\times$ and $\lambda_i\in \mathbb{C}^\times\setminus \{1\}$, $i=1,\cdots,$ such that, setting $\rho_0=\rho$ and $\rho_i=\on{MC}_{\lambda_i}(\rho_{i-1}\otimes\chi_i),$ we have $$1\leq \on{rk}(\rho_i)<\on{rk}(\rho_{i-1})$$ and hence eventually $\on{rk}(\rho_t)=1$ for some $t$. As the operations $\on{MC}_\lambda$ and $-\otimes \chi$ are invertible (and rank one representations are necessarily rigid, as a $1\times 1$ matrix is determined by its conjugacy class), this tells us that every rigid irreducible is ``generated" by rank one representations under these operations.

It turns out that the conjugacy class of $\on{MC}_\lambda(\rho)(\gamma_i)$ depends only on $\lambda$ and the conjugacy class of $\rho(\gamma_i)$. Thus given a tuple of conjugacy classes $C_1, \cdots, C_n \subset \on{SL}_r(\mathbb{C})$, there exists another (explicit) tuple $C_1', \cdots, C_n'\subset \on{SL}_{r'}(\mathbb{C})$ such that $\on{MC}_{\lambda}$ induces a map $$Y(\underline{C})\to Y(\underline{C'}).$$
Katz's description of the functor $\on{MC}_\lambda$ is algebro-geometric in nature; we explain a variant in \autoref{subsubsection:middle-convolution}. There are now a number of different expositions of Katz's classification from various more algebraic points of view, notably \cite{dettweiler-reiter-katz} and \cite{volklein}, which make the middle convolution operation completely explicit and computable.
\subsection{Dynamics}
We now turn to the dynamics of the $\langle \sigma_1, \cdots,\sigma_{n-1}\rangle$ action on $Y(0, n, r)$. How does this dynamics arise? A hint can be found in the observation that the two tuples $$\sigma_i\sigma_{i+1}\sigma_i\cdot (A_1, \cdots, A_n) \text{ and } \sigma_{i+1}\sigma_i\sigma_{i+1}\cdot (A_1, \cdots, A_n)$$ are conjugate for $1\leq i\leq n-1$, and $\sigma_i, \sigma_j$ commute for $|i-j|\geq 2$. That is, the action of $\langle \sigma_1, \cdots, \sigma_{n-1}\rangle$ on $Y(0,n, r)$ factors through the quotient $$\langle \sigma_1, \cdots, \sigma_{n-1}\mid \sigma_i\sigma_{i+1}\sigma_i=\sigma_{i+1}\sigma_i\sigma_{i+1} \text{ and } \sigma_i\sigma_j=\sigma_j\sigma_i \text{ for } |i-j|\geq 2\rangle,$$ which is the usual Artin presentation of the braid group, which is (up to quotienting by the center) the mapping class group $$\on{Mod}_{0,n}:=\pi_0(\on{Homeo}^+(\mathbb{CP}^1\setminus\{x_1, \cdots, x_n\}))$$ of orientation-preserving self-homeomorphisms of $\mathbb{CP}^1\setminus\{x_1, \cdots, x_n\}$, up to isotopy. This latter group has an obvious outer action on $\pi_1(\mathbb{CP}^1\setminus\{x_1, \cdots, x_n\})$,\footnote{Here $\on{Mod}_{0,n}$ only has an \emph{outer} action, as opposed to an honest action, as a self-homeomorphism of $\mathbb{CP}^1\setminus\{x_1, \cdots, x_n\}$ will typically not preserve a basepoint.} induced by the action of $\on{Homeo}^+(\mathbb{CP}^1\setminus\{x_1, \cdots, x_n\})$ on $\mathbb{CP}^1\setminus\{x_1, \cdots, x_n\}$ and hence acts on conjugacy classes of representations of $\pi_1(\mathbb{CP}^1\setminus\{x_1, \cdots, x_n\})$, i.e.~points of $Y(0,n,r)$.

There is an evident surjection $$\on{Mod}_{0,n}\to S_n$$ given by sending $\sigma_i$ to the transposition $(i, ~i+1)$ (or equivalently induced by the action of $\on{Homeo}^+(\mathbb{CP}^1\setminus\{x_1, \cdots, x_n\})$ on $\{x_1, \cdots, x_n\}$). Set $\on{PMod}_{0,n}$ to be the kernel of this surjection---the pure mapping class group. It turns out that $\on{PMod}_{0,n}$ is the fundamental group of the moduli space $\mathscr{M}_{0,n}$ of genus zero curves with $n$ marked points, and the outer action $$\on{PMod}_{0,n}\to \on{Out}(\pi_1(\mathbb{CP}^1\setminus\{x_1, \cdots, x_n\}))$$ is induced by the short exact sequence (a special case of the Birman exact sequence) $$1\to \pi_1(\mathbb{CP}^1\setminus\{x_1, \cdots, x_n\})\to \pi_1(\mathscr{M}_{0,n+1})\to \pi_1(\mathscr{M}_{0,n})\to 1$$ corresponding to the fibration $$\mathscr{M}_{0,n+1}\to \mathscr{M}_{0,n}$$ given by forgetting the $n+1$-st marked point. So in particular the analysis of the outer action of $\on{PMod}_{0,n}$ on $\pi_1(\mathbb{CP}^1\setminus\{x_1, \cdots, x_n\})$ fits into the paradigm with which we started these notes in \autoref{section:introduction}, namely that of analyzing the geometry of an algebraic map in terms of the induced structure of the map of fundamental groups. Note that $\on{PMod}_{0,n}$ acts on each $Y(\underline{C})$.

We begin with the most basic possible question about this action: what are the finite orbits of the action of $\on{Mod}_{0,n}$ (equivalently, of $\on{PMod}_{0,n}$), on $Y(0,n,r)$?
\subsubsection{The cases $n=0,1,2,3$} For $n=0,1$, there is almost nothing to say---the fundamental group of $\mathbb{CP}^1\setminus\{x_1, \cdots, x_n\}$ is trivial. For $n=2$ the fundamental group is $\mathbb{Z}$, and $\on{Mod}_{0,2}$ is finite. For $n=3$ the fundamental group of $\mathbb{CP}^1\setminus\{x_1, \cdots, x_3\}$ is the free group on two generators, and hence has many interesting representations. But in this case $\on{Mod}_{0,3}$ is still finite---isomorphic to the symmetric group $S_3$---and so all orbits are finite. (One may see this geometrically---the group $\on{PMod}_{0,3}$ is the fundamental group of $\mathscr{M}_{0,3}$, which is a point.)

In fact, in this last case all irreducible $2$-dimensional representations are \emph{rigid} in the sense of \autoref{subsection:existence-uniqueness}; while the dynamics are not interesting, this is the source of the (extremely rich) theory of hypergeometric functions, and the corresponding representations of $\pi_1(\mathbb{CP}^1\setminus\{x_1, x_2, x_3\})$ are precisely given by the monodromy of the hypergeometric functions ${}_2F_1(a,b,c | z)$ (see e.g.~ \cite{beukers-gauss}).
\subsubsection{The case $n=4$}\label{subsubsection:n-4} So the first interesting case of our dynamical question---that of classifying finite $\on{Mod}_{0,n}$-orbits on $Y(0,n,r)$---is the case when $n=4$, and this case is very interesting indeed. The dynamics of this situation were originally studied by Markoff in the 19th century \cite{markoff1879formes, markoff1880formes} in a different guise.

Markoff studied integer solutions to the cubic equation \begin{equation} \label{eqn:markoff} x^2+y^2+z^2-3xyz=0.\end{equation} Given one solution $(x,y,z)$ to this equation (for example $(x,y,z)=(1,1,1)$) one may produce more by ``Vieta jumping": fixing $y$ and $z$, \eqref{eqn:markoff} becomes quadratic in $x$, and hence admits another solution with the same $y, z$ coordinates, namely $$T_x(x,y,z)=(3yz-x, y,z).$$ Analogously we define $$T_y(x,y,z)=(x, 3xz-y, z), T_z(x,y,z)=(x, y, 3xy-z);$$ the group $\langle T_x, T_y, T_z\rangle$ acts on the affine cubic surface $S$ defined by \eqref{eqn:markoff}. It turns out that the complex points of the surface $S$ parametrize semisimple representations in $Y(\underline{C})$, where $$C_1=C_2=C_3=\left[\begin{pmatrix} 0 & -1 \\ 1 & 0 \end{pmatrix}\right], C_4=\left[\begin{pmatrix} 1 & 1\\ 0 & 1 \end{pmatrix}\right].$$ More generally, if $C_1,\cdots, C_4\subset \on{SL}_2(\mathbb{C}),$ one may parametrize points of $Y(\underline{C})$ corresponding to semisimple representations of $\pi_1(\mathbb{CP}^1\setminus\{x_1,\cdots, x_4\})$ by the complex points of $$X_{A,B,C,D}=\{(x,y,z\mid x^2+y^2+z^2+xyz=Ax+By+Cz+D\}$$ where $$A=\on{tr}(C_1)\on{tr}(C_2)+\on{tr}(C_3)\on{tr}(C_4),$$
$$B=\on{tr}(C_2)\on{tr}(C_3)+\on{tr}(C_1)\on{tr}(C_4),$$
$$C=\on{tr}(C_1)\on{tr}(C_3)+\on{tr}(C_2)\on{tr}(C_4), \text{ and}$$
$$D=4-\on{tr}(C_1)^2-\on{tr}(C_2)^2-\on{tr}(C_3)^2-\on{tr}(C_4)^2-\prod_{i=1}^4 \on{tr}(C_i).$$
One may similarly define an action of the group $\mathbb{Z}/2\mathbb{Z}*\mathbb{Z}/2\mathbb{Z}*\mathbb{Z}/2\mathbb{Z}$ on $X_{A,B,C,D}$ by Vieta jumping; it turns out that, up to finite index, this is the same as the natural action of $\on{Mod}_{0,4}$ on $Y(\underline{C})$ discussed above \cite[(1.5) and \S1.2]{cantat-loray}. 

Markoff was not interested in finite orbits---rather, he showed that the integral solutions to \eqref{eqn:markoff} form a single orbit under these dynamics. We will return to questions about integral points later in these notes, in \autoref{subsubsection:integral-points}; instead we now turn to the origin of our question about finite orbits.
\subsubsection{$n=4$ and the Painlev\'e VI equation} \label{subsubsection:painleve}
The Painlev\'e VI equation $\on{PVI}(\alpha, \beta,\gamma, \delta)$, discovered by Richard Fuchs \cite{fuchs-painleve}, is given by:
\begin{multline}\label{equation:PVI}
\frac{d^2y}{dt^2}=
\frac{1}{2}\left(\frac{1}{y}+\frac{1}{y-1}+\frac{1}{y-t}\right)\left( \frac{dy}{dt} \right)^2
-\left(\frac{1}{t}+\frac{1}{t-1}+\frac{1}{y-t}\right)\frac{dy}{dt} \\  +
\frac{y(y-1)(y-t)}{t^2(t-1)^2}
\left(\alpha+\beta\frac{t}{y^2}+\gamma\frac{t-1}{(y-1)^2}+\delta\frac{t(t-1)}{(y-t)^2}\right)
\end{multline}
where $y$ is a function of $t$ and $\alpha, \beta, \gamma, \delta$ are complex numbers.

Painlev\'e famously claimed, in his Stockholm lectures, \cite[Le\c{c}ons sur la th\'eorie analytique des \'equations diff\'erentielles profess\'ees \`a Stockholm]{painleve-oeuvres} that solutions to this equation are given by ``new transcendents": that is, functions that could not be expressed in terms of classical functions. While this is true for generic values of the parameters $(\alpha, \beta, \gamma, \delta)$, Painlev\'e's argument was not rigorous by modern standards, and correct proofs, largely due to Umemura and his school (see e.g.~\cite{umemura1, umemura2}), rely on a classification of the (rare) algebraic solutions (i.e.~algebraic functions satisfying \eqref{equation:PVI}) that do exist.

We shall see later in these notes (in \autoref{subsubsection:isomonodromic}) that each algebraic solution corresponds to a certain finite orbit of $\on{Mod}_{0,4}$ on $Y(0, 4, 2)$, and conversely, each finite orbit gives rise to a (countable) family of algebraic solutions to \eqref{equation:PVI}.

Algebraic solutions to the Painlev\'e VI equations (equivalently, finite orbits of $\on{Mod}_{0,4}$ on $Y(0,4,2)$) have been classified by Lisovyy and Tykhyy \cite{lisovyy-tykhyy}. Their classification builds on work of many people, including Andreev-Kitaev \cite{andreev2002transformations}, Boalch \cite{boalch-reflections, boalch-icosahedral, boalch-six, boalch2005klein, boalch-higher, boalch2007towards}, Doran \cite{doran}, Dubrovin-Mazzocco \cite{dubrovin-mazzocco}, Hitchin \cite{hitchen-poncelet, hitchin2003lecture}, and Kitaev \cite{kitaev-special}, and relies on an effective form of the Manin-Mumford conjecture for tori (originally due to Lang \cite{lang-gm}) and a somewhat involved computer computation. Up to a slightly complicated equivalence relation, which we will not summarize here, there are:
\begin{itemize}
\item four continuous families of finite orbits,
\item one countably infinite (discrete) family of finite orbits, of unbounded size, and
\item forty-five exceptional finite orbits.
\end{itemize}
\begin{example}
The countably infinite family of orbits mentioned above have representatives given by
\[
A_1= \begin{pmatrix} 1+x_2x_3/x_1 & -x_2^2/x_1 \\  x_3^2/x_1 & 1-x_2x_3/x_1 \end{pmatrix},
A_2=\begin{pmatrix} 1 & -x_1 \\ 0 & 1\end{pmatrix}, \ 
A_3=\begin{pmatrix} 1 & 0 \\ x_1 & 1\end{pmatrix}, \ 
A_4=(A_1A_2A_3)^{-1} \
\]
where 
\[
x_1=2\cos\bigg(\frac{\pi(\alpha+\beta)}{2}\bigg), \ x_2=2  \sin\bigg(\frac{\pi \alpha}{2}\bigg), \ x_3= 2 \sin\bigg(\frac{\pi \beta}{2}\bigg)
\]
for $\alpha, \beta \in \mb{Q}$. 	See \cite[Example 1.1.7]{lamlitt} for a discussion of this example, and the rest of that paper for an involved analysis of some related arithmetic questions. See also \autoref{subsection:cayley-solutions} for a brief further discussion of this example.
\end{example}

The upshot is that, at least at first glance, the classification is somewhat complicated!
\subsubsection{The case $n>4$} Not much was known about finite orbits of $\on{Mod}_{0,n}$ on $Y(0, n, 2)$ when $n>4$, with the exception of a very interesting paper of Tykhyy \cite{tykhyy-garnier} which gives a computer-aided classification when $n=5$ (though I have not yet understood the extent to which there is a rigorous proof that this classification is complete).  Aside from this there are a few sporadic constructions of finite orbits, e.g.~\cite{diarra, calligaris-mazzocco}. As before this question can be understood as asking for a classification of algebraic solutions to a certain non-linear ODE. 

In general finite orbits of the action of $\on{Mod}_{0,n}$ on $Y(0,n,r)$ correspond to algebraic solutions to the system
\begin{equation}\label{equation:schlesinger}
\begin{cases}\displaystyle
	\frac{\partial B_i}{\partial\lambda_j}=\frac{[B_i, B_j]}{\lambda_i-\lambda_j} & \text{ if } i\neq j\\
\displaystyle	\sum_i \frac{\partial B_i}{\partial \lambda_j}=0 &
\end{cases}
\end{equation}
where the $B_i$ are $\mathfrak{gl}_r(\mathbb{C})$-valued functions (see \autoref{example:schlesinger-derivation}). 
In this language the question was clearly of interest to Painlev\'e, Schlesinger, Gambier, Garnier, etc. (see for example Garnier's 1912 paper \cite{garnier}, where Garnier writes down, among other things, many classical solutions to \eqref{equation:schlesinger}), but to my knowledge the idea that a complete classification might be possible first appeared in print in work of Dubrovin-Mazzocco \cite{dubrovin-mazzocco-schlesinger}.

It turns out that one can give a quite clean classification of the finite orbits of the $\on{PMod}_{0,n}$-action on $Y(\underline{C})$, when at least one of the conjugacy classes $C_i$ of $\underline{C}$ has infinite order, using algebro-geometric techniques, as we now explain. We make the following convenient definition:
\begin{definition}\label{definition:interesting-rep}
	Say that a representation $$\rho: \pi_1(\mathbb{CP}^1\setminus\{x_1, \cdots, x_n\})\to \on{SL}_2(\mathbb{C})$$ (equivalently, an $n$-tuple of matrices $A_1, \cdots, A_n\in \on{SL}_2(\mathbb{C})$ such that $\prod_i A_i=\text{id}$) is \emph{interesting} if 
	\begin{enumerate}
		\item Its image is Zariski-dense, i.e.~it has infinite image and the image can't be conjugated into one of the subgroups $$\begin{pmatrix} * & * \\ 0 & * \end{pmatrix} \text{ or } \begin{pmatrix} * & 0 \\ 0 & * \end{pmatrix} \cup \begin{pmatrix} 0 & * \\ * & 0 \end{pmatrix},$$
		\item No $A_i$ is a scalar matrix, and
		\item The point of $Y(\underline{C})$ corresponding to $\rho$ has finite orbit under the action of $\on{PMod}_{0,n}$, and it is isolated as a finite orbit of $\on{PMod}_{0,n}$. That is, if $\Gamma\subset \on{PMod}_{0,n}$ is the stabilizer of $[\rho]\in Y(\underline{C})$, $[\rho]$ is an isolated point of $Y(\underline{C})^\Gamma$, with the natural topology inherited from $Y(\underline{C})$ via its presentation as a quotient of $C_1\times \cdots\times C_n$.
	\end{enumerate}
	We say a local system on $\mathbb{CP}^1\setminus\{x_1, \cdots, x_n\}$ is interesting if its monodromy representation is interesting.
\end{definition}
\begin{remark}
The careful reader may have noticed that we are now considering $\on{SL}_2(\mathbb{C})$-representations	rather than $\on{GL}_2(\mathbb{C})$-representations. There is no loss of generality here: given $(A_1, \cdots, A_n)\in Y(0,n,2)$ with finite $\on{Mod}_{0,n}$-orbit, one may always scale the matrices by appropriately chosen scalars so that they lie in $\on{SL}_2(\mathbb{C})$, while preserving the property that the tuple have finite $\on{Mod}_{0,n}$-orbit.
\end{remark}

The point of the conditions of \autoref{definition:interesting-rep} is that representations not satisfying these conditions have been classified classically. Let us say a word about each of the conditions before proceeding to our partial classification of interesting representations.
\begin{enumerate}
	\item The non-Zariski-dense subgroups of $\on{SL}_2(\mathbb{C})$ are either finite, can be conjugated into the Borel subgroup $$\begin{pmatrix} * & *\\ 0 & *\end{pmatrix},$$ or can be conjugated into the (infinite) dihedral group $$\begin{pmatrix} * & 0 \\ 0 & * \end{pmatrix} \cup \begin{pmatrix} 0 & * \\ * & 0 \end{pmatrix}.$$ The conjugacy class of a representation $$\rho: \pi_1(\mathbb{CP}^1\setminus \{x_1, \cdots, x_n\})\to \on{SL}_2(\mathbb{C})$$ with finite image always has finite orbit under $\on{Mod}_{0,n}$; the finite subgroups of $\on{SL}_2(\mathbb{C})$ were essentially classified by Euclid (in his classification of Platonic solids), and explicitly classified by Schwarz \cite{schwarz}.

	The representations that can be conjugated into the Borel subgroup are precisely the representations which are not irreducible. The finite orbits of such representations were classified by Cousin-Moussard \cite{cousinmoussard}, with an essentially (but perhaps not obviously) equivalent classification found by McMullen \cite[Theorem 8.1]{mcmullen-braid}. 
	
	Finally, the representations that can be conjugated into the infinite dihedral group were classified by Tykhyy \cite[p.~122, A]{tykhyy-garnier}; alternately one can deduce this classification from the main result of \cite{dacampo}.
	\item The assumption that no $A_i$ is a scalar matrix is harmless, as we now explain. A scalar matrix in $\on{SL}_2(\mathbb{C})$ is $\pm\on{id}$. If $A_i=\on{id}$, we may remove it from our tuple of matrices, replacing $n$ with $n-1$. If $A_i=-\on{id}$, let $j\neq i$ be such that $A_{j}$ is non-scalar, and consider the new tuple of matrices $A_k'$ where $A_k'=A_k$ for $k\neq i, j$, $A_i'=-A_i, A_{j}'=-A_{j}$. This new tuple has $A_{i'}=\on{id}$, and hence we may remove it as before. So it is easy to satisfy this assumption by multiplying by scalars and removing some of our matrices.
	\item It is much less obvious that the condition of \autoref{definition:interesting-rep}(3) is at all relevant to the problem. That said, results of Corlette-Simpson \cite{corlette-simpson} show that Zariski-dense representations $\rho$ with finite $\on{PMod}_{0,n}$-orbit, not satisfying this condition, have a very special form: they are so-called ``pullback solutions" studied by Doran \cite{doran}, Kitaev \cite{kitaev-special} and others and classified by Diarra \cite{diarra}. We will discuss these to some extent in \autoref{subsubsection:pullback}.
\end{enumerate}
By the discussion above, the only finite orbits that remain to be classified are the interesting ones. We can almost do so.
\begin{theorem}[Lam-Landesman-L--\, {\cite[Corollary 1.1.7]{lam2023finite}}]\label{theorem:sl2-classification}
	Suppose a conjugacy class of tuples of matrices $(A_1, \cdots, A_n)\in Y(0, n,2)$ is interesting. If some $A_i$ has infinite order, then there exists $\lambda \in \mathbb{C}^\times\setminus\{1\}$ and $\alpha_1, \cdots, \alpha_n\in \mathbb{C}^\times$ such that $$(\alpha_1A_1, \cdots, \alpha_n A_n)=\on{MC}_\lambda(B_1, \cdots, B_n),$$ where each $B_i\in GL_{n-2}(\mathbb{C})$, the group $\langle B_1, \cdots, B_n \rangle$ is an irreducible finite complex reflection group, and $B_1, \cdots, B_{n-1}$ are pseudoreflections.
\end{theorem}
Here $\on{MC}_{\lambda}$ is the middle convolution operation introduced by Katz in \cite{katz-rigid} and discussed earlier in \autoref{subsection:existence-uniqueness}. 

Note that if $n>4$, the $B_i$ are not $2\times 2$ matrices---they are $(n-2)\times(n-2)$ matrices. So we have classified finite $\on{Mod}_{0,n}$-orbits on $Y(0,n,2)$ in terms of certain finite subgroups of $\on{GL}_{n-2}(\mathbb{C})$. In what sense is this actually a classification? The point is that finite complex reflection groups were classified by Shephard and Todd \cite{shephard-todd} in 1954. We briefly recall their definition and classification; see \cite{lehrertaylor} for a modern exposition.
\begin{definition}
	A matrix $A\in \on{GL}_r(\mathbb{C})$ is a \emph{pseudoreflection} if it has finite order and $\on{rk}(A-\on{id})=1$. A subgroup $G\subset \on{GL}_r(\mathbb{C})$ is a \emph{finite complex reflection group} if $G$ is finite and generated by pseudoreflections. A finite complex reflection group $G\subset \on{GL}_r(\mathbb{C})$ is irreducible if the corresponding rank $r$ representation of $G$ is irreducible. 
\end{definition}
\begin{theorem}[Shephard-Todd\,\cite{shephard-todd}]
	There is one infinite class of finite complex reflection groups, denoted $G(m,p,n)\subset GL_n(\mathbb{C})$, where $p$ divides $m$. The group $G(m, 1, n)$ consists of all $n\times n$ matrices with exactly one non-zero entry in each row and column, where that non-zero entry is an $m$-th root of unity. The group $G(m, p, n)\subset G(m, 1,n)$ is the subgroup consisting of matrices whose non-zero entries multiply to an $m/p$-th root of unity.
	
	There are 34 exceptional irreducible finite complex reflection groups not conjugate to one of the $G(m, p, n)$, including the Weyl groups $W(E_6), W(E_7), W(E_8),$ the Valentiner group, the group $PSL_2(\mathbb{F}_7)$ with its natural 3-dimensional representation, the automorphism group of the icosahedron, and so on.
\end{theorem}
In other words, the interesting finite orbits of $\on{Mod}_{0,n}$ on $Y(0, n, 2)$ are classified in terms of the symmetry groups of some interesting ``complex polytopes.'' It is worth noting that \autoref{theorem:sl2-classification} recovers and generalizes that of Dubrovin-Mazzocco \cite{dubrovin-mazzocco} when $n=4$; this is perhaps surprising given that their motivation comes from a completely different direction (namely, as I understand it, the theory of Frobenius manifolds).
\begin{example}
The five matrices
$$A_1=\begin{pmatrix} -1 & 1 \\ 0 & -1 \end{pmatrix}, A_2=	\begin{pmatrix} -1 & 0 \\ -1 & -1 \end{pmatrix}, A_3 = \begin{pmatrix} \frac{-1-\sqrt{5}}{2} & 1 \\ \frac{-3+\sqrt{5}}{2} & \frac{-3+\sqrt{5}}{2} \end{pmatrix}, A_4= \begin{pmatrix} \frac{1-\sqrt{5}}{2} & \frac{3-\sqrt{5}}{2} \\ \frac{-3+\sqrt{5}}{2} & \frac{-5+\sqrt{5}}{2} \end{pmatrix}$$
$$A_5=(A_1A_2A_3A_4)^{-1}$$ give rise to a finite $\on{Mod}_{0,5}$-orbit on $Y(0,5,2)$ \cite[(20) on p.~39]{tykhyy-garnier}. Under the correspondence of \autoref{theorem:sl2-classification}, they are related by middle convolution to the automorphism group of the icosahedron, viewed as a subgroup of $\on{GL}_3(\mathbb{C})$ (namely $W(H_3)$, in Shephard-Todd's notation).
\end{example}

The only condition that prevents \autoref{theorem:sl2-classification} from being a complete classification is the requirement that some $A_i$ has infinite order. As we will soon see, this condition is algebro-geometrically natural (see \autoref{remark:bad-reduction}), but it would be of great interest to find a classification  without it.

Here is a very concrete corollary, which we only know how to prove using \autoref{theorem:sl2-classification} and the Shephard-Todd classification:
\begin{corollary}[Lam-Landesman-L--\,{\cite[Corollary 1.1.8]{lam2023finite}}]
	Suppose $(A_1, \cdots, A_n)$ is an interesting tuple, and some $A_i$ has infinite order. Then $n\leq 6$.
\end{corollary}
The corollary above is sharp; there exist examples of tuples with $n=6$ (see e.g.~ \cite[(21) on p.~40]{tykhyy-garnier} or \cite[\S5.8.4]{lam2023finite}.
\subsection{An algebro-geometric interpretation}\label{subsection:an-algebro-geometric-interpretation}
As we prepare to explain the idea of the proof of \autoref{theorem:sl2-classification}, let us consider an a priori slightly different question to that of classifying finite $\on{Mod}_{0,n}$-orbits on $Y(0,n,r)$. 
\begin{definition}\label{definition:geometric-origin}
	Let $X$ be a smooth complex algebraic variety. A complex local system $\mathbb{V}$ on $X$ is \emph{of geometric origin} if there exists a dense open subset $U\subset X$, a smooth proper morphism $\pi: Y\to U$, and an integer $i\geq 0$ such that $\mathbb{V}$ is a summand of $R^i\pi_*\mathbb{C}$.\footnote{Here $R^i\pi_*\mathbb{C}$ is the local system on $U$ whose fiber at $x\in U$ is $H^i(Y_x, \mathbb{C}).$}
\end{definition}
One interpretation of our fundamental question from the introduction---how does the geometry of a variety $X$ influence the structure of its fundamental group---is: can one classify local systems of geometric origin on $X$? As we will see in \autoref{section:non-abelian-conjectures}, this is in some sense the \emph{non-abelian} analogue of the question of understanding algebraic cycles on $X$, and it has conjectural answers analogous to the Hodge conjecture, the Tate conjecture, and so on.

Let us return to our very special case:
\begin{question}\label{question:geom-origin-cp1}
	Let $x_1, \cdots, x_n\subset \mathbb{CP}^1$ be $n$ \emph{generic} points. Can one classify local systems of rank 2 on $\mathbb{CP}^1\setminus\{x_1, \cdots, x_n\}$ that are of geometric origin?
\end{question}
It turns out that this question is \emph{the same} as classifying finite $\on{Mod}_{0,n}$-orbits on $Y(0,n,2)$, as we now explain. 

The following is immediate from the proof of \cite[Corollary 2.3.5]{landesman2022canonical} (note that the condition that $g\geq 1$ is unnecessary in our setting, as we are working with $\on{SL}_2(\mathbb{C})$-representations):
\begin{proposition}\label{proposition:AG-interpretation}
	A representation $\rho: \pi_1(\mathbb{CP}^1\setminus\{x_1, \cdots, x_n\})\to \on{GL}_r(\mathbb{C})$ has conjugacy class with finite $\on{Mod}_{0,n}$-orbit if and only if there exists
	\begin{itemize}
	\item 	a family of $n$-punctured curves of genus $0$, $\pi: \mathscr{X}\to \mathscr{M}$, with the induced map $\mathscr{M}\to \mathscr{M}_{0,n}$ dominant, and
	\item a local system $\mathbb{V}$ on $\mathscr{X}$ whose restriction to a fiber of $\pi$ has monodromy conjugate to $\rho$.
	\end{itemize}
\end{proposition}
Thus classifying finite $\on{Mod}_{0,n}$-orbits on $Y(0,n,r)$ is the same as classifying local systems on families of curves $\mathscr{X}$ as in \autoref{proposition:AG-interpretation}. Now a result of Corlette-Simpson \cite[Theorem 2]{corlette-simpson} and Loray-Pereira-Touzet \cite[Theorem A]{loray-etc} (see also \autoref{theorem:corlette-simpson} in these notes) tells us that Zariski-dense $\on{SL}_2$-local systems on smooth quasi-projective varieties come in two (not necessarily disjoint) flavors:
\begin{enumerate}
	\item rigid\footnote{Here the notion of rigidity is a generalization of what was discussed earlier, in section \autoref{subsection:existence-uniqueness}. Namely, these are local systems on $\mathscr{X}$ with no non-trivial deformations. We will discuss these later, in \autoref{subsection:rigid-local-systems}. These local systems correspond in our setting to those satisfying \autoref{definition:interesting-rep}(3).} local systems, which are of geometric origin, and
	\item local systems pulled back from Deligne-Mumford curves (equivalently, the corresponding $\on{PGL}_2$-local system is pulled back from an orbifold curve).
\end{enumerate}
It turns out that the local systems on $\mathscr{X}$ not of geometric origin can be classified by elementary means, as we now explain. So (after the next section) only local systems of geometric origin will remain to be classified.
\subsubsection{Pullback families}\label{subsubsection:pullback} Let us first turn to the local systems $\mathbb{V}$ of type (2) above, i.e.~those pulled back from curves. We will refer to these as being \emph{of pullback type}; we learned this construction from the paper \cite{doran}. In this case, there exists an orbifold curve $C$, a projective local system $\mathbb{W}$ on $C$, and a map $f:\mathscr{X}\to C$ so that the restriction of $f^*\mathbb{W}$ to a fiber of $\pi$ is isomorphic to $\mathbb{P}V$. In particular, this monodromy is non-trivial, so $C$ is dominated by any fiber of $\pi$, and hence has genus zero. 

Put another way, there exists a divisor $D\subset \mathbb{CP}^1$ and a projective local system $\mathbb{W}$ with Zariski-dense monodromy on $\mathbb{CP}^1\setminus D$, such that for general $\{x_1, \cdots, x_n\}\subset \mathbb{CP}^1$, there exists $g: \mathbb{CP}^1\to \mathbb{CP}^1$ so that $g^*\mathbb{W}$ is unramified away from $\{x_1, \cdots, x_n\}$. We must have $g(\{x_1, \cdots, x_n\})\subset D$. Moreover $g$ must be ramified over at least $n-3$ points away from $D$, by counting parameters (i.e., the dimension of $\mathscr{M}_{0,n}$ is $n-3$), and in particular, assuming $n>3$, we have $d=\on{deg}(g)>1$.

That $g^*\mathbb{W}$ is unramified away from $\{x_1, \cdots, x_n\}$ tells us that for $y\in D$, any point of $g^{-1}(y)$ must be either among the $x_i$, or must be ramified (of order divisible by the order of the local monodromy of $\mathbb{W}$ about $y$). Now Riemann-Hurwitz gives $$2\leq 2d-n+3-(d\deg D-n)/2=2d-n/2-(d\deg D)/2+3$$
and hence (as $n\geq 4$), $$1\leq d(2-\deg D/2).$$ So $\deg D=3$, and we may without loss of generality assume $D=\{0,1,\infty\}$.

In summary, classifying Zariski-dense local systems of pullback type on $\mathbb{CP}^1\setminus\{x_1, \cdots, x_n\}$ is the same as classifying projective local systems $\mathbb{W}$ on $\mathbb{CP}^1\setminus\{0,1,\infty\}$, and covers $g:\mathbb{CP}^1\to \mathbb{CP}^1$ branched over $\{0,1,\infty\}$ and $n-3$ auxiliary points, so that at most $n$ points $x$ with $g(x)\in \{0,1,\infty\}$ have ramification order not divisible by the local monodromy of $\mathbb{W}$. This was done by Diarra \cite{diarra}, again by a Riemann-Hurwitz computation.
\subsubsection{Local systems of geometric origin} Having handled local systems of pullback type, all that remains in our classification of interesting finite $\on{Mod}_{0,n}$-orbits on $Y(0,n,2)$ is to classify those local systems on $\mathscr{X}$ as in \autoref{proposition:AG-interpretation} of geometric origin. Before discussing this classification, let's observe that this is really the same as \autoref{question:geom-origin-cp1}. Indeed, given a family $\pi: \mathscr{X}\to \mathscr{M}$ as in \autoref{proposition:AG-interpretation} and a local system of geometric origin on $\mathscr{X}$, the restriction to a general fiber will give a local system of geometric origin on $\mathbb{CP}^1\setminus\{x_1, \cdots, x_n\}$ with the $x_i$ general. Conversely, by ``spreading out",\footnote{That is, if our local system appears in the cohomology of a family $\mathscr{Y}\to \mathbb{CP}^1\setminus\{x_1, \cdots, x_n\}$, with the $x_i$ general, then $\mathscr{Y}$ extends over a family of $n$-punctured curves of genus $0$ whose base dominates $\mathscr{M}_{0,n}$.} a local system of geometric origin on $\mathbb{CP}^1\setminus\{x_1, \cdots, x_n\}$ (with the $x_i$ general) can be extended to a family as in \autoref{proposition:AG-interpretation}. 

Here is one form of the classification:
\begin{theorem}[Lam-Landesman-L--]\label{theorem:cp1-geometric-origin}
	Let $x_1, \cdots, x_n\in \mathbb{CP}^1$ be general points. Then any non-isotrivial rank two local system $\mathbb{V}$ on $\mathbb{CP}^1\setminus \{x_1, \cdots, x_n\}$ of geometric origin, with infinite monodromy about one of the $x_i$ and non-scalar monodromy at each $x_i$, has the form $$\mathbb{V}=\on{MC}_\lambda(\mathbb{W})\otimes \mathbb{L},$$ where 
	\begin{itemize}
	\item $\lambda\in\mathbb{C}^\times\setminus\{1\}$,
	\item 	$\mathbb{W}$ has rank $n-2$ and monodromy given by a finite complex reflection group,
	\item  and $\mathbb{L}$ is a rank one local system of finite order on $\mathbb{CP}^1\setminus\{x_1, \cdots, x_n\}$. 
	\end{itemize}
\end{theorem}
\begin{remark}\label{remark:bad-reduction}
In fact Corlette-Simpson show \cite[Theorem 9.3 and Proposition 10.5]{corlette-simpson} that any non-isotrivial rank $2$ local system of geometric origin on a smooth variety $X$ arises in the cohomology of an \emph{abelian scheme of $\on{GL}_2$-type} over $X$. In this way \autoref{theorem:cp1-geometric-origin} classifies abelian schemes of $\on{GL}_2$-type over $\mathbb{CP}^1\setminus \{x_1, \cdots, x_n\}$, with the $x_i$ generic, which do not have potentially good reduction everywhere. This last condition is a geometric interpretation of the condition that the local monodromy about some $x_i$ have infinite order, and (at least in my view) makes this condition more natural.
\end{remark}
\subsubsection{Middle convolution}\label{subsubsection:middle-convolution}
For completeness (and before sketching the proof), we give a definition and discussion of middle convolution. The non-algebro-geometrically inclined reader may wish to skip this section.

\begin{definition}
	Let $\{x_1, \cdots, x_n\}\subset \mathbb{CP}^1$ be a finite set of points, with $\infty$ among the $x_i$, and set $X=\mathbb{CP}^1\setminus\{x_1, \cdots, x_n\}$. Consider the diagram
	$$\xymatrix{
	& X\times X\setminus \Delta \ar@{^(->}[rr]^j  \ar[ld]^{\pi_1} \ar[rd]^\alpha  & & \mathbb{CP}^1 \times X  \ar[rd]^{\pi_2} & \\
	X & & \mathbb{C}^\times & & X 
	}$$
	where $\Delta$ is the diagonal, $j$ is the natural inclusion, $\pi_1, \pi_2$ are projections onto the first and second factor respectively, and $\alpha$ is the map $(a,b)\mapsto a-b$. For $\lambda$ a non-zero complex number, let $\chi_\lambda$ be the rank one local system on $\mathbb{C}^\times$ with local monodromy about $0$ given by $\lambda$. Then given a local system $\mathbb{V}$ on $X$, we define $$\on{MC}_\lambda(\mathbb{V}):=R^1\pi_{2*}j_*(\pi_1^*\mathbb{V}\otimes \alpha^*\chi_\lambda).$$
\end{definition}
This definition may be appear intimidating at first glance, but it is in fact quite computable; see e.g. \cite{dettweiler-reiter-katz, dettweiler2003middle, dettweiler-reiter-painleve} for a beautiful explanation of how to perform these computations. In the case of interest, namely when $\mathbb{V}$ has finite monodromy, it even has a fairly simple geometric interpretation. Indeed, if $G$ is the monodromy group of $\mathbb{V}$, and $\lambda$ is a $c$-th root of unity, then $$\on{MC}_\lambda(\mathbb{V})$$ appears in the cohomology of a family of curves $$\mathscr{Y}\to \mathbb{CP}^1\setminus\{x_1, \cdots, x_n\},$$ where for $s\in \mathbb{CP}^1\setminus\{x_1, \cdots, x_n\}$, the curve $\mathscr{Y}_s$ is a $G\times \mathbb{Z}/c\mathbb{Z}$-cover of $\mathbb{CP}^1$ branched at $s, x_1, \cdots, x_n$.

In fact one may be completely explicit; Amal Vayalinkal \cite{vayalinkal2024enumeratingfinitebraidgroup} has taught a computer how to perform the middle convolution and used it to explicitly enumerate the possible finite complex reflection groups that may appear in the classifications of \autoref{theorem:sl2-classification} and \autoref{theorem:cp1-geometric-origin}.
\subsubsection{Completing the argument} We now briefly sketch the proof of \autoref{theorem:sl2-classification} and \autoref{theorem:cp1-geometric-origin}.
\begin{proof}[Sketch of proof]
 Suppose $\mathbb{V}$ is an interesting local system on $X=\mathbb{CP}^1\setminus\{x_1, \cdots, x_n\}$; in particular, by Corlette-Simpson \cite[Theorem 9.3 and Proposition 10.5]{corlette-simpson} it is of geometric origin. We wish to find $\mathbb{W}$ with finite monodromy, $\mathbb{L}$ of rank one and finite order, and $\lambda\in \mathbb{C}^\times$ a root of unity, such that $$\mathbb{V}=\on{MC}_\lambda(\mathbb{W})\otimes \mathbb{L}.$$ As $\on{MC}_\lambda$ has (quasi)-inverse $\on{MC}_{\lambda^{-1}}$, it suffices to find $\lambda^{-1}\in \mathbb{C}^\times$ a root of unity and $\mathbb{L}$ of rank one and finite order such that $$\mathbb{W}:=\on{MC}_{\lambda^{-1}}(\mathbb{V}\otimes \mathbb{L}^\vee)$$ has finite monodromy.

Here $\mathbb{V}$ is assumed to be of geometric origin, and the operations $-\otimes\mathbb{L}^\vee$ and $\on{MC}_{\lambda^{-1}}$ preserve this property. Thus whatever $\mathbb{W}$ we produce is of geometric origin itself, and hence has all the attendant structures of local systems of geometric origin. Crucially for us, the local system $\mathbb{W}$ necessarily:
\begin{itemize}
	\item is defined over the ring of integers $\mathscr{O}_K$ of a number field $K$, and 
	\item underlies a polarizable $\mathscr{O}_K$-variation of Hodge structure.
\end{itemize}
By \cite[Lemma 7.2.1]{landesman2022geometric}, a local system $\mathbb{U}$ which is defined over the ring of integers $\mathscr{O}_K$ of a number field $K$, and such that for each embedding $\mathscr{O}_K\hookrightarrow \mathbb{C}$ the associated complex local system $\mathbb{U} \otimes_{\mathscr{O}_K}\mathbb{C}$ is unitary, necessarily has finite monodromy. Thus it suffices to choose $\lambda^{-1}, \mathbb{L}$, so that $\mathbb{W}$ is unitary under all such complex embeddings of $\mathscr{O}_K$.

The main idea here is to use the polarization on the Hodge structure carried by $\mathbb{W}$ --- if the Hodge filtration on $\mathbb{W}$ has length $1$, the polarization is necessarily definite, and hence $\mathbb{W}$ is unitary. That a choice of $\lambda^{-1}, \mathbb{L}$ with this property exist is something of a combinatorial miracle \cite[Theorem 4.3.2]{lam2023finite}, and depends on an analysis of the structure of the Hodge filtration on $\mathbb{V}$ \cite[Proposition 2.3.1]{lam2023finite}, ultimately relying on the fact that the points $x_i$ are generic.

Having chosen $\lambda^{-1}, \mathbb{L}$ appropriately so that $\mathbb{W}$ has finite monodromy, one checks the monodromy is in fact given by a finite complex reflection group just by observing that the local monodromy about each $x_i\neq \infty$ is a pseudoreflection, which is a local calculation.
\end{proof}
\subsection{Some questions}
These results leave a number of questions unresolved; we briefly record two such questions here for the reader who will depart prematurely---there are many, many more such questions later in these notes.
\begin{question}
Can one classify conjugacy classes of finite tuples of matrices $(A_1, \cdots, A_n)\in Y(0,n,2)$ with $\on{Mod}_{0,n}$-orbit, without the condition that some $A_i$ have infinite order?
\end{question}
\begin{question}
	Can one say anything about finite $\on{Mod}_{0,n}$-orbits in $Y(0,n,r)$ with $r>2$? 
\end{question}
We will give a number of conjectural answers and partial results towards this latter question, and its generalizations, in the coming sections. It is worth noting that a few of the ideas here---in particular the intervention of some hidden rigidity, through the use of Corlette-Simpson's work \cite{corlette-simpson}, and the properties of local systems of geometric origin---will appear again and again throughout these notes.
\section{Geometry and dynamics of character varieties in higher genus} \label{section:canonical}
We now consider (Riemann) surfaces not necessarily of genus zero. Our goal in this setting will be to discuss analogues of the questions considered in \autoref{section:matrices}. We will take this opportunity to introduce some of the players that will take center stage in the second half of these notes---in particular, we will explain the relationship between the questions we are studying here and the theory of ordinary differential equations.
\subsection{A more general setup}
Let $\Sigma_{g,n}$ be an orientable (topological) surface of genus $g$, with $n$ punctures, and let $$\on{Mod}_{g,n}=\pi_0(\on{Homeo}^+(\Sigma_{g,n}))$$ be the component group of the group of orientation-preserving self-homeomorphisms of $\Sigma_{g,n}$. Again, $\on{Mod}_{g,n}$ has an an algebro-geometric interpretation if $2-2g-n<0$: the index $n!$ subgroup $\on{PMod}_{g,n}$ preserving the punctures is the (orbifold) fundamental group of the moduli space $\mathscr{M}_{g,n}$ of Riemann surfaces of genus $g$ with $n$ marked points. The group $\on{Mod}_{g,0}=\pi_1(\mathscr{M}_g)$ has a simple group-theoretic interpretation: it is of index $2$ in $\on{Out}(\pi_1(\Sigma_g))$ (in particular, it is the subgroup acting on $H_1(\Sigma_g, \mathbb{Z})$ with determinant $1$).

Generalizing our previous definitions, we set $$Y(g, n, r)=\on{Hom}(\pi_1(\Sigma_{g,n}), \on{GL}_r(\mathbb{C}))/\text{conjugation}.$$
As $\pi_1(\Sigma_{g,n})$ has the standard presentation \begin{equation}\label{equation:surface-pi1} \pi_1(\Sigma_{g,n})=\langle a_1, \cdots, a_g, b_1, \cdots, b_g, c_1, \cdots, c_n \mid \prod_{i=1}^g [a_i, b_i]\prod_{j=1}^n c_j\rangle,\end{equation} we may think of $Y(g,n,r)$ as the set of (simultaneous conjugacy classes of) tuples of $r\times r$ matrices $(A_1, \cdots, A_g, B_1, \cdots, B_g, C_1, \cdots, C_n)$ such that \begin{equation} \label{eqn:surface-identity} \prod_{i=1}^g [A_i, B_i]\prod_{j=1}^n C_j=\text{id}.\end{equation} Again the natural outer action of $\on{Mod}_{g,n}$ on $\pi_1(\Sigma_{g,n})$ induces an action $$\on{Mod}_{g,n}\curvearrowright Y(g,n,r).$$
The most basic question one can ask about this action (whose answer we are quite far from understanding in general, though we will discuss some conjectural answers in the next sections) is:
\begin{question}\label{question:higher-genus-orbits}
	What are the finite orbits of this action?
\end{question}
One can again make the action of $\on{Mod}_{g,n}$ on $Y(g,n,r)$ completely explicit \cite[Tables 3 and 4]{cousin2021algebraic} but the formulas are perhaps less illuminating than those in the genus zero setting.

To our knowledge, prior to the work we explain in this section, there were two basic conjectural (partial) answers to \autoref{question:higher-genus-orbits}.
\begin{conjecture}[Consequence of a conjecture of Esnault-Kerz {\cite[Question 9.1(1)]{esnault2020arithmetic}}, {\cite[Conjecture 1.1]{esnault2023local}} and Budur-Wang {\cite[Conjecture 10.3.1]{budur2020absolute}}]\footnote{We will discuss their actual conjecture and some variants later in these notes, in \autoref{subsection:special-points-and-dense-subloci}.}\label{conjecture:ekbw}
	The finite orbits of the $\on{Mod}_{g,n}$-action on $Y(g,n,r)$ are Zariski-dense in $Y(g,n,r)$.\footnote{To make sense of this one should view $Y(g,n,r)$ as having the topology induced by the Zariski topology on $2g+n$-tuples of matrices satisfying \eqref{eqn:surface-identity}.}
\end{conjecture}
\begin{conjecture}[Kisin {\cite[p.~3]{biswas-whang}, \cite[p.~1]{sinz:thesis}}, Whang {\cite[Question 1.5.3]{lawrence2022representations}}]\label{conjecture:kw}
	For $g\gg r$, the finite $\on{Mod}_{g,n}$-orbits in $Y(g,n,r)$ correspond exactly to those representations $$\pi_1(\Sigma_{g,n})\to \on{GL}_r(\mathbb{C})$$ with finite image.\footnote{Note that representations with finite image necessarily have finite $\on{Mod}_{g,n}$-orbit.}
\end{conjecture}
These two conjectures are in tension with one another; the first should be viewed as saying there are many finite orbits, and the second as saying that there are not too many. And indeed, they contradict one another when $r\geq 2$. To see this, we use the following theorem of Jordan:
\begin{theorem}[Jordan\, {\cite[p.~91]{jordan1878memoire}}]
	There exists a constant $n(r)>0$ such that for any finite subgroup $G$ of $\on{GL}_r(\mathbb{C})$, $G$ contains an abelian subgroup of index at most $n(r)$. 
\end{theorem}
In particular, if $A, B\in \on{GL}_r(\mathbb{C})$ generate a finite subgroup, then $[A^{n(r)!}, B^{n(r)!}]=\on{id}$. So in particular the polynomial $$\on{tr}([\rho(a_1)^{n(r)!}, \rho(b_1)^{n(r)!}])-r$$ (with $a_1, b_1$ as in \autoref{equation:surface-pi1}) vanishes on the locus of $\rho$ with finite image; but it is not hard to see that it is not identically zero on all of $Y(g, n,r)$.

It turns out that \autoref{conjecture:ekbw} is false and \autoref{conjecture:kw} is true.
\begin{theorem}[Landesman-L--\,{\cite[Theorem 1.2.1]{landesman2022canonical}}]\label{theorem:canonical-reps}
	If $g\geq r^2$ and $$\rho: \pi_1(\Sigma_{g,n})\to \on{GL}_r(\mathbb{C})$$ is a representation whose conjugacy class has finite orbit under $\on{Mod}_{g,n}$, then $\rho$ has finite image.
\end{theorem}
\begin{remark}
As we will see, the proof of \autoref{theorem:canonical-reps} relies on the full force of non-abelian Hodge theory and the Langlands program for function fields. The theorem was proven earlier in rank $r=2$ in the beautiful paper \cite{biswas-whang}, by completely different (and much more elementary and explicit) methods. It would be of great interest to find a proof of \autoref{theorem:canonical-reps} of a similarly explicit nature.
\end{remark}
One may again, as in \autoref{proposition:AG-interpretation}, give an algebro-geometric interpretation of the condition that a representation has finite $\on{Mod}_{g,n}$-orbit. 

\begin{proposition}[see {\cite[Corollary 2.3.5]{landesman2022canonical}}]\label{proposition:AG-interpretation-higher-genus}
	An irreducible representation $\rho: \pi_1(\Sigma_{g,n})\to \on{GL}_r(\mathbb{C})$ has conjugacy class with finite $\on{Mod}_{g,n}$-orbit if and only if there exists 
	\begin{itemize}
	\item 	a family of $n$-punctured curves of genus $g$, $\pi: \mathscr{X}\to \mathscr{M}$, with the induced map $\mathscr{M}\to \mathscr{M}_{g,n}$ dominant, and
	\item a local system $\mathbb{V}$ on $\mathscr{X}$ whose restriction to a fiber of $\pi$ has monodromy conjugate to $\rho$.
	\end{itemize}
\end{proposition}

The following is immediate:
\begin{corollary}\label{corollary:isotrivial-local-systems-generic-curve}
Let $C$ be a very general Riemann surface of genus $g\geq r^2$ with $n\geq 0$ very general marked points $\{x_1, \cdots, x_n\}$. Then any local system of geometric origin on $C\setminus\{x_1, \cdots, x_n\}$ of rank $r$ has finite monodromy.
\end{corollary}
Indeed, any local system of geometric origin (defined in \autoref{definition:geometric-origin}) on a very general $n$-times punctured curve would spread out over the total space of a family as in \autoref{proposition:AG-interpretation-higher-genus}. 
\begin{remark}
One may improve the bound in \autoref{corollary:isotrivial-local-systems-generic-curve} to $g\geq r^2/4$; see \cite[Corollary 1.2.7]{landesman2022geometric}. We do not know how to similarly improve  \autoref{theorem:canonical-reps}.
\end{remark}

It is not at all clear to us to what extent the bound in \autoref{theorem:canonical-reps} is sharp, but some bound is indeed necessary---for all $g$ and $n$ with $2-2g-n<0$ there exist many fascinating representations of $\pi_1(\Sigma_{g,n})$ with finite $\on{Mod}_{g,n}$-orbit. It would be very interesting to find some kind of classification of these; we make some conjectures in this direction later in these notes, in \autoref{section:p-curvature} and \autoref{section:mcg}.
\begin{example}[The Kodaira-Parshin trick, see {\cite{kodaira1967certain}, \cite[Proposition 7]{parshin:algebraic-curves-over-function-fields-i}, \cite{atiyah1969signature}, \cite[Proposition 5.1.1]{landesman2022geometric}}]\label{example:kp-trick}
  By \autoref{proposition:AG-interpretation-higher-genus}, one way to construct representations of $\pi_1(\Sigma_{g,n})$ with finite $\on{Mod}_{g,n}$-orbit is to construct local systems on the total space of a family of $n$-punctured curves of genus $g$, $\mathscr{X}\to \mathscr{M}$, so that the induced map $\mathscr{M}\to \mathscr{M}_{g,n}$ is dominant. There is a famous way to do so: the Kodaira-Parshin trick. 
	
 Let $\mathscr{C}_{g,n}$ be the universal curve over $\mathscr{M}_{g,n}$. Loosely speaking, the idea of the Kodaira-Parshin trick is to construct a family of curve $q: \mathscr{Y}\to \mathscr{C}_{g,n}$ whose fiber over a point $[C, x_1, \cdots, x_{n+1}]\in \mathscr{C}_{g,n}$ is a disjoint union of $G$-covers of $C$ branched over $x_1, \cdots, x_{n+1}$, where $G$ is some finite group. The local system $R^1q_*\mathbb{C}$ then provides a (typically very interesting) local system on $\mathscr{C}_{g,n}$, and hence, by \autoref{proposition:AG-interpretation-higher-genus}, a $\pi_1(\Sigma_{g,n})$-representation with finite $\on{Mod}_{g,n}$-orbit.
 
 There are other ways to construct such representations, using e.g.~TQFT methods; see \cite[Example 10.2.4]{landesman2022canonical} for a discussion.
\end{example}

We will shortly sketch the proof of \autoref{theorem:canonical-reps}---but before doing so we will take the opportunity to introduce some of the ideas and themes that will motivate the rest of these notes. The fundamental idea is that $Y(g,n,r)$ has several different avatars, whose structure is the non-abelian analogue of the structures on the cohomology of an algebraic variety (e.g.~Hodge structures on Betti/de Rham cohomology, the Gauss-Manin connection on the cohomology of a family, the Galois action on $\ell$-adic cohomology, and so on). As in \autoref{section:matrices}, we will see that rigid local systems play a special role in the argument; this observation will also motivate some conjectures in \autoref{section:non-abelian-conjectures} and \autoref{section:mcg}.
\subsection{Isomonodromy}
Before diving into the proof of \autoref{theorem:canonical-reps}, we will need yet another interpretation of the $\on{Mod}_{g,n}$-action on $Y(g,n,r)$, closely related to our earlier discussion of the Painlev\'e VI equation \eqref{equation:PVI} and the Schlesinger system \eqref{equation:schlesinger}. The crucial idea is to pass through the Riemann-Hilbert correspondence, between local systems and ordinary differential equations.
\subsubsection{The Riemann-Hilbert correspondence}
\begin{definition}
	Let $X$ be a complex manifold and $D\subset X$ a smooth divisor. A flat bundle $(\mathscr{E}, \nabla)$ on $X$ with regular singularities along $D$ is a holomorphic vector bundle $\mathscr{E}$ on $X$, equipped with a $\mathbb{C}$-linear map $\nabla: \mathscr{E}\to \mathscr{E}\otimes\Omega^1_X(\log D)$ satisfying the Leibniz rule, i.e. $$\nabla(f\cdot s)=f\nabla(s)+s\otimes df$$ for  $U\subset X$ an open subset, $f\in \mathscr{O}_X(U)$, and $s\in \mathscr{E}(U)$, and such that $$\nabla\circ \nabla: \mathscr{E}\to \mathscr{E}\otimes \Omega^2_X(\log D)$$ is identically zero.\footnote{\label{mixed partials} If one views $\nabla$ as a rule for differentiating sections to $\mathscr{E}$ along vector fields, the expression $\nabla\circ\nabla=0$ is a fancy way of saying that {mixed partials commute}.} (Here $\Omega^i_X(\log D)$ is the sheaf of holomorphic $i$-forms on $X$ with logarithmic poles along $D$, see e.g.~\cite[\S4.1]{peters2008mixed}.)
	
	The fiber of the sheaf $\Omega^1_X(\log D)$ at a point $x\in D$ is canonically trivialized by the $1$-form $dz/z$, where $z$ is any local equation for $D$ at $x$. A local computation shows that the composite map $$\mathscr{E}\overset{\nabla}{\longrightarrow} \mathscr{E}\otimes \Omega^1_X(\log D)\to \mathscr{E}\otimes (\Omega^1_X(\log D)/\Omega^1_X)\simeq \mathscr{E}|_D$$ is $\mathscr{O}_X$-linear and hence factors through $\mathscr{E}|_D$. We denote fiber of the induced map $\mathscr{E}|_D\to \mathscr{E}|_D$ at $x\in X$ by $\on{Res}_x(\nabla)$ and refer to it as the \emph{residue} of the connection $\nabla$. 
	
	We denote by $\on{MIC}(X, D)$ the category of flat vector bundles on $X$ with regular singularities along $D$ (where morphisms are given by the evident commutative squares).\footnote{Here MIC stands for \emph{modules with integrable connection}. Many of the notions discussed here can be generalized, e.g. to the case where $D$ has normal crossings, but we choose not to do so for simplicity.} 
\end{definition}
There is an evident functor from $\on{MIC}(X,D)$ to the category of local systems on $X\setminus D$, given by $$(\mathscr{E}, \nabla)\mapsto \ker(\nabla|_{X\setminus D}),$$ often referred to as the \emph{Riemann-Hilbert correspondence}. If $D=\emptyset$, this functor is an equivalence of categories. In general it is faithful and essentially surjective but not full. Indeed, given a local system on $X\setminus D$ there are always many ways to extend it to an object of $\on{MIC}(X,D)$ if $D$ is non-empty.

\begin{example}[Fuchsian ODEs]\label{example:mic-p1}
To bring things down to earth, let's consider the case where $X=\mathbb{CP}^1$ and $D=x_1+\cdots+ x_n$, say with $\infty$ not among the $x_i$. Take $\mathscr{E}=\mathscr{O}_X^r$. Then a connection on $\mathscr{E}$ with regular singularities along $D$ has the form $$\nabla=d-\sum_{i=1}^n \frac{A_i}{z-x_i}dz$$ where the $A_i\in \mathfrak{gl}_r(\mathbb{C})$ are $r\times r$ matrices, and $$\sum_{i=1}^n A_i=0$$ (this last condition ensures there is no singularity at infinity). The condition that $\nabla s=0$ is precisely the linear ordinary differential equation $$\frac{\partial s}{\partial z}=\sum_{i=1}^n \frac{A_i}{z-x_i}s.$$ The matrix $-A_i$ is the residue of $\nabla$ at $x_i$.

To obtain a local system from this ODE, choose some basepoint $x\not\in D$, and consider a basis $s_1, \cdots, s_r$ of local solutions to the ODE at $x$. Given a loop $\gamma: S^1\to X\setminus D$ based at $x$, one may analytically continue these solutions around the loop $\gamma$; the output only depends on the homotopy type of $\gamma$. This yields a new basis of solutions $\gamma^*(s_1, \cdots, s_r)$. But we have just defined an action of homotopy classes of based loops in $X\setminus D$ on bases of local solutions to our ODE. Equivalently, this is a representation of $\pi_1(X\setminus D, x)$ on the space of local solutions to our ODE, or equivalently a local system of rank $r$ whose fiber at $x$ is the space of local solutions to our ODE in a small neighborhood of $x$.
\end{example}
\subsubsection{Isomonodromic deformation}\label{subsubsection:isomonodromic}
We now consider the following question: given an object $(\mathscr{E},\nabla)$ of $\on{MIC}(X,D)$, what happens as we deform the complex structure on $(X,D)$? The topology of $X\setminus D$ does not change, and objects of $\on{MIC}(X,D)$, loosely speaking, correspond to topological objects, namely local systems. So perhaps it is plausible that there is a canonical deformation of $(\mathscr{E},\nabla)$ along any deformation of $(X,D)$. Indeed this is the case, as we now explain.

Let $U$ be a contractible complex manifold and $\pi: \mathscr{X}\to U$ a family of compact Riemann surfaces (i.e.~a proper holomorphic submersion of relative dimension one, with connected fibers). Let $\mathscr{D}\subset \mathscr{X}$ be a divisor in $\mathscr{X}$, finite \'etale over $U$. Let $0\in U$ be a point and set $(X,D)=(\mathscr{X}_0, \mathscr{D}_0)$ to be the fiber over $0$. Let $(\mathscr{E}_0, \nabla_0)$ be a flat bundle on $X$ with regular singularities along $D$.
\begin{definition}\label{definition:isomonodromy}
	An isomonodromic deformation $(\widetilde{\mathscr{E}}, \widetilde{\nabla}, \varphi)$ is a vector bundle $\widetilde{\mathscr{E}}$ on $\mathscr{X}$ equipped with a \emph{relative} connection $$\widetilde{\nabla}:\widetilde{\mathscr{E}}\to \widetilde{\mathscr{E}}\otimes \Omega^1_{\mathscr{X}/U}(\log \mathscr{D}),$$ and equipped with an isomorphism $\varphi: \widetilde{\mathscr{E}}|_X\overset{\sim}{\to} \mathscr{E}_0$ such that 
	\begin{enumerate}
		\item $\widetilde{\nabla}$ admits an extension to a global flat connection $ \widetilde{\mathscr{E}}\to  \widetilde{\mathscr{E}}\otimes \Omega^1_{\mathscr{X}}(\log \mathscr{D})$, and
		\item $\varphi$ induces an isomorphism $( \widetilde{\mathscr{E}}, \widetilde{\nabla})|_X \overset{\sim}{\to} (\mathscr{E}_0, \nabla_0)$.
\end{enumerate}  
\end{definition}
It is not hard to see that isomonodromic deformations exist and are unique up to canonical isomorphism. (See e.g.~\cite[\S0.16]{sabbah2007isomonodromic}.) What do they mean? Condition (2) of \autoref{definition:isomonodromy} simply says that $(\widetilde{\mathscr{E}},\widetilde{\nabla})$ is in fact a deformation of $(\mathscr{E}_0,\nabla_0)$. We claim that condition (1) means that the monodromy of this deformation is constant. Indeed, as $U$ is contractible, the inclusion of $(X, D)$ into $(\mathscr{X}, \mathscr{D})$ induces an isomorphism of fundamental group $\pi_1(X\setminus D)\overset{\sim}{\to}\pi_1(\mathscr{X}\setminus\mathscr{D})$, and hence the restriction map induces a bijection between (isomorphism classes of) local systems on $\mathscr{X}\setminus\mathscr{D}$ and those on $X\setminus D$. Moreover (and we leave this to the reader), the conjugacy classes of the residues of $\widetilde{\nabla}$ are locally constant along $\mathscr{D}$. Thus the isomonodromic deformation $(\widetilde{\mathscr{E}}, \widetilde{\nabla})$ is the deformation of $(\mathscr{E}_0, \nabla_0)$ with constant monodromy and residues (hence the name).

\begin{example}\label{example:schlesinger-derivation}
	We return to the situation of \autoref{example:mic-p1}, and consider the behavior of the connection $$\nabla=d-\sum_{i=1}^n \frac{A_i}{z-x_i}dz$$ under isomonodromic deformation, i.e.~after perturbing the $x_i$. So we view the $A_i$ as functions of the parameters $\underline{x}=(x_1, \cdots, x_n)$, and consider a connection of the form $$\widetilde{\nabla}= d-\sum_{i=1}^n \frac{A_i(\underline{x})}{z-x_i}d(z-x_i)+\sum C_i dx_i,$$ i.e.~a connection with regular singularities along the evident divisors where $z=x_i$. Note that $\sum A_i(\underline{x})=0$ by our assumption that there is no pole at $\infty$.
	
	 The condition that $\widetilde{\nabla}$ be isomonodromic is simply the condition that $\widetilde{\nabla}\circ \widetilde{\nabla}=0$. What condition does this impose on the $A_i$? For $i\neq j$, considering the coefficient of $d(z-x_i)\wedge d(z-x_j)$ gives precisely the first line of \eqref{equation:schlesinger}; differentiating the identity $\sum A_i(\underline{x})=0$ gives the second. In other words, the Schlesinger equations precisely control isomonodromic deformations of ODEs as in \autoref{example:mic-p1} --- so-called Fuchsian ODEs.
	 
	 The Painlev\'e VI equation may be obtained from the Schlesinger equation when the $A_i\in \mathfrak{sl}_2(\mathbb{C})$ and $n=4$ by a change of coordinates; see 
	 \cite[\S3]{malgrange1983deformations} for a complete derivation of the Schlesinger equations, and \cite{jimbo1981monodromy, jimbo1981monodromy2} for details of the connection to Painlev\'e VI.
\end{example}

In general, isomonodromic deformations of flat vector bundles on Riemann surfaces are controlled by certain non-linear algebraic ODE; we have just made this explicit in genus zero, and one can do so in higher genus as well \cite{krichever2002isomonodromy, hurtubise2008geometry}. The question of analyzing finite mapping class group orbits on $Y(g,n,r)$ is precisely the question of understanding algebraic solutions to these non-linear ODE. We will return to this from an arithmetic point of view in the next section.

The first key fact that goes into the proof of \autoref{theorem:canonical-reps} is an analysis of the properties of flat bundles under isomonodromic deformation; we will return to questions along these lines in \autoref{subsubsection:Riemann-Hilbert}.
\begin{theorem}[Landesman-L--\,{\cite[Corollary 6.1.2]{landesman2022geometric}}]\label{theorem:stability}
	Let $X$ be a compact Riemann surface and $D\subset X$ a reduced effective divisor. Let $(\mathscr{E},\nabla)$ be a flat vector bundle on $X$ with regular singularities along $D$ and irreducible monodromy. If $g\geq \on{rk}(\mathscr{E})^2/4$, there exists a small perturbation $(X',D')$ of the complex structure on $(X,D)$ such that the isomonodromic deformation $(\mathscr{E}', \nabla')$ of $(\mathscr{E}, \nabla)$ to $(X', D')$ has $\mathscr{E}'$ (parabolically) semistable.
\end{theorem}
\begin{remark}
Here we say a vector bundle $\mathscr{E}$ is \emph{semistable} if for all nonzero sub-bundles $\mathscr{F}\subset\mathscr{E}$, we have $$\frac{\deg(\mathscr{F})}{\on{rk}(\mathscr{F})}\leq \frac{\deg(\mathscr{E})}{\on{rk}(\mathscr{E})}.$$
The word ``parabolically" appearing in \autoref{theorem:stability} refers to the natural parabolic structure on a bundle with connection---see \cite[\S3]{landesman2022geometric} for details. Vector bundles with parabolic structure (henceforth \emph{parabolic bundles}) admit another notion of degree and hence another notion of stability; we elide this issue for the rest of these notes.
\end{remark}
\begin{remark}
	It is natural to ask if one can achieve stability in \autoref{theorem:stability}, rather than just semistability; this is ongoing work of Andy Ramirez-Cote \cite{andy-rc}.
\end{remark}
\autoref{theorem:stability} is technical indeed but it has a crucial implication:
\begin{corollary}\label{corollary:unitary-c-vhs}
	Let $\pi: \mathscr{X}\to \mathscr{M}$ be a family of $n$-punctured curves of genus $g$ as in \autoref{proposition:AG-interpretation-higher-genus}, with $\mathscr{M}$ connected, and let $X$ be any fiber of $\pi$. Any local system of rank $r$ on $\mathscr{X}$, with $g\geq r^2/4$, which underlies a polarizable complex variation of Hodge structure, restricts to a unitary local system on $X$.
\end{corollary}
\begin{proof}[Proof sketch]
	Let $\mathbb{V}$ be a local system on $\mathscr{X}$ as in the statement, with $(\mathscr{E}, \nabla, F^\bullet)$ the associated flat bundle with (decreasing) Hodge filtration $F^\bullet$. Then $\mathscr{E}|_X$ has (parabolic) degree zero. For simplicity we assume that the monodromy of $\mathbb{V}|_X$ is irreducible; in this case we show that $\mathbb{V}$ itself is unitary. Let $i$ be maximal such that $F^i$ is non-zero; suppose that $F^i\neq \mathscr{E}$. Then by \cite[Corollary 4.1.8]{landesman2022geometric}, $F^i|_{X}$ has positive (parabolic) degree, so (as $X$ was an arbitrary fiber of $\pi$) $\mathscr{E}|_{X'}$ is not semistable for any fiber $X'$ of $\pi$.  But this contradicts \autoref{theorem:stability} --- there is no way to perturb $X$ to make $\mathscr{E}|_X$ semistable.
	
	So the Hodge filtration has length one, whence the polarization on $\mathbb{V}$ is a definite Hermitian form. But this form is preserved by the monodromy, which is hence unitary.
\end{proof}
 We will return to questions about the behavior of flat bundles under isomonodromic deformation later in these notes (in \autoref{subsubsection:Riemann-Hilbert}), but before doing so we will introduce the other main ingredient of the proof of \autoref{theorem:canonical-reps}. 
 \subsection{Big monodromy, rigidity, and vanishing theorems}
 Thus far we have largely studied the action of $\on{Mod}_{g,n}$ on $Y(g,n,r)$, the space of rank $r$ representations of $\pi_1(\Sigma_{g,n})$. In \autoref{subsection:an-algebro-geometric-interpretation}, I suggested that this study is analogous to classical questions in algebraic geometry (the Hodge conjecture, the Tate conjecture, and so forth); we will discuss this further in \autoref{section:non-abelian-conjectures}. Part of this analogy is the idea that $Y(g,n,r)$ ought to be viewed as a non-abelian analogue of first cohomology. Indeed, $$Y(g,n,1)=H^1(\Sigma_{g,n}, \mathbb{C}^\times).$$ So perhaps it is unsurprising that classical questions about monodromy actions on cohomology will intervene in our study of the mapping class group action on $Y(g,n,r)$.
 
 Let $$\rho: \pi_1(\Sigma_{g,n})\to \on{GL}_r(\mathbb{C})$$ be an irreducible representation whose conjugacy class has finite orbit under $\on{Mod}_{g,n}$, say with stabilizer $\Gamma\subset \on{Mod}_{g,n}$ a subgroup of finite index. There is a natural $\Gamma$-representation associated to $\rho$, namely the action of $\Gamma$ on the tangent space $$T_{[\rho]}Y(g,n,r)=H^1(\pi_1(\Sigma_{g,n}), \text{ad}(\rho)).$$  (See e.g.~\cite{sikora2012character} for a discussion of this cohomological interpretation of tangent spaces.)
 
 The main result on monodromy actions on cohomology that we will use here, which follows from an analysis of derivatives of period maps---discussed further in \autoref{subsection:big-monodromy-RH}---is:
 \begin{theorem}[Landesman-L-- {\cite[Consequence of Theorem 1.7.1]{landesman2022canonical}}]\label{theorem:vanishing}
 	Let $\pi: \mathscr{X}\to \mathscr{M}$ be a family of $n$-punctured curves of genus $g$, so that the induced map $\mathscr{M}\to\mathscr{M}_{g,n}$ is dominant. Let $\mathbb{V}$ be a unitary local system on $\mathscr{X}$ of rank less than $g$. Then $$H^0(\mathscr{M}, R^1\pi_*\mathbb{V})=0.$$
 \end{theorem}
Put another way, let $X$ be a fiber of $\pi$; then $\pi_1(\mathscr{M})$ admits a natural action on $H^1(X, \mathbb{V}|_X)$. The theorem says that this action has no nonzero invariants (and moreover, since the statement remains the same on passing to covers of $\mathscr{M}$, no nonzero finite orbits). As we will see later in \autoref{section:mcg}, this sort of statement is closely connected to major open questions in surface topology.

Let us explain the relevance to us. Consider the case where $\mathbb{V}=\on{ad}(\mathbb{W})$, with $\mathbb{W}$ unitary and irreducible and $\on{rk}(\on{ad}(\mathbb{W}))< g$. Letting $\rho$ be the monodromy of $\mathbb{W}|_X$, we see that the conjugacy class of $\rho$ has finite orbit under $\on{Mod}_{g,n}$, say with stabilizer $\Gamma\subset \on{Mod}_{g,n}$. And the action of $\pi_1(\mathscr{M})$ on $H^1(X, \on{ad}(\mathbb{W})|_X)$ factors through the $\Gamma$-action on $T_{[\rho]}Y(g,n,r)=H^1(\pi_1(\Sigma_{g,n}), \text{ad}(\rho)),$ so we may study it by analyzing $H^0(\mathscr{M}, R^1\pi_*\on{ad}\mathbb{W})$, which vanishes, by \autoref{theorem:vanishing}.

That this space is zero says that $\rho$ is isolated as a $\Gamma$-fixed point in $Y(g,n,r)$. Indeed, if $\Gamma$ fixed a positive-dimensional subvariety $Z$ of $Y(g,n,r)$ passing through $[\rho]$, then $T_{[\rho]}Z\subset T_{[\rho]}Y(g,n,r)$ would be $\Gamma$-fixed as well. (Compare to \autoref{definition:interesting-rep}(3).)

One can ask finer questions about the monodromy of local systems of the form $R^1\pi_*\mathbb{V}$ as in \autoref{theorem:vanishing} (for example, what is their precise monodromy group?), and we will do so later, in \autoref{subsection:big-monodromy-RH}.
\subsection{Idea of the proof}
We now turn to the idea of the proof of \autoref{theorem:canonical-reps}. We will make several simplifying assumptions over the course of the sketch, which hopefully the reader will forgive.
\begin{proof}[Sketch of proof of \autoref{theorem:canonical-reps}]
	Let $$\rho: \pi_1(\Sigma_{g,n})\to \on{GL}_r(\mathbb{C})$$ be a representation whose conjugacy class has finite $\on{Mod}_{g,n}$-orbit. We would like to show that if $g\geq r^2$, then $\rho$ has finite image. We assume for simplicity that $\rho$ is irreducible. Then by \autoref{proposition:AG-interpretation-higher-genus}, there exists a family of $n$-punctured curves of genus $g$, $$\pi: \mathscr{X}\to \mathscr{M},$$ with the associated map $\mathscr{M}\to \mathscr{M}_{g,n}$ dominant, and a local system $\mathbb{V}$ on $\mathscr{X}$, such that the monodromy of the restriction of $\mathbb{V}$ to a fiber of $\pi$ is given by $\rho$.
	
	By work of Mochizuki \cite[Th.~10.5]{mochizuki2006kobayashi}, $\mathbb{V}$ can be deformed to a polarizable complex variation of Hodge structure $\mathbb{V}'$ (see \autoref{theorem:deformation-to-vhs} for a variant of this result and a sketch of the proof). We assume for simplicity that $\mathbb{V}'$ is irreducible when restricted to a fiber of $\pi$.\footnote{This assumption is extremely strong, and in fact to give a correct proof one must circumvent it. Doing so is the source of many of the technical considerations in \cite{landesman2022canonical}.}  By \autoref{corollary:unitary-c-vhs}, $\mathbb{V}'$ is in fact unitary. Now by \autoref{theorem:vanishing}, applied to $\on{ad}(\mathbb{V}')$, $\mathbb{V}'$ is \emph{cohomologically rigid}, i.e.~it admits no non-trivial infinitesimal deformations.
	
	This observation has two consequences:
	\begin{enumerate}
		\item as we have deformed $\mathbb{V}$ to a representation with \emph{no non-trivial deformations}, we must have that $\mathbb{V}$ and $\mathbb{V}'$ are isomorphic to one another, and
		\item $\mathbb{V}'$ (and hence $\mathbb{V}$) are defined over $\mathscr{O}_K$, the ring of integers of some number field $K$, by work of Esnault-Groechenig \cite[Theorem 1.1]{esnault2018cohomologically} (which ultimately relies on the Langlands program over function fields).\footnote{We discuss this work further in \autoref{subsection:rigid-local-systems}.}
	\end{enumerate} 
	We may now apply a similar argument to that used in the proof of \autoref{theorem:sl2-classification} and \autoref{theorem:cp1-geometric-origin}. Again by \cite[Lemma 7.2.1]{landesman2022geometric}, it suffices to show that for each embedding $$\iota: \mathscr{O}_K\hookrightarrow \mathbb{C},$$ the local system $\mathbb{V}\otimes_{\iota, \mathscr{O}_K} \mathbb{C}$ is unitary. But we know $\mathbb{V}$ is rigid (as this is an algebraic property, independent of a choice of complex embedding), hence by Mochizuki \cite[Th.~10.5]{mochizuki2006kobayashi}, $\mathbb{V}\otimes_{\iota, \mathscr{O}_K} \mathbb{C}$ underlies a complex variation of Hodge structure. Now by \autoref{corollary:unitary-c-vhs}, $\mathbb{V}\otimes_{\iota, \mathscr{O}_K} \mathbb{C}$ is unitary, completing the proof.
\end{proof}
Our simplifying assumptions (that the various local systems appearing in the argument are irreducible) allow us to avoid substantial complications in the argument, but we elide them here. See \cite{landesman2022canonical} for details.

We have in this section classified all finite $\on{Mod}_{g,n}$-orbits on $Y(g,n,r)$, when $g\geq r^2$, and in \autoref{section:matrices} we classified finite $\on{PMod}_{0,n}$-orbits on $Y(\underline{C})$, where $\underline{C}=(C_1, \cdots, C_n)$ contains some conjugacy class $C_i$ of infinite order. Much remains to do---we understand little of the dynamics of the $\on{Mod}_{g,n}$-action on $Y(g,n,r)$ when $r\gg g$. In the next section, we give a conjectural arithmetic characterization of finite orbits of this action and its generalizations, and prove some important special cases of this conjecture; in the following two sections we discuss other conjectural characterizations from the points of view of algebraic geometry and more classical surface topology.
\section{Algebraic differential equations and dynamics on character varieties}\label{section:p-curvature}
We turn now to a much more general setting. Let $S$ be smooth and let $$f: \mathscr{X}\to S$$ be a smooth proper morphism with connected fibers\footnote{Essentially all of the results in this section apply when $\mathscr{X}$ is equipped with a simple normal crossings divisor over $S$, but we omit this here to avoid notational complication, and some mild complications around the Riemann-Hilbert correspondence in the presence of a simple normal crossings divisor. The reader will lose almost nothing by assuming $f$ has relative dimension one, i.e. that it is a family of curves.} over the complex numbers. Let $X_o$ be a general fiber of $f$, say over some point $o\in S(\mathbb{C})$. We let $\mathscr{M}_B(X_o, r)$ be the (stack) quotient $$\on{Hom}(\pi_1(X_o), \on{GL}_r(\mathbb{C}))/\text{conjugation}$$ and $M_B(X_o, r)$ the quotient in the sense of GIT (the latter being an affine complex variety). As before we are interested in the natural action of $\pi_1(S,o)$ on $M_B(X_o, r)$, and in its action on the set of isomorphism classes of objects of $\mathscr{M}_B(X_o, r)$ (i.e. conjugacy classes representations of $\pi_1(X_o)$). The goal of this section is to give a conjectural classification of \emph{all} finite orbits of this action, motivated by the Grothendieck-Katz $p$-curvature conjecture, in terms of arithmetic invariants of the corresponding flat vector bundles, and to sketch a proof of this conjecture in important cases of interest.

As we have previously discussed in special cases (e.g.~in \autoref{subsubsection:painleve} and \autoref{example:schlesinger-derivation}), these finite orbits are related to algebraic solutions to certain non-linear differential equations. We now explain this in general.

\subsection{Dynamics and foliations}
Let $\widetilde{S}$ be the universal cover of $S$ (which is typically only a complex manifold, and does not have the structure of an algebraic variety). We may consider the projection $$\pi: M_B(\mathscr{X}/S,r):=(\widetilde{S}\times M_B(X_o, r))/\pi_1(S, o)\to S,$$ where here $\pi_1(S, o)$ acts on $\widetilde{S}$ freely by deck transformations and on $M_B(X_o, r)$ through its outer action on $\pi_1(X_o)$. This quotient naturally has the structure of a \emph{local system of schemes} in the sense of \cite[p.~12]{simpson1995moduli}, which loosely speaking means that one has a notion of \emph{flat sections} to $\pi$. Explicitly, a section to $\pi$ is flat if and only if, locally on $S$, it lifts to a constant section to the projection $\widetilde{S}\times M_B(X_o,r)\to \widetilde{S}$, i.e.~it corresponds to a family of representations whose conjugacy class is constant. One may similarly construct a moduli stack $\mathscr{M}_B(\mathscr{X}/S, r)$, see \cite{simpson1995moduli} for details.

On the smooth locus\footnote{For example, if $f$ has relative dimension one, this locus contains the open subset of $M_B(\mathscr{X}/S)$ corresponding to irreducible representations of $\pi_1(X_o)$.} $\pi^{\text{sm}}$ of $\pi$ this structure has a simple interpretation; it is nothing more than a horizontal foliation, i.e.~a splitting of the short exact sequence $$0\to T_{\pi^{\text{sm}}}\to T_{M_B(\mathscr{X}/S)}|_{\pi^{sm}}\to \pi^*T_S|_{\pi^{sm}}\to 0,$$ compatible with the Lie algebra structure on these sheaves.  A representation $\rho$ of $\pi_1(X_o)$ has finite $\pi_1(S, o)$-orbit if and only if the leaf of this foliation through $[\rho]$ is finite over $S$, by definition.

The above construction is highly transcendental in nature, relying as it does on the universal cover of $S$ and on the topological fundamental groups $\pi_1(X_o), \pi_1(S, o)$. Nonetheless the construction has an algebraic avatar.

There is a moduli stack $\mathscr{M}_{dR}(\mathscr{X}/S, r)$ over $S$, which loosely speaking represents the functor that sends an $S$-scheme $T$ to the groupoid of rank $r$ flat bundles on $\mathscr{X}_T/T$. The Riemann-Hilbert correspondence gives an analytic isomorphism between $\mathscr{M}_{dR}(\mathscr{X}/S, r)^{\text{an}}$ and $\mathscr{M}_B(\mathscr{X}/S, r)$ \cite[Proposition 7.8]{simpson1995moduli}. It turns out that the structure of a local system of schemes on $M_B(\mathscr{X}/S, r)$ may be interpreted algebraically on $\mathscr{M}_{dR}(\mathscr{X}/S, r)$; the latter stack is a \emph{crystal} over $S$. This is the non-abelian analogue of the algebraic nature of the Gauss-Manin connection on the de Rham cohomology of a family of varieties \cite{katz1968differentiation}. For details see \cite[Theorem 8.6]{simpson1995moduli}; we now give a brief explanation of how this works at the level of horizontal foliations. The idea is that flat sections to this morphism correspond to isomonodromic deformations in the sense of \autoref{subsubsection:isomonodromic}; we make the corresponding foliation explicit at smooth points corresponding to irreducible representations. In order to do so we need to introduce some notation to give a cohomological description of the tangent spaces to $\mathscr{M}_{dR}(\mathscr{X}/S, r)$.
\subsubsection{The Atiyah-de Rham complex}
The deformation theory of a pair $(X, \mathscr{E})$, with $X$ a smooth variety and $\mathscr{E}$ a vector bundle on $X$, is controlled by a vector bundle called the Atiyah bundle.
\begin{definition}
	The sheaf of first-order differential operators on $\mathscr{E}$, $\on{Diff}^1(\mathscr{E}, \mathscr{E})$, is the sheaf of $\mathbb{C}$-linear maps $\delta: \mathscr{E}\to \mathscr{E}$ such that the map $$s\mapsto \delta_f(s):=\delta(f  s)-f\delta(s)$$ is $\mathscr{O}_X$-linear for all local sections $s$ of $\mathscr{E}$ and $f$ of $\mathscr{O}_X$. The Atiyah bundle $\on{At}(\mathscr{E})\subset \on{Diff}^1(\mathscr{E}, \mathscr{E})$ is the sheaf of first-order differential operators $\delta$ such that $\delta_f$ is given by multiplication by a section of $\mathscr{O}_X$, for all $f$.
\end{definition}
Direct computation shows that for a local section $\delta$ to $\on{At}(\mathscr{E})$, the assignment $\tau_\delta: f\mapsto \delta_f$ is in fact an $\mathscr{O}_X$-valued derivation; let $$\tau: \on{At}(\mathscr{E})\to T_X$$ $$\delta\mapsto \tau_\delta$$ be the corresponding map.

By construction there is a short exact sequence $$0\to \on{End}(\mathscr{E})\to \on{At}(\mathscr{E})\overset{\tau}{\to} T_X\to 0,$$ called the \emph{Atiyah exact sequence}. The data of a flat connection $\nabla$ on $\mathscr{E}$ is the same as the data of an $\mathscr{O}$-linear splitting $q^\nabla$ of this sequence, where $q^\nabla(v)(s)$ is given by contracting $v$ with $\nabla(s)$.

Now suppose we are given a flat connection $\nabla$ on $\mathscr{E}$, and consider the complex $$\on{At}(\mathscr{E})^\bullet_{dR}: \on{At}(\mathscr{E})\to \on{End}(\mathscr{E})\otimes \Omega^1_X\overset{\nabla}{\to} \on{End}(\mathscr{E})\otimes \Omega^2_X\overset{\nabla}{\to} \cdots$$
where the first differential is given by taking the commutator of a differential operator with $\nabla$ (see \cite[Lemma-Definition 4.3]{katzarkov1999non} for a precise formula) and the rest are given by the connection on $\on{End}(\mathscr{E})$ induced by $\nabla$. We refer to this complex as the \emph{Atiyah-de Rham complex} of $(\mathscr{E}, \nabla)$. There is a short exact sequence of complexes \begin{equation}\label{equation:atiyah-de-rham-sequence} 0\to \on{End}(\mathscr{E})_{dR}^\bullet\to \on{At}(\mathscr{E})^\bullet_{dR}\to T_X\to 0,\end{equation} where $\on{End}(\mathscr{E})_{dR}^\bullet$ is the de Rham complex of $\on{End}(\mathscr{E})$ with its induced connection.
\begin{proposition}[see e.g.~{\cite[Proposition 4.4]{chen2012associated}} for the case of curves]\label{proposition:deformation}
	There is a natural identification between $$\mathbb{H}^1(\on{At}_{dR}^\bullet(\mathscr{E}))$$ and the space of first-order deformations of $(X, \mathscr{E}, \nabla)$. Under this identification, the map $$\mathbb{H}^1(\on{At}_{dR}^\bullet(\mathscr{E}))\to H^1(T_X)$$ induced by \eqref{equation:atiyah-de-rham-sequence} sends a deformation of $(X, \mathscr{E},\nabla)$ to the corresponding deformation of $X$ (under the natural Kodaira-Spencer identification of first-order deformations of $X$ with $H^1(T_X)$).
\end{proposition}
Note that the natural map $$\on{At}(\mathscr{E})^\bullet_{dR}\to T_X$$ has a splitting, given by the splitting $q^\nabla$ of the Atiyah exact sequence discussed above; that this is a map of complexes follows from the flatness of $\nabla$. The induced splitting of the map $$\mathbb{H}^1(\on{At}_{dR}^\bullet(\mathscr{E}))\to H^1(T_X)$$ is the source of the foliation on $\mathscr{M}_{dR}(\mathscr{X}/S, r)/S$, as we now explain. Let $s\in S$ be a point and set $X=\mathscr{X}_s$. Then the Kodaira-Spencer map yields a map $T_sS\to H^1(X, T_X)$. Let $(\mathscr{E},\nabla)\in \on{MIC}(X)$ be an irreducible flat bundle corresponding to a smooth point of $\mathscr{M}_{dR}(\mathscr{X}/S, r)$. Then the tangent space to $\mathscr{M}_{dR}(\mathscr{X}/S, r)$ fits into a pullback square
$$\xymatrix{
T_{[(X, \mathscr{E}, \nabla)]}\mathscr{M}_{dR}(\mathscr{X}/S, r)\ar[r] \ar[d] & \mathbb{H}^1(\on{At}_{dR}^\bullet(\mathscr{E})) \ar[d] \\
T_sS \ar[r] & H^1(T_X)
}$$
where the left vertical map is the differential of the natural map $\mathscr{M}_{dR}(\mathscr{X}/S, r)\to S$.

Thus the natural splitting $q^\nabla$ to the right vertical map induces a section to the left vertical map, i.e.~ a horizontal foliation on $\mathscr{M}_{dR}(\mathscr{X}/S, r)$ over $S$. This is the \emph{isomonodromy foliation} --- its leaves correspond precisely to isomonodromic deformations as before (see \cite[Proposition 5.1]{chen2012associated} for the case were $\mathscr{X}/S$ is the universal curve over $\mathscr{M}_g$; the general case follows identically). By comparison to the analytically isomorphic space $\mathscr{M}_B(\mathscr{X}/S,r)$, the leaves of this foliation that are finite over $S$ correspond precisely to the representations with finite $\pi_1(S,s)$-orbit, under the Riemann-Hilbert correspondence. These are, by e.g. \cite[Theorem A]{cousin2021algebraic} or \cite[Remark 2.3.6]{landesman2022canonical}, precisely the \emph{algebraic leaves} of this foliation.
\subsection{Algebraic solutions to differential equations}
We are thus lead to consider the following question:
\begin{question}
How can one characterize the algebraic solutions to an algebraic differential equation?	
\end{question}
In the case of linear ODEs, this question has a classical (conjectural) answer, due to Grothendieck and Katz: the Grothendieck-Katz $p$-curvature conjecture. We recall their conjecture and then discuss its analogue for the isomonodromy foliation discussed above.
\subsubsection{The $p$-curvature conjecture}
Let $A\in \on{Mat}_{r\times r}(\overline{\mathbb{Q}}(z))$ be a matrix of rational functions with algebraic coefficients, and consider the linear ODE $$\left(\frac{d}{dz}-A\right)\vec f(z)=0.$$
We are interested in understanding when this ODE admits a basis of algebraic solutions, i.e.~when, if one takes the formal power series expansion of the solutions to this ODE at a point where $A$ has no poles, the resulting power series are all algebraic over $\overline{\mathbb{Q}}(z)$. Equivalently---when $A$ has only simple poles along a divisor $D\subset\mathbb{CP}^1$---we are, under the Riemann-Hilbert correspondence, interested in understanding when the corresponding representation of $\pi_1(\mathbb{CP}^1\setminus D)$ has finite image.
\begin{conjecture}[Grothendieck-Katz {\cite{katz-p-curvature}}]\label{conjecture:p-curvature-1}
	The above linear ODE has a basis of algebraic solutions if and only if $$\left(\frac{d}{dz}-A\right)^p\equiv 0\bmod p$$ for almost all primes $p$.
\end{conjecture}
Here the condition of the conjecture makes sense because the entries of $A$ have only finitely many denominators, so the given expression can be reduced mod $p$ for almost all primes $p$.
\begin{example}
	Let us consider the differential equation $$\left(\frac{d}{dz}-\frac{a}{z}\right)f(z)=0,$$ for $a\in \overline{\mathbb{Q}}$. The local solutions to this ODE have the form $cz^a$; this is an algebraic function if and only if $a\in \mathbb{Q}$. On the other hand, $$\left(\frac{d}{dz}-\frac{a}{z}\right)^pz^n=(n-a)(n-a-1)\cdots (n-a-p+1)z^{n-p}$$ is identically zero mod $p$ for almost all $p$ iff $p$ splits completely in $\mathbb{Q}(a)$ for almost all $p$; this happens if and only if $a\in \mathbb{Q}$, by the Chebotarev density theorem.
\end{example}
There is a more intrinsic formulation of \autoref{conjecture:p-curvature-1}, which makes sense on general smooth bases. To formulate it, we recall the notion of $p$-curvature.
\begin{definition}
	Let $k$ be a field of characteristic $p>0$ and $X/k$ a smooth $k$-scheme. Let $(\mathscr{E}, \nabla)$ be a flat bundle on $X/k$, where we view $\nabla$ as a $k$-linear map $T_X\to \on{End}_k(\mathscr{E})$. The \emph{$p$-curvature morphism} $$\psi_p: F_{\text{abs}}^*T_X\to \on{End}_{\mathscr{O}_X}(\mathscr{E})$$ is the map induced by $$v\mapsto \nabla(v)^p- \nabla(v^p).$$
\end{definition} 
Here $v^p$ is a section to $T_X$ --- in characteristic $p>0$, the $p$-th power of a derivation is itself a derivation. If $X$ is an open subset of $\mathbb{P}^1$ and $v=\frac{d}{dz}$, then $v^p=0$, and so the vanishing of $p$-curvature is simply the condition $\nabla(\frac{d}{dz})^p=0$, which we saw before in \autoref{conjecture:p-curvature-1}. One may think of $\psi_p$ as a measure of the failure of the map $\nabla: T_X\to \on{End}_k(\mathscr{E})$ to commute with taking $p$-th powers (just as the curvature $\nabla\circ\nabla$ is a measure of the failure of a connection to commute with the natural Lie algebra structures on $T_X, \on{End}_k(\mathscr{E})$; see \autoref{mixed partials}).

Now let $R$ be a finitely-generated integral $\mathbb{Z}$-algebra with fraction field $K$, $X/R$ a smooth scheme, and $(\mathscr{E}, \nabla)$ a flat vector bundle on $X/R$.

\begin{conjecture}[Grothendieck-Katz]\label{conjecture:gk-p-curvature}
	The flat vector bundle $(\mathscr{E}, \nabla)_K\in \on{MIC}(X_K)$ admits a basis of flat algebraic sections\footnote{It follows that for some (equivalently all) embedding(s) $K\hookrightarrow \mathbb{C},$ the corresponding flat bundle $(\mathscr{E}, \nabla)_\mathbb{C}$ has finite monodromy. In fact the arithmetic condition of \autoref{conjecture:gk-p-curvature} (on vanishing of $p$-curvature) implies (non-trivially!) that given a simple normal crossings compacitification $\overline{X}$ of $X$, $(\mathscr{E},\nabla)$ extends to a flat connection on $\overline{X}$ with regular singularities along the boundary. It follows that---to prove the $p$-curvature conjecture---it suffices to show that the arithmetic hypotheses of the conjecture imply that $(\mathscr{E}, \nabla)_\mathbb{C}$ has finite monodromy.} if and only if there exists a dense open subset $U\subset \on{Spec}(R)$ such that for all closed points $\mathfrak{p}\in U$, the $p$-curvature of $(\mathscr{E},\nabla)_\mathfrak{p}$ is identically zero.
\end{conjecture}
This conjecture is more or less completely open; the primary evidence we have for it, due to Katz, is that it is true when $(\mathscr{E},\nabla)$ is a \emph{Picard-Fuchs equation}. 
\begin{definition}\label{definition:picard-fuchs-equation}
	Let $X$ be a smooth variety. A flat vector bundle $(\mathscr{E}, \nabla)$ is a \emph{Picard-Fuchs equation} if there exists a dense open subset $U\subset X$, a smooth proper morphism $\pi: Y\to U$, and an integer $i\geq 0$ such that $$(\mathscr{E}, \nabla)|_U\simeq (R^i\pi_*(\Omega^\bullet_{dR, Y/U}), \nabla_{GM}),$$ where $\nabla_{GM}$ is the Gauss-Manin connection.
\end{definition}
\begin{theorem}[{\cite[Theorem 5.1]{katz-p-curvature}}]\label{theorem:katz-p-curvature}
	With notation as in \autoref{conjecture:gk-p-curvature}, suppose $(\mathscr{E}, \nabla)_K$ is a Picard-Fuchs equation. Then the $p$-curvature conjecture is true for $(\mathscr{E},\nabla)$.
\end{theorem}
\begin{remark}
Katz's theorem is in fact a bit more general than the result stated above, but it has some fairly strong restrictions. For example, we do not even know that the $p$-curvature conjecture holds true for direct summands of Picard-Fuchs equations, except under restrictive hypotheses.	 But see \cite[\S16]{andre2004conjecture} for some results in this direction.
\end{remark}
Katz's proof of \autoref{theorem:katz-p-curvature} inspires many of the arguments in these notes, including those in \autoref{section:canonical}; we briefly recall it as we will shortly see a ``non-abelian" analogue of his approach. See \cite[\S2.3]{esnault-local-systems} for another brief exposition.
\begin{proof}[Proof sketch]
	The flat bundle $(\mathscr{E}, \nabla)$, arising as it does as a Picard-Fuchs equation, carries a decreasing Griffiths-transverse Hodge filtration $F^\bullet$; here Griffiths transversality means that $$\nabla(F^i)\subset F^{i-1}\otimes \Omega^1_X$$ for all $i$. The failure of $\nabla$ to preserve this filtration is thus measured by a collection of ($\mathscr{O}$-linear) maps $$\on{gr}^i\nabla: \on{gr}^i_{F^\bullet} \mathscr{E}\to \on{gr}^{i-1}_{F^\bullet} \mathscr{E}\otimes\Omega^1_X.$$ It suffices to show that these maps are zero; indeed, in this case the monodromy of $(\mathscr{E}, \nabla)$ preserves the Hodge filtration $F^\bullet$ and hence the Hodge decomposition. But then the monodromy is unitary, as the monodromy also preserves the polarization on the local system $\ker \nabla$, which is definite on each graded piece of the Hodge decomposition. As $(\mathscr{E}, \nabla)$ is a Picard-Fuchs equation, its monodromy is defined over $\mathbb{Z}$; combined with unitarity, this implies the monodromy is finite.
	
	To show that the maps $\on{gr}^i\nabla$ are identically zero, Katz compares their reductions mod $p$ to the associated graded of the $p$-curvature maps $\psi_p$ with respect to the conjugate filtration, which are zero by assumption. This comparison is a lengthy computation about which we will say nothing.
\end{proof}

There are a few other cases in which the $p$-curvature conjecture is known. The case where the corresponding monodromy representation is solvable was resolved by Chudnovsky-Chudnovsky \cite{chudnovsky2006applications}, Bost \cite{bost2001algebraic}, and Andr\'e \cite{andre2004conjecture}, and the case where it underlies a rigid $\mathbb{Z}$-local system, by Esnault-Groechenig \cite{esnault2020rigid}. There are a number of other interesting related works and special cases known, e.g.~\cite{farb2009rigidity, esnault2018d, shankar2018p, patel2021rank, tang2018algebraic}; in fact the two papers by Shankar and Patel-Shankar-Whang cited here are what originally interested the author in this subject.
\subsection{Finite orbits}
We now formulate a variant of the $p$-curvature conjecture for the isomonodromy foliation on $\mathscr{M}_{dR}(\mathscr{X}/S)\to S$; the upshot of this will be a (conjectural) complete classification of finite $\pi_1(S,s)$ orbits on conjugacy classes of representations of $\pi_1(X_s)$. 

\begin{conjecture}[Imprecise conjecture]\label{conjecture:imprecise}
	Let $R$ be a finitely-generated integral $\mathbb{Z}$-algebra with fraction field $K$ and $\mathscr{X}\to S$ a smooth projective morphism of $R$-schemes. Let $s\in S(R)$ be an $R$-point, $X=\mathscr{X}_s$, and $(\mathscr{E}, \nabla)$ a flat bundle on $X/R$. Then the leaf of the isomonodromy foliation through $[(X, \mathscr{E},\nabla)]\in \mathscr{M}_{dR}(\mathscr{X}/S)$ is algebraic if and only if it is integral. This occurs if and only if it is $p$-integral to order $\omega(p)$, for almost all primes $p$.
\end{conjecture}
Here $\omega: \text{Primes}\to \mathbb{Z}$ is any function growing faster than any $\epsilon p$ for all $\epsilon>0$, i.e.~$$\lim_{p\to\infty} \frac{\omega(p)}{p}=\infty.$$

\begin{remark}
The last sentence (about $p$-integrality to order $p$) is a bit technical, and we largely ignore it for the rest of this note. However, see \autoref{prop:lawrence-litt}, which is our primary motivation for including this condition; the point is to give a variant of the $p$-curvature conjecture for $\mathscr{M}_{dR}(\mathscr{X}/S)$ which implies the classical $p$-curvature conjecture, \autoref{conjecture:p-curvature-1}.
\end{remark}

There are a number of imprecisions in this statement---notably $\mathscr{M}_{dR}(\mathscr{X}/S,r)$ is not in general smooth, or even a scheme. One can make it precise, again using the language of crystals. We have also not explained what it means for the leaf to be integral, (resp.~$p$-integral to order $n$). Loosely speaking, this latter simply means that the Taylor coefficients of the power series defining the formal leaf of the foliation are integral (resp.~the coefficients of monomials of degree at most $n$ are $p$-integral). 

Let us make friends with this conjecture by first formulating a variant in the spirit of \autoref{conjecture:p-curvature-1}, for general (possibly non-linear) ODE, and then specialize to the case of the isomonodromy foliation.
\begin{conjecture}
	Let $$f^{(n)}(z)=F(z, f(z), \cdots, f^{(n-1)}(z))$$ be an ordinary differential equation, with $F\in \mathbb{Q}(z_0, \cdots, z_n).$ Then the Taylor series expansion of a solution $$f(z)\in \mathbb{Q}[[t]], f(z)=\sum_{n\geq 0} a_nz^n$$ to this ODE with $(0, f(0), \cdots f^{(n-1)}(z)(0))$ a non-singular point of $F$ is algebraic if and only if there exists $N>0$ such that $a_n\in \mathbb{Z}[1/N]$ for all $n$. This occurs if and only if $a_1, \cdots, a_{\omega(p)}$ are $p$-integral for almost all $p$, where $\omega$ is some function such that $$\lim_{p\to\infty} \frac{\omega(p)}{p}=\infty.$$
\end{conjecture}

Let us make \autoref{conjecture:imprecise} precise now:
\begin{conjecture}[Precisification]\label{conjecture:precisification}
	Let $R$ be a finitely-generated integral $\mathbb{Z}$-algebra with fraction field $K\subset \mathbb{C}$ and $\mathscr{X}\to S$ a smooth projective morphism of $R$-schemes. Let $s\in S(R)$ be an $R$-point, $X=\mathscr{X}_s$, and $(\mathscr{E}, \nabla)$ a flat bundle on $X/R$. Let $s_K$ be the $K$-point of $S$ associated to $s$, $\widehat{S_K}$ the formal scheme obtained by completing $S_K$ at $s_K$, and $\widehat{S_R}$ the formal scheme obtained by completing $S$ at $s$. Then the following are equivalent:
	\begin{enumerate}
		\item 	there exists a element $N\in R$ such that the isomonodromic deformation of $(\mathscr{E}, \nabla)$ over $\widehat{S_K}$ descends to $\widehat{S_R}[1/N]$
		\item For each embedding (equivalently, some embedding) $K\hookrightarrow \mathbb{C}$, the conjugacy class of the monodromy representation $$\rho: \pi_1(X(\mathbb{C})^{\text{an}})\to \on{GL}_r(\mathbb{C})$$ associated to $(\mathscr{E}, \nabla)_{\mathbb{C}}$ has finite orbit under $\pi_1(S(\mathbb{C})^{\text{an}}, s)$.
	\end{enumerate}
\end{conjecture}
Note that in this precisification we have elided the $\omega(p)$-integrality condition; we leave formulating this to the reader, or see \cite{lamlitt-painleve}.

This conjecture is an arithmetic answer to the general question that has concerned us in these notes. And it is closely related to the classical $p$-curvature conjecture:
\begin{proposition}[Lawrence-L--\,{\cite{lawrence2022representations}}]\label{prop:lawrence-litt}
	Suppose \autoref{conjecture:imprecise} is true for $\mathscr{X}/S=\mathscr{C}_g/\mathscr{M}_g$ (the universal curve of genus $g$ over the moduli space of genus $g$ curves) for all $g\gg 0$. Then \autoref{conjecture:gk-p-curvature} is true in general.
\end{proposition}
\begin{proof}[Proof sketch]
	It is well-known that the $p$-curvature conjecture can be reduced to the case of flat bundles on smooth proper curves. Now let $(\mathscr{E}, \nabla)$ be a flat bundle of rank $r$ on a smooth proper curve $X$ with vanishing $p$-curvature for almost all $p$, and choose a finite \'etale cover $Y\to X$ so that the genus $g$ of $Y$ is large enough that \autoref{conjecture:imprecise} holds true for $\mathscr{C}_g/\mathscr{M}_g$, and in particular at least $r^2$. By a direct computation with Taylor series, the leaf of the isomonodromy foliation through $[(\mathscr{E},\nabla)]\in \mathscr{M}_{dR}(\mathscr{X}/\mathscr{S})$ is $p$-integral to order $\omega(p)$.  Then \autoref{conjecture:precisification} implies that the monodromy of $(\mathscr{E}, \nabla)|_Y$ has finite orbit under $$\on{Mod}_g=\pi_1(\mathscr{M}_g);$$ hence by \autoref{theorem:canonical-reps}, $(\mathscr{E}, \nabla)$ has finite monodromy.
\end{proof}
\begin{remark}
	Here the key input was that the hypothesis of \autoref{conjecture:gk-p-curvature} is stable under pullback. In the sketch above we deduced the statement from \autoref{theorem:canonical-reps}, but in \cite{lawrence2022representations} we give a much more elementary argument.
\end{remark}
Our main evidence for \autoref{conjecture:precisification} is the following analogue of \autoref{theorem:katz-p-curvature}:
\begin{theorem}[Lam-L--\,\cite{lamlitt-painleve}]\label{theorem:non-abelian-p-curvature}
	\autoref{conjecture:precisification} is true if $(\mathscr{E},\nabla)_K$ is a Picard-Fuchs equation in the sense of \autoref{definition:picard-fuchs-equation}.
\end{theorem}
For example, when we consider $\mathscr{C}_g/\mathscr{M}_g$, \autoref{conjecture:precisification} completely characterizes finite $\on{Mod}_g$-orbits in $Y(g,0,r)$, as long as they are Picard-Fuchs equation for a single complex structure on the surface $\Sigma_g$.
\begin{remark}
We also prove a version of \autoref{theorem:non-abelian-p-curvature} in the non-proper setting in \cite{lamlitt-painleve}---e.g.~that of the Painlev\'e VI and Schlesinger equations---but because we have not described the isomonodromy foliation in this setting, we do not discuss it further here.
\end{remark}
In examples, the hypothesis that the flat bundle in question be a Picard-Fuchs equation does not seem unduly restrictive. For example, the following is a consequence of our discussion in \autoref{subsection:an-algebro-geometric-interpretation}:
\begin{theorem}
Let $(C_1, \cdots, C_n)$ be an $n$-tuple of quasi-unipotent conjugacy classes in $\on{SL}_2(\mathbb{C})$. Any finite orbit of the $\on{PMod}_{0,n}$-action on $Y(\underline{C})^{\text{irr}}$ corresponds to a local system of geometric origin.
\end{theorem}
\subsection{The idea of the proof}
The proof of \autoref{theorem:non-abelian-p-curvature} is too technical for us to do much more than gesture at it here (and we imagine many readers feel the same way about the statement). That said, we will try to sketch some of the key ideas. The key Hodge-theoretic input is the following result of Deligne:
\begin{theorem}[{\cite[Th\'eor\`eme 0.5]{deligne1987theoreme}}]
	Let $X$ be a smooth complex algebraic variety and $r\geq 0$ a positive integer. The set of isomorphism classes of rank $r$ complex local systems on $X$ which underly a polarizable $\mathbb{Z}$-variation of Hodge structure is finite.
\end{theorem}
We immediately deduce:
\begin{corollary}\label{corollary:deligne-finite-orbit}
	Let $X$ be a smooth complex variety, and $\Gamma\subset \on{Out}(\pi_1(X(\mathbb{C})^{\text{an}})$ a subgroup. Let $\mathbb{V}$ be a complex local system on $X$, and suppose that for each $\gamma\in \Gamma$, the local system $\mathbb{V}^\gamma$ underlies a $\mathbb{Z}$-variation of Hodge structure. Then the orbit of the isomorphism class of $\mathbb{V}$ under $\Gamma$ is finite.
\end{corollary}
In particular, in the setting and notation of \autoref{conjecture:precisification} this gives us a \emph{local} condition for a Picard-Fuchs equation on $X$ to have monodromy with finite orbit under $\pi_1(S(\mathbb{C})^{\text{an}}, s)$, as we now explain.
\begin{corollary}\label{corollary:infinitesimal-criterion}
	With notation as in \autoref{conjecture:precisification}, suppose that $(\mathscr{E}, \nabla)$ underlies a $\mathbb{Z}$-variation of Hodge structure on $X$, with Hodge filtration $F^\bullet$. Let $\widehat{X}_K$ be the formal scheme obtained by completing $\mathscr{X}_K$ along $X_K$, and $\widehat{S}_K$ the formal completion of $S_K$ at $s_K$. Let $$(\widehat{\mathscr{E}}, \widehat\nabla: \widehat{\mathscr{E}}\to \widehat{\mathscr{E}}\otimes \Omega^1_{\widehat{X}/\widehat{S}})$$ be the isomonodromic deformation of $(\mathscr{E},\nabla)$ to $\widehat{X}_K$. (Here $\mathscr{X}_K, X_K$ are the schemes $\mathscr{X}\otimes_R K, X\otimes_R K$.)
	
	If the Hodge filtration $F^\bullet$ extends to a Griffiths-transverse filtration on $(\widehat{\mathscr{E}}, \widehat\nabla)$, then the conjugacy class of the monodromy representation of $(\mathscr{E}, \nabla)$ has finite orbit under $\pi_1(S,s)$.
\end{corollary}
\begin{proof}
Let $\widetilde{S}$ be the universal cover of $S(\mathbb{C})^{\text{an}}$, and let $\mathscr{X}_{\widetilde{S}}$ be the base-change of $\mathscr{X}(\mathbb{C})^{\text{an}}$ to $\widetilde{S}$. Choosing a lift $\widetilde s\in \widetilde S$ of $s$, the inclusion $X\to \mathscr{X}_{\widetilde{S}}$ as the fiber over $\widetilde s$ induces an isomorphism of fundamental groups. Hence we have a local system $\mathbb{V}$ on $\mathscr{X}_{\widetilde{S}}$ with monodromy the same as that of $(\mathscr{E},\nabla)$. Let $\on{NL}(\mathscr{E}, \nabla)\subset \widetilde{S}$ be the set of $s'\in S$ such that the restriction of the fiber $\mathbb{V}$ to the fiber of $\mathscr{X}_{\widetilde{S}}\to \widetilde{S}$ at $s'$ underlies a polarizable $\mathbb{Z}$-variation of Hodge structure. By \cite[\S12]{simpson1997hodge} this is a closed analytic subset of $\widetilde{S}$. 

But the hypothesis that $F^\bullet$ extend to a formal neighborhood of $X$ in $\mathscr{X}$ implies that $\on{NL}(\mathscr{E}, \nabla)$ in fact contains an open subset of $\widetilde{S}$; hence it is all of $\widetilde{S}$. Translating this statement through the language of deck transformations, this implies that every conjugate of the monodromy representation of $(\mathscr{E}, \nabla)$ under $\pi_1(S,s)$ underlies a polarizable $\mathbb{Z}$-variation of Hodge structure; hence there are finitely many such conjugates by \autoref{corollary:deligne-finite-orbit}.
\end{proof}
\begin{remark}
	The notation $\on{NL}(\mathscr{E},\nabla)$ stands for ``Noether-Lefschetz" --- these loci are more or less what Simpson terms ``non-Abelian Noether-Lefschetz loci." They are the non-abelian analogue of Hodge loci, and Simpson asks \cite[consequence of Conjecture 12.3]{simpson1997hodge} if their images in $S$ are disjoint unions of algebraic subvarieties; this question is the non-abelian analogue of the famous theorem of Cattani-Kaplan-Deligne on algebraicity of Hodge loci. See \cite{engel2023non} for some partial results in this direction.
\end{remark}
Before explaining the idea of the proof, we need to briefly discuss the $p$-curvature of the isomonodromy foliation.

We define the $p$-curvature of a foliation:
\begin{definition}\label{definition:p-curvature-foliation}
	Let $k$ be a field of characteristic $p>0$, $X/k$ a smooth variety, and $\mathscr{F}\subset T_X$ a foliation, i.e.~ a subbundle closed under the Lie bracket. The $p$-curvature of $\mathscr{F}$ is the morphism $$\psi_p: F_{\text{abs}}^*\mathscr{F}\to T_X/\mathscr{F}$$ induced by $v\mapsto v^p$.
\end{definition}

Let $\pi:\mathscr{M}_{dR}(\mathscr{X}/S, r)\to S$ be the structure morphism. The isomonodromy foliation (at least over the smooth locus) is a splitting of the tangent exact sequence for $\pi$, and hence its $p$-curvature should be a morphism $$F_{\text{abs}}^*\pi^*T_S\to T_{\mathscr{M}_{dR}(\mathscr{X}/S, r)/S}.$$ At a characteristic $p>0$ point $[(X_s, \mathscr{E}, \nabla)]$ of $\mathscr{M}_{dR}(\mathscr{X}/S, r)$, this is a map $$\psi_p(X_s, \mathscr{E}, \nabla): F^*_{\text{abs}}T_sS\to \mathbb{H}^1(\on{End}(\mathscr{E})^\bullet_{dR}).$$
\begin{proposition-definition}\label{proposition-definition:p-curvature}
	The map $\psi_p(X_s, \mathscr{E}, \nabla)$ is given by the composition $$F_{\text{abs}}^*T_sS\to H^1(X'_s, T_{X'_s})\to \mathbb{H}^1((F_{\text{abs}}^*T_{X_s})^\bullet_{dR})\overset{\psi_p}{\to} \mathbb{H}^1(\on{End}(\mathscr{E})^\bullet_{dR})$$ where $X'_s$ is the Frobenius twist of $X_s$, the first map is the Frobenius pullback of the Kodaira-Spencer map, the second is induced by the natural inclusion $T_{X'_s}\hookrightarrow F_{\text{abs}}^*T_{X_s}$, and the third is induced by the $p$-curvature map for $\mathscr{E}$. Here $(F_{\text{abs}}^*T_{X_s})^\bullet_{dR}$ is the de Rham complex of $F_{\text{abs}}^*T_{X_s}$ with its canonical (Frobenius-pullback) connection, for which the flat sections are precisely $T_{X'_s}$.
\end{proposition-definition}

The above can be taken as a \emph{definition} of the $p$-curvature of the isomonodromy foliation for those who prefer not to work with crystals; that said, one may make sense of $p$-curvature for a crystal in positive characteristic \cite[\S2.3]{osserman2004connections}, and if one does so, the above is a computation of the $p$-curvature of $\mathscr{M}_{dR}(\mathscr{X}/S, r)/S$. The formula above is likely not very illuminating, but it is at least explicit.

\begin{proof}[Idea of the proof of \autoref{theorem:non-abelian-p-curvature}]
	It will be no surprise, now, that the idea of the proof of \autoref{theorem:non-abelian-p-curvature} is to verify the hypothesis of \autoref{corollary:infinitesimal-criterion}. That the Hodge filtration extends to first order is more or less an immediate consequence of Katz's comparison \cite[Theorem 3.2]{katz-p-curvature} of the Kodaira-Spencer map $\on{gr}^i(\nabla)$  with the associated graded of the $p$-curvature with respect to the conjugate filtration, discussed in our sketch of the proof of \autoref{theorem:katz-p-curvature}, combined with our formula for the $p$-curvature of the isomonodromy filtration, \autoref{proposition-definition:p-curvature}, and some elementary deformation theory. 
	
	To go beyond first order, some new ideas are needed---in particular, one needs to \begin{enumerate}
	 		\item (iteratively)	extend the conjugate filtration to a well-chosen mod $p$ isomonodromic deformation of $(\mathscr{E},\nabla)$, 
	 		\item compare the Kodaira-Spencer map to the associated graded of the $p$-curvature of this deformation of the conjugate filtration, and
	 		\item check that the $p$-curvature of the isomonodromy foliation vanishes along the chosen mod $p$ isomonodromic deformation of $(\mathscr{E},\nabla)$.
 		\end{enumerate}
 	Items (1) and (2) may be handled by judicious use of Ogus-Vologodsky's non-abelian Hodge theory \cite{ogus-vologodsky}. Item (3) relies on this work in combination with a slight extension of the theory of the Higgs-de Rham flow of Lan-Sheng-Zuo \cite{lan2019semistable}, as explained by Esnault-Groechenig \cite{esnault2023cristallinity}. Both of these ideas are beyond the scope of these notes.
\end{proof}
\begin{remark}\label{remark:esbt}
There is an analogue of the $p$-curvature conjecture for foliations, due to Ekedahl-Shephard--Barron-Taylor 	\cite{ekedahl99conjecture}, which gives a condition under which \emph{every} leaf of a foliation ought to be algebraic. This conjecture has been considered before in the context of the isomonodromy foliation in the PhD theses of Papaioannou and Menzies \cite{papaioannou2013algebraic, menzies2019p}. A variant of our \autoref{conjecture:precisification} has been considered for arbitrary rank one foliations in the preprint \cite[Conjecture 21]{movasatileaf}.
\end{remark}
\begin{remark}
	The approach to the proof of \autoref{theorem:non-abelian-p-curvature} is inspired in part by ideas of Katzarkov and Pantev's paper \cite{katzarkov1999non}, which proves a non-abelian analogue of the Theorem of the Fixed Part in the setting of a smooth projective morphism (the quasi-projective case was proven later by \cite{jost2001harmonic}). In fact one can deduce the main theorems of those papers from \autoref{theorem:non-abelian-p-curvature}, as will be explained in \cite{lamlitt-painleve}.
\end{remark}
\section{Analogues of conjectures on algebraic cycles}\label{section:non-abelian-conjectures}
The rest of these notes aim to collect a number of conjectures on the arithmetic and geometry of local systems on algebraic varieties, and some related questions in algebraic geometry and surface topology. We hope that they will be a fruitful starting point for young people interested in these subjects.

 We begin, however, with some philosophy.
\subsection{Non-abelian cohomology}\label{subsection:non-abelian-cohomology}
Let $\pi: \mathscr{X}\to S$ be a smooth proper morphism. The guiding principle here, due, we believe, to Simpson (see e.g.~\cite{simpson1997hodge}), is that the spaces $\mathscr{M}_{dR}(\mathscr{X}/S, r), \mathscr{M}_B(\mathscr{X}/S, r)$, parametrizing local systems on the fibers of $\pi$, are the non-abelian analogues of the cohomology sheaves $$R^i\pi_*\Omega_{\mathscr{X}/S, dR}^\bullet, R^i\pi_*\mathbb{C},$$ the de Rham and singular cohomology of the fibers of $\pi$. Thus any structure that exists on classical cohomology groups, and any conjecture as to their behavior, should have a non-abelian analogue. For example, the Hodge and Tate conjectures, characterizing cohomology classes corresponding to algebraic cycles, have analogues characterizing local systems of geometric origin. 

The characterization of local systems of geometric origin (defined as in \autoref{definition:geometric-origin}) can be thought of as some way of making precise our goal from \autoref{section:introduction}, of understanding something about the topology of algebraic maps in terms of representations of fundamental groups of algebraic varieties.
\begin{example}[The Hodge conjecture]\label{example:non-abelian-hodge}
	Let $X$ be a smooth projective variety over the complex numbers. The Hodge conjecture predicts that the image of the cycle class map $$\on{cl}^i: CH^i(X)\otimes \mathbb{Q}\to H^{2i}_{\text{sing}}(X(\mathbb{C})^{\text{an}}, \mathbb{C}(i))$$ is the $\mathbb{Q}$-span of the intersection $$H^{2i}_{\text{sing}}(X(\mathbb{C})^{\text{an}}, \mathbb{Z}(i))\cap H^{0,0}(X(\mathbb{C})^{\text{an}}, \mathbb{C}(i)).\footnote{Here the Tate twist $\mathbb{C}(i)$ has the effect of lowering the weight by $2i$, so the target has index $(0,0)$ instead of index $(i,i)$, as in some common statements of the Hodge conjecture.}$$
	We rewrite this as follows. There is a natural action of $\mathbb{C}^\times$ on $H^{2i}_{\text{sing}}(X(\mathbb{C})^{\text{an}}, \mathbb{C}(i))$, where for $v\in H^{p,q}(X(\mathbb{C})^{\text{an}}, \mathbb{C}(i))$, $\lambda\cdot v=\lambda^pv$. Then the Hodge conjecture says that the image of the cycle class map is spanned by $$H^{2i}_{\text{sing}}(X(\mathbb{C})^{\text{an}}, \mathbb{Z}(i))\cap H^{2i}(X(\mathbb{C})^{\text{an}}, \mathbb{C}(i))^{\mathbb{C}^\times}.$$
	The non-abelian analogue of this statement is \cite[Conjecture 12.4]{simpson1997hodge}: 
	\begin{conjecture}\label{conjecture:non-abelian-hodge}
		$\mathbb{Z}$-local systems on $X$ underlying (polarizable) complex variations of Hodge structure are of geometric origin.
	\end{conjecture}
 We write this in a form analogous to that of the usual Hodge conjecture, above. Namely, we let $M_{\text{Dol}}(X, r)$ be the (coarse) moduli space of semistable Higgs bundles of degree zero. Here a Higgs bundle is a pair $(\mathscr{E}, \theta: \mathscr{E}\to \mathscr{E}\otimes \Omega^1_X)$ with $\theta$ an $\mathscr{O}_X$-linear map such that the natural composition $$\theta\circ \theta: \mathscr{E}\to\mathscr{E}\otimes\Omega^2_X$$ is identically zero (see \cite[\S6]{simpson1995moduli} for details). The map $\theta$ is referred to as a \emph{Higgs field}.
	
	There is a natural $\mathbb{C}^\times$-action on $M_{\text{Dol}}(X,r)$, given by scaling the Higgs field: $$\lambda\cdot(\mathscr{E}, \theta)=(\mathscr{E}, \lambda\theta).$$ Moreover there is a natural (real-analytic!) homeomorphism $M_{\text{Dol}}(X, r)\simeq M_B(X,r)$ \cite[Theorem 7.18]{simpson1995moduli}. Now Simpson's \cite[Conjecture 12.4]{simpson1997hodge} may be rephrased as saying that the points of $$M_B(X,r)(\mathbb{Z})\cap M_{\text{Dol}}(X,r)^{\mathbb{C}^\times}$$ correspond to local systems of geometric origin, where we make sense of the intersection here using the homeomorphism of the previous sentence.
	
	There is little evidence for this conjecture (though I believe it!), aside from the important case of polarizable $\mathbb{Z}$-variations of Hodge structure of weight $1$, which come from Abelian schemes, and rigid local systems, where much is known by work of Katz \cite{katz-rigid}, Esnault-Groechenig \cite{esnault2018cohomologically, esnault2020rigid} and others; see \autoref{subsection:rigid-local-systems} for further discussion.
\end{example}
\begin{example}[The Tate conjecture]\label{example:non-abelian-Tate}
	Let $X$ be a smooth projective variety over a finitely generated field $K$, $K^s$ a separable closure of $K$, and $\ell$ a prime different from the characteristic of $K$. One form of the Tate conjecture is that the natural cycle class map
	$$\on{cl}^i: \on{CH}^i(X_{K^s})\otimes \mathbb{Q}_\ell\to H^{2i}(X_{K^s}, \mathbb{Q}_\ell(i))$$ 
	has image precisely $$\varinjlim_{K'/K}H^{2i}(X_{K^s}, \mathbb{Q}_\ell(i))^{\on{Gal}(K^s/K')},$$
	where the limit is taken over all finite extensions of $K$ in $K^s$. That is, the image of the cycle class map in $\ell$-adic cohomology is precisely the set of elements of $H^{2i}(X_{K^s}, \mathbb{Q}_\ell(i))$ with finite Galois orbit. Again this conjecture has a non-abelian analogue:
	\begin{conjecture}[Fontaine-Mazur-Petrov\,{\cite[Conjecture 1 bis]{petrov2023geometrically}}]\label{conjecture:fontaine-mazur}
		Let $\mathbb{V}$ be an irreducible $\ell$-adic local system on $X_{K^s}$. Then $\mathbb{V}$ is of geometric origin if and only if it its isomorphism class has finite orbit under $\on{Gal}(K^s/K)$.
	\end{conjecture}
	
	That is, the $\ell$-adic local systems on $X$ of geometric origin and rank $r$ are precisely $$\varinjlim_{K'/K}\on{M}_B(X,r)(\overline{\mathbb{Q}_\ell})^{\on{Gal}(K^s/K')}.$$
	Our primary evidence for this conjecture comes from the case where $X$ is a curve and $K$ is finite, where the conjecture follows from work of Lafforgue \cite{lafforgue2002chtoucas}.
	
		If $X$ is any variety over an algebraically close field $L$ (say of characteristic zero), and $\mathbb{V}$ is an $\ell$-adic local system on $X_{K}$, we may descend $X$ to a finitely generated field $K$ and $\mathbb{V}$ to $X_{K^s}$. If the isomorphism class of $\mathbb{V}$ has finite orbit under $\on{Gal}(K^s/K)$ we say it is \emph{arithmetic}. See \cite{littarithmetic1, littarithmetic2} for some discussion of this property.
\end{example}
\begin{remark}
	Prior to Petrov's work \cite{petrov2023geometrically}, the standard version of \autoref{conjecture:fontaine-mazur} included a $p$-adic Hodge theory condition, namely that $\mathbb{V}$ be \emph{de Rham} (see e.g.~\cite{liu2017rigidity}). Inspired by the analogy with the Tate conjecture, I conjectured and Petrov proved that this condition was redundant.
\end{remark}
\begin{remark}
Much of my interest in local systems on topological surfaces with finite orbit under the mapping class group comes from \autoref{conjecture:fontaine-mazur}; I view this conjecture as an arithmetic analogue of our topological questions about classifying such orbits, where the mapping class group is analogous to the absolute Galois group of $K$. See \autoref{proposition:hodge-to-arithmetic-rigidity} for a concrete relationship between the two, conditional on a plausible conjecture in surface topology (namely \autoref{conjecture:mod-rigidity}).
\end{remark}

\begin{example}[The non-abelian Ogus conjecture]
	We mention one more analogue along these lines. The Ogus conjecture \cite[Problem 2.4]{ogus1981hodge}  predicts which classes in algebraic de Rham cohomology arise from cycles in terms of the conjugate filtration. The analogous conjecture for local systems is:
	\begin{conjecture}[{\cite[Appendix to Chapter V]{andre1989g}}]\label{conjecture:andre-p-curvature}
		Let $X$ be a variety over a finitely generated integral $\mathbb{Z}$-algebra $R$, with fraction field $K$ of characteristic zero. Let $(\mathscr{E},\nabla)$ be a flat vector bundle on $X/R$. Then $(\mathscr{E}, \nabla)|_{X_K}$ is of geometric origin if and only if there exists a dense open subset $U\subset  \on{Spec}(R)$ such that for all $\mathfrak{p}\in U$, the $p$-curvature of $(\mathscr{E},\nabla)|_{X_{\mathfrak{p}}}$ is nilpotent.
	\end{conjecture}
	Here we view the $p$-curvature $\psi_p$ of $(\mathscr{E},\nabla)|_{X_{\mathfrak{p}}}$ as a map $$\psi_p: \mathscr{E}|_{X_{\mathfrak{p}}}\to \mathscr{E}|_{X_{\mathfrak{p}}}\otimes F_{\text{abs}}^*\Omega^1_{X_{\mathfrak{p}}}$$ by adjunction; we say it is nilpotent if there exists a filtration on $\mathscr{E}|_{X_{\mathfrak{p}}}$ such that the $p$-curvature vanishes on the associated graded vector bundle. For local systems of geometric origin, one may take this filtration to be the conjugate filtration.
\end{example}

There are other conjectural characterizations of local systems of geometric origin. For example, arguably \cite[Conjecture 10.3.1 and Conjecture 10.4.1(3)]{budur2020absolute} give analogues of the conjecture that absolute Hodge cycles are cycle classes. We also remark that Krishnamoorthy-Sheng conjecture \cite[Conjecture 1.4]{krishnamoorthy2020periodic} that \emph{periodic de Rham bundles}\footnote{Here periodicity refers to their behavior under the Higgs-de Rham flow of Lan-Sheng-Zuo \cite{lan2019semistable}.} on curves are of geometric origin; I am not sure if there is an ``abelian" analogue of this conjecture, characterizing the image of the cycle class map in similar terms, in the literature.

The conjectures listed here and their variants are quite beautiful, but we have little evidence for them. In fact there are few cases where these conjectural equivalent conditions for a local system to be of geometric origin are even known to be equivalent to one another!
\begin{question}
	Can one prove that any of the conditions here expected to be equivalent to being of geometric origin be proven to imply one another? For example, can one show that, after choosing an isomorphism $\mathbb{C}\simeq \overline{\mathbb{Q}_\ell}$, an arithmetic $\overline{\mathbb{Q}_\ell}$-local system underlies an integral variation of Hodge structure?
\end{question}
The only result I know along these lines is \autoref{corollary:deligne-finite-orbit}, which shows that certain variations of Hodge structure are arithmetic in a weak sense, as we now explain.
\begin{proposition}\label{proposition:hodge-to-arithmetic}
Let $\mathscr{X}\to S$ be a smooth proper morphism of complex varieties. Let $s\in S$ be very general, and set $X=\mathscr{X}_s$. Let $\eta$ be the generic point of $S$ and $\overline{\eta}$ a geometric generic point. Let $\mathbb{V}$ be a complex local system on $X$ underlying a polarizable $\mathbb{Z}$-variation of Hodge structure, and for each $\ell$ let $\mathbb{V}_\ell$ be the corresponding $\ell$-adic local system. Fix a specialization isomorphism $\on{sp}: \pi_1^{\text{\'et}}(\mathscr{X}_{\overline{\eta}})\to \pi_1^{\text{\'et}}(X_s)$. Then the isomorphism class of the local system $\on{sp}^*(X_{\overline{\eta}})\mathbb{V}_\ell$ obtained by pulling back $\mathbb{V}_\ell$ through this specialization map has finite orbit under $\on{Gal}(\overline{\eta}/\eta)$.
\end{proposition}
\begin{remark}
	The upshot of the above proposition is that $\mathbb{Z}$-variations of Hodge structure very general fibers of a smooth proper morphism $\mathscr{X}\to S$ have a weak arithmeticity property---we do not know how to show they have finite orbit under the Galois group of a finitely-generated field over which they spread out, but they do have finite orbit under a ``geometric subgroup" of this Galois group, coming from the geometric fundamental group of $S$.
\end{remark}
\begin{proof}[Proof of \autoref{proposition:hodge-to-arithmetic}]
	This is a reinterpretation of \autoref{corollary:deligne-finite-orbit}, closely following the proof of \autoref{corollary:infinitesimal-criterion}. As $s$ is very general, we obtain from the closedness of non-abelian Noether-Lefschetz loci \cite[\S12]{simpson1997hodge} that the Hodge filtration on the flat bundle corresponding to $\mathbb{V}$ extends to its isomonodromic deformation, whence $\mathbb{V}$ has finite $\pi_1(S,s)$-orbit by \autoref{corollary:infinitesimal-criterion}. The rest of the statement is a translation of this fact into the language of \'etale fundamental groups. 
\end{proof}
Below we include a table listing some of the ``abelian" aspects of cohomology, and their non-abelian analogues.
\begin{center}
\begin{tabular}{ |c|c| } 
 \hline
 Abelian & Non-abelian\\
 \hline
 de Rham cohomology & $\mathscr{M}_{dR}(X, r)$  \\ 
 Betti cohomology & $\mathscr{M}_B(X, r)$  \\ 
 Dolbeault cohomology & $\mathscr{M}_{\text{Dol}}(X, r)$ \\
 $F^1H^1_{dR}(X, \mathbb{C})$ & $\mathbb{C}$-VHS locus \\ 
 Hodge classes & $\mathbb{Z}$-VHS locus \\
 The Hodge conjecture & \autoref{conjecture:non-abelian-hodge} \\
 The Tate conjecture & \autoref{conjecture:fontaine-mazur}\\
\autoref{theorem:katz-p-curvature} & \autoref{theorem:non-abelian-p-curvature}\\ 
$H^1(X, U(1))\subset H^1(X, \mathbb{C}^\times)$ Zariski-dense & \autoref{question:dense-subloci}(1)\\
$H^1(X, \overline{\mathbb{Z}}^\times)\subset H^1(X, \mathbb{C}^\times)$ Zariski-dense & \autoref{question:dense-subloci}(2)\\
 \hline
\end{tabular}
\end{center}

\subsection{A case study}\label{subsection:cayley-solutions}
While we know little about the conjectures in the previous section, there is one case we more or less understand completely, due to beautiful conjectures of Sun-Yang-Zuo \cite[Conjectures 4.8, 4.9, 4.10]{sun2021projective}, proven independently by Lin-Sheng-Wang \cite{lin2022torsion} and Yang-Zuo \cite{yang2023constructing}, and by completely different methods by Lam and myself \cite{lamlitt}. We give a brief exposition of a variant of these conjectures, with an eye towards verifying our conjectural characterization of local systems of geometric origin in a very special case.

We return to our friend $$X=\mathbb{CP}^1\setminus\{x_1, \cdots, x_n\},$$ this time with $n=4$. By applying a fractional linear transformation we may assume $$x_1=0, x_2=1, x_3=\infty$$ and set $x_4=\lambda$. Our goal will be to classify rank $2$ local systems of geometric origin on $X$ with local monodromy in the conjugacy class $$C_1=C_2=C_3=\left[\begin{pmatrix} 1 & 1\\ 0 & 1\end{pmatrix}\right]$$ at $0,1,\lambda$ and in the conjugacy class $$C_4=\left[\begin{pmatrix} -1 & 1 \\ 0 & -1\end{pmatrix}\right]$$ at $\infty$. By \autoref{theorem:cp1-geometric-origin} we already know how to do so when $\lambda$ is \emph{generic}: these local systems are related, by middle convolution, to local systems of rank $2$ on $X$ with monodromy a finite complex reflection group.

In fact this is true not just for $\lambda$ generic. As in \autoref{section:matrices}, let $Y(\underline{C})$ be the space of isomorphism classes of local systems on $X$ with local monodromy at $x_i$ in the conjugacy class $C_i$. We have:
\begin{proposition}
	Let $p: E_\lambda\to \mathbb{P}^1$ be the double cover branched at $0,1,\infty, \lambda$. Let $\mathbb{V}$ be an irreducible local system in $Y(\underline{C})$. Then there exists a rank one local system $\mathbb{L}$ on the elliptic curve $E_\lambda$ such that $$\mathbb{V}=\on{MC}_{-1}(p_*\mathbb{L}|_X)=\on{MC}_{-1}(p_*\mathbb{L}^\vee|_X).$$  This construction yields a bijection between dual pairs of local systems of rank $1$ on $E_\lambda$ (of order not dividing $2$), and irreducible local systems $\mathbb{V}$ in $Y(\underline{C})$.
\end{proposition}
\begin{proof}[Proof sketch]
	A local computation shows that $\on{MC}_{-1}(\mathbb{V})$ has local monodromy conjugate to $$\begin{pmatrix} 1 & 0 \\ 0 & -1\end{pmatrix}$$ each of $0,1,\infty,\lambda$; hence the pullback $p^*\on{MC}_{-1}(\mathbb{V})$ extends to all of $E_\lambda$. As $\mathbb{V}$ is irreducible, the same is true for $\on{MC}_{-1}(\mathbb{V})$, whence $p^*\on{MC}_{-1}(\mathbb{V})$ is semisimple, hence (as $\pi_1(E_\lambda)$ is abelian) it is a direct sum of two rank one local systems $\mathbb{L}_1, \mathbb{L}_2$. We must have that $\mathbb{L}_1\oplus\mathbb{L}_2$ is self-dual, as it is pulled back from $X$, hence isomorphic to its pullback under inversion on $E_\lambda$. 
	
	Now $\on{MC}_{-1}(\mathbb{V})$ must be isomorphic to either $p_*\mathbb{L}_1|_X$ or $p_*\mathbb{L}_2|_X$. Hence one of the $\mathbb{L}_i$ (WLOG $\mathbb{L}_1$) must have order greater than $2$, as otherwise $\on{MC}_{-1}(\mathbb{V})$ would be reducible. But then $\mathbb{L}_1$ is not self-dual, hence must be isomorphic to $\mathbb{L}_2^\vee$. Now the result follows from the invertibility of $\on{MC}_{-1}$.
\end{proof}

Here $\on{MC}$ is the middle convolution operation discussed in \autoref{section:matrices}.

One may easily deduce the following from the fact that rank one local systems of geometric origin have finite order, the fact that $\on{MC}_\lambda$ for $\lambda$ a root of unity preserves the property of being of geometric origin, and the invertibility of the middle convolution:
\begin{corollary}[{\cite[Proof of Proposition 4.2.2]{lamlitt}}]\label{corollary:mc-finite-order}
	An irreducible local system $\mathbb{V}$ in $Y(\underline{C})$ is of geometric origin if and only if it has the form $$\mathbb{V}=\on{MC}_{-1}(p_*\mathbb{L}|_X)$$ for $\mathbb{L}$ a non-trivial rank one local system of finite order.
\end{corollary}
As we have a very good understanding of the Galois action on rank one local systems on an elliptic curve, their $p$-curvature, and so on, one may (not entirely trivially) deduce:
\begin{corollary}
	\autoref{conjecture:non-abelian-hodge}, \autoref{conjecture:fontaine-mazur},  \autoref{conjecture:andre-p-curvature}, \cite[Conjecture 10.4.1(3)]{budur2020absolute}, and \cite[Conjecture 1.4]{krishnamoorthy2020periodic} are true for local systems in $Y(\underline{C})$.
\end{corollary}
\begin{proof}[Proof sketch]
Note that $\mathbb{V}$ is of geometric origin if and only if the same is true for $\on{MC}_{-1}(\mathbb{V})$; this latter is of geometric origin if and only if the same is true for $\mathbb{L}$ as in \autoref{corollary:mc-finite-order}. But middle convolution preserves the properties of being an $\mathscr{O}_K$-variation of Hodge structure, being arithmetic, having nilpotent $p$-curvature, and being an absolute $\overline{\mathbb{Q}}$-point (in the sense of \cite{budur2020absolute}). \cite[Conjecture 1.4]{krishnamoorthy2020periodic}  amounts to showing that $\on{MC}_{-1}$ commutes with the Higgs-de Rham flow; for this see \cite{lamlitt} in the special case under consideration.  
\end{proof}
\begin{remark}
We say nothing else about the Higgs-de Rham flow in these notes; let me just remark for experts that it would be interesting to show that the middle convolution commutes with the Higgs-de Rham flow in general.
\end{remark}
\subsection{Rigid local systems}\label{subsection:rigid-local-systems}
\autoref{conjecture:fontaine-mazur} makes a strong prediction, which was explicitly conjectured by Simpson; most of our evidence for the conjectures in \autoref{subsection:non-abelian-cohomology} comes from studying this prediction. Before stating it, we make a definition:
\begin{definition}\label{definition:rigidity}
	Let $X$ be a smooth projective variety. An irreducible local system $\mathbb{V}$ of rank $r$ on $X$ is \emph{rigid} if it corresponds to an isolated point of ${M}_B(X,r)$.
\end{definition}
\begin{conjecture}[{\cite[Conjecture 1.1]{langer2018rank}}]\label{conjecture:rigid-to-motivic}
	Any rigid local system is of geometric origin.
\end{conjecture}
Indeed, a rigid local system is necessarily arithmetic in the sense of \autoref{example:non-abelian-Tate}, as there are finitely many rigid local systems (since $M_B(X,r)$ is of finite type) and they are permuted by the absolute Galois group of any finitely-generated field to which $X$ descends.

We have substantial evidence for \autoref{conjecture:rigid-to-motivic}. Simpson showed \cite[Theorem 5]{simpson1992higgs}  that rigid local systems underly $K$-variations of Hodge structure, for $K$ some number field. The papers \cite{corlette-simpson} and \cite{langer2018rank} show that the conjecture holds for $\on{SL}_r(\mathbb{C})$-local systems with $r\leq 3$. 

More recently, a spectacular series of papers by Esnault and Groechenig \cite{esnault2018cohomologically, esnault2020rigid, esnault2023cristallinity}, has verified a number of predictions of \autoref{conjecture:rigid-to-motivic}. For example, local systems of geometric origin are direct summands of $\mathbb{Z}$-local systems, and hence are defined not just over a number field $K$ but in fact over the ring of integers of $\mathscr{O}_K$. We do not know how to prove this for rigid local systems in general, but Esnault and Groechenig do so in an important special case (as used crucially  in the proof of \autoref{theorem:canonical-reps}).
\begin{definition}\label{definition:cohomologically-rigid}
	Let $X$ be a smooth projective variety. An irreducible complex local system $\mathbb{V}$ on $X$ is \emph{cohomologically rigid} if $H^1(X, \on{ad}(\mathbb{V}))=0$. 
\end{definition}
\begin{remark}
 \autoref{definition:cohomologically-rigid} has a natural geometric interpretation as a strengthening of \autoref{definition:rigidity} --- a cohomologically rigid local system corresponds to an \emph{reduced} isolated point of $M_B(X, r)$. 
\end{remark}
\begin{theorem}[Special case of {\cite[Theorem 1.1]{esnault2018cohomologically}}]
	A cohomologically rigid local system is defined over the ring of integers $\mathscr{O}_K$ of some number field $K$.
\end{theorem}
\begin{remark}
It is natural to ask if there exist rigid local systems which are not cohomologically rigid. An example was found by de Jong, Esnault, and Groechenig \cite{de2022rigid}. However, this example does have a form of cohomological rigidity, as we now explain. Let $G$ be an algebraic group over $\mathbb{C}$, with Lie algebra $\mathfrak{g}$. Let $X$ be a smooth projective variety and $\mathbb{V}$ a $G(\mathbb{C})$-local system on $X$. We say that $\mathbb{V}$ is $G$-cohomologically rigid if $H^1(X, \mathfrak{g})=0$, where we view $\mathfrak{g}$ as a local system on $X$ via the adjoint representation of $G$. 

Klevdal and Patrikis \cite{klevdal2022G} show that $G$-cohomologically rigid local systems are integral, building on Esnault-Groechenig's work in the case that $G$ has identity component $\on{SL}_r$. We believe the examples of de Jong, Esnault, and Groechenig are $G$-cohomologically rigid, where we take $G$ to be the Zariski closure of their monodromy.
\end{remark}
\begin{question}
	Do there exist irreducible local systems which are rigid but not $G$-cohomologically rigid, where $G$ is the Zariski-closure of their monodromy group?
\end{question}
We mention two other pieces of evidence for \autoref{conjecture:rigid-to-motivic}. First, Katz \cite{katz-rigid} shows, as a consequence of his classification of rigid local systems on $\mathbb{P}^1\setminus\{x_1, \cdots, x_n\}$ (discussed in \autoref{subsection:existence-uniqueness}), that such rigid local systems on  $\mathbb{P}^1\setminus\{x_1, \cdots, x_n\}$, with quasi-unipotent local monodromy at the $x_i$, are of geometric origin. Second, Esnault and Groechenig show \cite{esnault2020rigid, esnault2023cristallinity} that such rigid local systems have nilpotent $p$-curvature, i.e.~they satisfy \autoref{conjecture:andre-p-curvature}.
\begin{remark}
All the conjectures and results in this section have analogues for smooth quasi-projective varieties; in this case one studies local systems with fixed quasi-unipotent monodromy about the boundary divisors of a strict normal crossings compactification. In particular, Katz's results on $\mathbb{P}^1\setminus\{x_1, \cdots, x_n\}$ fall into this paradigm, and the results of Esnault-Groechenig and Klevdal-Patrikis mentioned above hold in this generality.
\end{remark}

\subsection{Special points and dense subloci}\label{subsection:special-points-and-dense-subloci}
As we have just seen, the conjectures of \autoref{subsection:non-abelian-cohomology}, and \autoref{conjecture:rigid-to-motivic}, ``explain" the zero-dimensional components of the moduli of local systems on a smooth projective variety $X$---they all (conjecturally) correspond to local systems on $X$ of geometric origin. It is natural to seek structural properties of the moduli of local systems on $X$ which generalize these conjectures to the entire moduli space.

One such conjecture was proposed separately by Esnault-Kerz and Budur-Wang:
\begin{conjecture}[{\cite[Conjecture 1.1]{esnault2023local}} and Budur-Wang {\cite[Conjecture 10.3.1]{budur2020absolute}}]\label{conjecture:ekbw2}
	Let $X$ be a smooth variety. The local systems of geometric origin are Zariski-dense in $M_B(X, r)$.
\end{conjecture}
Of course this conjecture would imply \autoref{conjecture:rigid-to-motivic}, by considering density of local systems of geometric origin in the zero-dimensional components of $M_B(X,r)$. Aside from a desire to generalize \autoref{conjecture:rigid-to-motivic}, we believe that Esnault and Kerz were motivated by desired applications to hard Lefschetz theorems for local systems, while Budur and Wang were motivated by analogies with the Andr\'e-Oort conjecture. 

Taking $X$ to be a generic curve, \autoref{conjecture:ekbw2} implies \autoref{conjecture:ekbw}, as a local system of geometric origin on a generic curve will spread out to a family of curves as in \autoref{proposition:AG-interpretation-higher-genus}, and hence have finite mapping class group orbit.  As \autoref{conjecture:ekbw} is false by \autoref{theorem:canonical-reps}, the same is true for \autoref{conjecture:ekbw2}; this was part of our motivation for proving \autoref{theorem:canonical-reps}.

Nonetheless we have found the philosophy behind \autoref{conjecture:ekbw2} to be inspirational. The goal of this section is to propose some ways that it might be salvaged. One plausible approach, pursued already in \cite{esnault2023local} and \cite{de2024integrality}, is to consider local systems having some, but not all, of the properties of local systems of geometric origin, and to try to prove their Zariski-density in the moduli of all local systems. For example:
\begin{theorem}[{\cite{esnault2023local}}]
	Let $X$ be a quasi-projective complex variety with $\overline{X}$ a simple normal crossings compactification.
The local systems on $X$ with quasi-unipotent local monodromy about the components of $\overline{X}\setminus X$ are Zariski-dense  in $M_B(X,r)$.
\end{theorem}
As local systems with geometric origin have this quasi-unipotent local monodromy property, this was intended as evidence for \autoref{conjecture:ekbw2}. Esnault-de Jong define a notion of ``weakly arithmetic" local system (a weakening of the notion discussed in \autoref{example:non-abelian-Tate}) and prove the density of such local systems in $M_B(X, r)$.

Here we propose some analogous questions:
\begin{question}\label{question:dense-subloci}
Let $X$ be a smooth projective variety over $\mathbb{C}$. Are the following loci Zariski-dense in $M_B(X, r)$:
\begin{enumerate}
	\item The locus of $\mathbb{C}$-variations of Hodge structure?
	\item The locus of local systems defined over $\overline{\mathbb{Z}}$?
\end{enumerate}
Does there exist a fixed number field $K$ such that the $\mathscr{O}_K$-points of $\mathscr{M}_B(X,r)$ are Zariski-dense in its $\mathbb{C}$-points? It is also natural also to ask about stronger forms of density than Zariski-density; for example \cite{bourgain2016markoff} and \cite{chen2020nonabelian} consider a form of strong approximation for the character varieties discussed in \autoref{subsubsection:n-4}.
\end{question}
For experts, we also give a $p$-adic analogue of this question:
\begin{question}\label{question:p-adic-periodic}
	Let $X$ be a smooth projective variety over $W(k)$, the ring of Witt vectors of a finite field $k$ of characteristic at least $3$. There is a ``Frobenius pullback" map $$F^*: \mathscr{M}_{dR}(X,r)(W(k))\to \mathscr{M}_{dR}(X,r)(W(k))$$ (see e.g.~\cite[\S3.1]{esnault2023cristallinity}). Are the periodic points of this map Zariski-dense?
\end{question}
\begin{remark}
The density of \emph{weakly arithmetic} local systems, proven by \cite{de2024integrality}, is an $\ell$-adic version of \autoref{question:p-adic-periodic}. For a stronger (still open) version, see \cite[Question 5.2.2]{kadets2022level}.
\end{remark}

We now briefly discuss some evidence, aside from the case of (cohomologically) rigid local systems, that the answer \autoref{question:dense-subloci} might be positive.
\subsubsection{Density of $\mathbb{C}$-VHS}
One result that gives hope that \autoref{question:dense-subloci}(1) might have a positive answer is the following famous result of Simpson, a generalization of which (due to Mochizuki) was used in the proof of \autoref{theorem:canonical-reps}.
\begin{theorem}[{\cite[Theorem 3]{simpson1992higgs}}]\label{theorem:deformation-to-vhs}
	Let $X$ be a smooth projective variety and $\mathbb{V}$ an irreducible local system on $X$. Then $\mathbb{V}$ can be deformed to a complex variation of Hodge structure.
\end{theorem}
\begin{proof}[Proof sketch]
The proof proceeds by considering the stable Higgs bundle $(\mathscr{E}, \theta)$ corresponding to $\mathbb{V}$ under the non-abelian Hodge correspondence. Simpson shows (using the properness of the Hitchin map) that the limit as $t\to 0$ of $(\mathscr{E}, t\theta)$ in $M_{\text{Dol}}(X, r)$ exists. The limit is evidently a $\mathbb{C}^\times$-fixed point, and Simpson shows these are precisely the Higgs bundles corresponding to complex variations of Hodge structure. The result now follows from the fact that $M_{\text{Dol}}(X, r)$ and $M_B(X, r)$ are (real-analytically) isomorphic, and in particular homeomorphic.
\end{proof}
In fact the proof shows that there is a $\mathbb{C}$-VHS in each irreducible component of $M_B(X,r)$.

The following is part of ongoing joint work with Botong Wang and Ruijie Yang:
\begin{theorem}[L--, Wang, Yang]
	Suppose $X$ is a compact Riemann surface of genus $g$. Then the $\mathbb{C}$-VHS locus is Zariski-dense in $M_B(X, r)$. If $(C_1, \cdots, C_n)$ are a collection of semisimple quasi-unipotent conjugacy classes in $GL_2(\mathbb{C})$, then the $\mathbb{C}$-VHS locus is Zariski-dense in the character variety parametrizing rank $2$ local systems on $\mathbb{P}^1\setminus\{x_1, \cdots, x_n\}$, with local monodromy about $x_i$ lying in $C_i$.
\end{theorem}
The first part of this theorem is almost trivial; the locus of unitary local systems is Zariski-dense. The second part is a bit more involved, and we say nothing more about it here, except to note that already the locus of \emph{supermaximal} local systems, in the sense of \cite{deroin2019supra}, are Zariski dense. See \cite[Remark 4.3.4]{lam2023finite} for a discussion of supermaximal local systems in a related context.

\begin{example}
Some evidence against the most optimistic possible answer here is given by considering the case of local systems on $\mathbb{P}^1\setminus\{x_1, \cdots, x_4\}$, where the local monodromy around each puncture is conjugate to $$\begin{pmatrix} 1 & 1 \\ 0 & 1\end{pmatrix}.$$ Suppose one has a $\mathbb{C}$-VHS on $\mathbb{P}^1\setminus\{x_1, \cdots, x_4\}$ with these local monodromies---the corresponding Higgs bundle is necessarily stable of degree zero and non-unitary, hence corresponds to a direct sum of two line bundles $\mathscr{O}_{\mathbb{P}^1}(a)\oplus \mathscr{O}_{\mathbb{P}^1}(-a),$ with $a>0$ and non-zero Higgs field $$\theta: \mathscr{O}_{\mathbb{P}^1}(a)\to  \mathscr{O}_{\mathbb{P}^1}(-a)\otimes \Omega^1_{\mathbb{P}^1}(x_1+\cdots+x_4)= \mathscr{O}_{\mathbb{P}^1}(2-a).$$ The only possibility is that $a=1$ and $\theta$ is an isomorphism, whence this Higgs bundle is unique up to isomorphism---hence there is a unique $\mathbb{C}$-VHS with the given local monodromies (this is the so-called \emph{uniformizing local system}). But the character variety parametrizing local systems on $\mathbb{P}^1\setminus\{x_1, \cdots, x_4\}$ with the given local monodromies is positive-dimensional, whence the $\mathbb{C}$-VHS locus is not Zariski dense. (See \cite[Theorem 1.1]{lam2022motivic} for some closely related examples.)
\end{example}
\subsubsection{Density of integral points}\label{subsubsection:integral-points}
We now briefly discuss \autoref{question:dense-subloci}(2) and its variants. We are quite far from even the weakest possible forms of a positive answer. For example, as far as I know the following is open:
\begin{question}\label{question:exist-OK-point}
	Let $X$ be a smooth projective variety over $\mathbb{C}$, and suppose there exists an irreducible complex local system of rank $r$ on $X$. Does there exist a number field $K$ and an irreducible $\mathscr{O}_K$-local system on $X$ of rank $r$?
\end{question}
Some evidence for a positive answer to \autoref{question:exist-OK-point} is given by the following:
\begin{theorem}[Special case of {\cite[Theorem 1.1]{de2024integrality}}]
	Let $X$ be a smooth projective variety over $\mathbb{C}$, and suppose there exists an irreducible complex local system of rank $r$ on $X$. Then for each prime $\ell$ there exists an irreducible $\overline{\mathbb{Z}_\ell}$-local system on $X$.
\end{theorem}
The proof relies on both the arithmetic and geometric Langlands program.

A great deal of work has been done on related questions in the case of character varieties of surfaces. For example, it is not hard to show:
\begin{theorem}
Let $X$ be a compact Riemann surface of genus $g\geq 2$. Then $\overline{\mathbb{Z}}$-points are Zariski dense in $M_B(X, 2)$.
\end{theorem}
\begin{proof}[Proof sketch]
 As integral points are dense in $\mathbb{C}^\times$, it suffices to prove density in the locus of local systems with trivial determinant.
	
	From the main result of \cite[Theorem 1.4]{previte2002topological}, taking the boundary to be empty, and the Zariski-density of the unitary locus, it is enough to produce one  $\pi_1(X)$-representation defined over $\overline{\mathbb{Z}}$ with a unitary Galois conjugate---the orbit of this representation under the mapping class group of $X$ will be Zariski dense. But now the tautological local system on any compact Shimura curve (or \'etale cover thereof) of genus $g$ suffices (and such exist for all $g\geq 2$, as one may take \'etale covers of a compact Shimura curve of genus $2$).
\end{proof}
The key idea here was to use the mapping class group action on the character variety to produce an abundance of integral points. This action on integral points has been studied in the beautiful papers \cite{whang2020arithmetic, whang2020global, whang2020nonlinear}.

Of course one may consider stronger forms of density, e.g.~density of integral points in the analytic topology on $M_B(X, r)$, or in $p$-adic topologies (i.e.~strong approximation questions). In particular, the example of certain Markoff surfaces, mentioned in \autoref{subsubsection:n-4}, has been studied by Bourgain-Gamburd-Sarnak \cite{bourgain2016markoff, bgs1} and \cite{chen2020nonabelian}.

\begin{remark}
Aside from their intrinsic interest and the application to conjectures on integrality of rigid local systems, Zariski-density of integral points in character varieties would have a number of useful applications, for example to the conjecture of Ekedahl-Shephard--Barron-Taylor for the isomonodromy foliation, discussed in \autoref{remark:esbt}.
\end{remark}

\subsection{Motivic subvarieties}\label{subsubsection:motivic-subvarieties} The conjectures discussed in \autoref{subsection:non-abelian-cohomology} attempt to intrinsically characterize  those local systems (i.e.~\emph{points} of $\mathscr{M}_B(X, r), \mathscr{M}_{dR}(X,r)$, and so on) of geometric origin. What about higher-dimensional subvarieties?

It is not clear to us what a precise definition of a motivic subvariety of e.g.~ $\mathscr{M}_B(X, r)$ should be, but it should, for example, be stable undering taking intersections;  for each $f: X\to Y$, the image of a motivic subvariety under the induced map $$\mathscr{M}_B(Y, r)\overset{f^*}{\to} \mathscr{M}_B(X,r);$$ should be motivic; for each smooth proper $g: Z\to X$, the image of $$W_{g,r}\overset{R^ig_*}{\longrightarrow} M_B(X,r),$$ where $W_{g,r}\subset M_B(Z,r')$ is the set of local systems $\mathbb{V}$ on $Z$ with $\on{rk}R^ig_* \mathbb{V}=r$; the $W_{g,r}$ themselves should be motivic; subloci of ${M}_B(X,r)$ consisting of local systems with monodromy contained in a conjugate of a fixed subgroup $G\subset \on{GL}_r(\mathbb{C})$, etc.

It would be of great interest to find a characterization of motivic subvarieties analogous to the conjectures in \autoref{subsection:non-abelian-cohomology}; a number of authors have proposed and studied such characterizations, notably \cite{budur2020absolute}, \cite{esnault2020arithmetic}. For example, a motivic subvariety of the character variety $M_B(X,r)$ should:
\begin{enumerate}
	\item Be defined over the ring of integers of a number field,
	\item Have its image in $M_{\text{Dol}}(X,r)$ under the comparison map of non-abelian Hodge theory be stable under the natural $\mathbb{C}^\times$-action on $M_{\text{Dol}}(X,r)$,
	\item Have its $\mathbb{Z}_\ell$-points stable under the absolute Galois group of some finitely-generated field to which $X$ descends,
\end{enumerate}
and so on. We will discuss a variant of the question of characterizing motivic subvarieties in \autoref{subsection:invariant-subvarieties}.
\section{Vector bundles and mapping class groups}\label{section:mcg}
We finally return to our fundamental example, that of the universal curve $$\mathscr{C}_{g,n}\to \mathscr{M}_{g,n}$$ over the moduli space $\mathscr{M}_{g,n}$ of curves of genus $g$ with $n$ marked points. The goal of this section is to record a number of conjectures on the (pure) mapping class group $\on{PMod}_{g,n}=\pi_1(\mathscr{M}_{g,n})$, its representations, and its action on $Y(g,n,r)=M_B(\Sigma_{g,n}, r)$, motivated by some of the considerations in previous chapters. Some of these are well-known to surface topologists; others are aimed at drawing relationships between algebro-geometric questions about Riemann surfaces and questions in surface topology.
\subsection{Superrigidity}
We begin with a classical conjecture of Ivanov:
\begin{conjecture}[{\cite[\S7]{ivanov2006fifteen}}]\label{conjecture:mod-ivanov}
	Let $g\geq 3$. Then any finite index subgroup $\Gamma$ of $\on{Mod}_{g,n}$ has finite abelianization, i.e.~$H^1(\Gamma, \mathbb{C})=0.$
\end{conjecture}
This conjecture is motivated in part by a well-known analogy between $\on{Mod}_{g,n}$ and lattices in simple Lie groups of higher rank (see e.g.~\cite{farb1998superrigidity, farb2009rigidity} for further considerations along these lines). All representations of such groups are in fact cohomologically rigid, by e.g.~Margulis super-rigidity \cite{zimmer2013ergodic}. It seems natural to conjecture the same is true for irreducible representations of $\on{Mod}_{g,n}$; in fact, we conjecture:
\begin{conjecture}\label{conjecture:mod-rigidity}
	Let $g\geq 3$. Any irreducible representation $\rho$ of $\on{Mod}_{g,n}$ is cohomologically rigid, i.e.~ $H^1(\on{Mod}_{g,n}, \on{ad}(\rho))=0$. 
\end{conjecture}  
\begin{remark}
A variant of this conjecture is considered by Simpson in \cite[p.~14]{simpson2004construction}, who attributes it to Hain and Looijenga. Note that he suggests that there exist rigid but not cohomologically rigid irreducible representations of $\on{Mod}_{g,n}$, whose construction he attributes to Hain and Looijenga; according to Hain \cite{hain-personal}, this was a miscommunication, and no such representations are known to exist. Indeed, Simpson suggests that one may construct such representations as sub-objects of tensor powers of the standard representation $\on{Mod}_{g,n}\to \on{Sp}_{2g}$; but any irreducible such sub-object is in fact cohomologically rigid.
\end{remark}
Note that \autoref{conjecture:mod-rigidity} implies \autoref{conjecture:mod-ivanov}. Indeed, suppose $\Gamma\subset \on{Mod}_{g,n}$ is a finite index subgroup with $H^1(\Gamma, \mathbb{C})\neq 0$. Then $\Gamma$ admits a surjection onto $\mathbb{Z}$, and hence a non-trivial family of rank one representations $$\rho_t: \Gamma\twoheadrightarrow \mathbb{Z}\to \mathbb{C}^\times,$$ sending $1\in \mathbb{Z}$ to $t\in \mathbb{C}^\times$. The induced representations $\on{Ind}_\Gamma^{\on{Mod}_{g,n}}\rho_t$ form a non-trivial family of semisimple representations of $\on{Mod}_{g,n}$, hence for generic $t$ there is some summand of $\rho_t$ which is not rigid, contradicting \autoref{conjecture:mod-rigidity}.
\begin{remark}
	Note that the analogy between $\on{Mod}_{g,n}$ and lattices in higher-rank Lie groups is imperfect. For the latter, \emph{all} representations are rigid. On the other hand $\on{Mod}_{g,n}$ admits non-rigid reducible representations. See e.g.~\cite[Example 10.1.6]{landesman2022canonical} for examples.
\end{remark}
This conjecture is motivated in part by an explanation of the ubiquity of ``hidden rigidity" in the arguments of \autoref{section:matrices} and \autoref{section:canonical}. 
And a positive answer would have a pleasant consequence towards the conjectures discussed in \autoref{subsection:non-abelian-cohomology}. 
\begin{proposition}\label{proposition:hodge-to-arithmetic-rigidity}
Let $X$ be a very general curve of genus $g\geq 3$. Let $\mathbb{V}$ be an irreducible local system on $X$ which underlies a $\mathbb{Z}$-variation of Hodge structure. Assume \autoref{conjecture:mod-rigidity}. Then $\mathbb{V}\otimes \mathbb{Z}_\ell$ is arithmetic, in the sense of \autoref{example:non-abelian-Tate}.
\end{proposition}
\begin{proof}
	The proof is closely related to that of \autoref{proposition:hodge-to-arithmetic}. As in the proof of that theorem, the hypotheses of \autoref{corollary:infinitesimal-criterion} are satisfied (taking $\mathscr{X}\to S$ to be $\mathscr{C}_g\to \mathscr{M}_g$). Hence the isomorphism class of $\mathbb{V}$ has finite orbit under $\on{Mod}_g=\pi_1(\mathscr{M}_g)$. By \cite[Lemma 2.2.2]{landesman2022canonical}, there exists a subgroup $\Gamma\subset \pi_1(\mathscr{C}_g)$ containing $\pi_1(X)$, and a projective representation of $\Gamma$ whose restriction to $\pi_1(X)$ has monodromy that of $\mathbb{P}\mathbb{V}$. By \autoref{conjecture:mod-rigidity} this projective local system is rigid, hence arithmetic.
	
	Thus $\mathbb{PV}$ is arithmetic. As it necessarily has finite determinant (as it is a $\mathbb{Z}$-VHS), $\mathbb{V}$ is itself arithmetic. 
\end{proof}
Moreover, \autoref{conjecture:mod-rigidity} would imply a different sort of classification of finite orbits of $\on{Mod}_{g,n}$ on the character varieties $Y(g,n,r)$ to \autoref{conjecture:precisification}, conditional on \autoref{conjecture:fontaine-mazur}, by an analogous argument:
\begin{conjecture}\label{conjecture:mcg-finite-geometric-origin}
Suppose $g\geq 3$. Let $$\rho: \pi_1(\Sigma_{g,n})\to \on{GL}_r(\mathbb{C})$$ be an irreducible representation whose conjugacy class has finite orbit under $\on{Mod}_{g,n}$. Then for \emph{any} complex structure on $\Sigma_{g,n}$, the local system associated to $\rho$ is of geometric origin.
\end{conjecture}
While the above conjecture is likely out of reach---we have few methods to prove that some abstract local system is of geometric origin---the following prediction might be approachable:
\begin{conjecture}\label{conjecture:finite-fixed-points}
	Let $g\geq 3$, and $\Gamma\subset\on{Mod}_{g,n}$ a finite index subgroup. Then $\Gamma$ acts on $Y(g,n,r)^{\text{irr}}$ with only finitely many fixed points. All such fixed points correspond to local systems defined over the ring of integers of some number field $\mathscr{O}_K$.
\end{conjecture}
Of course both conjectures above follow immediately in the regime where \autoref{theorem:canonical-reps} applies.
\begin{remark}
It may be instructive to compare the statement of \autoref{conjecture:mod-rigidity} to the classification of representations of $\on{Sp}_{2g}(\mathbb{Z}), g\geq 2,$ which follows from e.g.~superrigidity. Not only are such representations rigid---in fact, they all factor through a continuous representation of $\on{Sp}_{2g}(\mathbb{C})\times\on{Sp}_{2g}(\widehat{\mathbb{Z}})$ via the natural (diagonal) embedding $$\on{Sp}_{2g}(\mathbb{Z})\hookrightarrow \on{Sp}_{2g}(\mathbb{C})\times\on{Sp}_{2g}(\widehat{\mathbb{Z}}).$$ Any continuous complex representation of the factor $\on{Sp}_{2g}(\widehat{\mathbb{Z}})$ has finite image.

Now $\on{Sp}_{2g}(\mathbb{Z})$ is the fundamental group of the moduli stack $\mathscr{A}_g$ of principally polarized Abelian varieties of dimension $g$. Thus superrigidity, along with the classification of representations of $\on{Sp}_{2g}(\mathbb{C})$ (in particular, they are all summands of tensor powers of the standard representation), tells us that if $\mathbb{V}$ is any local system on $\mathscr{A}_g$, there exists a finite \'etale cover $\pi: \mathscr{A}'\to\mathscr{A}_g$, and some $n\geq 0$, such that $\pi^*\mathbb{V}\subset \pi^*\mathbb{W}_{\text{taut}}^{\otimes n}$, where $\mathbb{W}_{\text{taut}}$ is the tautological local sytem on $\mathscr{A}_{g}$.

How does this compare to \autoref{conjecture:mod-rigidity}? If one accepts in addition \autoref{conjecture:rigid-to-motivic}, \autoref{conjecture:mod-rigidity} tells us that irreducible local systems on $\mathscr{M}_{g,n}$ should be of geometric origin---in particular, they should be pulled back from a period domain.
\end{remark}

\subsection{The Putman-Wieland conjecture}
Most of our evidence for \autoref{conjecture:mod-rigidity} and \autoref{conjecture:mod-ivanov} comes from work towards a conjecture of Putman and Wieland, which is essentially equivalent to \autoref{conjecture:mod-ivanov}. We spend the next few sections on a digression about this conjecture and related ``big monodromy" conjectures, and their relationships to questions about vector bundles on curves, before returning to our ``non-abelian" study of the $\on{Mod}_{g,n}$-action on $Y(g,n,r)$ and attempting to make sense of the notion big monodromy there.

	Let $\Sigma_{g,n}$ be an orientable surface of genus $g$ with $n$ punctures. Fixing a base-point $x_0$ in $\Sigma_{g,n}$, there is a natural action of $\on{Mod}_{g,n+1}$ on $\pi_1(\Sigma_{g,n}, x_0)$; hence if $\Sigma_{g'}\to \Sigma_g$ is a cover branched at $n$ points, a finite index subgroup of $\on{Mod}_{g,n+1}$ naturally acts on $H_1(\Sigma_{g'}, \mathbb{Z})$. Indeed, let $\Sigma_{g',n'}$ be the complement of the ramification points in $\Sigma_{g'}$; then $\pi_1(\Sigma_{g',n'})$ is a subgroup of $\pi_1(\Sigma_{g,n})$, and hence admits a natural action by its stabilizer $\Gamma$ in $\on{Mod}_{g,n+1}$, a finite index subgroup. Hence $\Gamma$ acts naturally on $H_1(\Sigma_{g',n'}, \mathbb{Z})$, its abelianization, and one can check that this action descends to the quotient $H_1(\Sigma_{g'}, \mathbb{Z})$.
\begin{conjecture}[{\cite[Conjecture 1.2]{putman2013abelian}}] \label{conjecture:pw}
Suppose $g\geq 3$ and $n\geq 0$. Then there is no non-zero vector in $H_1(\Sigma_{g'}, \mathbb{Z})$ with finite $\Gamma$-orbit, under the action of $\Gamma$ on $H_1(\Sigma_{g'},\mathbb{Z})$ described above.
\end{conjecture}
We refer to \autoref{conjecture:pw} for fixed $(g,n)$ as $\on{PW}(g,n)$.
\begin{remark}
	Putman and Wieland show \cite[Theorem C]{putman2013abelian} that the conjecture $\on{PW}(g-1,n+1)$ implies \autoref{conjecture:mod-ivanov}  for $\on{Mod}_{g,n}$.
\end{remark}
\begin{remark}\label{remark:markovic-pw}
	Putman and Wieland originally made conjecture \autoref{conjecture:pw} for all $g\geq 2$. However Markovi\'c observed \cite{markovic2022unramified} that there are counterexamples in genus $2$, using a beautiful construction of Bogomolov and Tschinkel \cite{bogomolov2002unramified}.
\end{remark}
A number of cases of the Putman-Wieland conjecture have been verified by topological means---see e.g.~\cite{looijenga1997prym, grunewald2015arithmetic,  looijenga2021arithmetic, boggi2023generating} and the references therein. But much recent progress has been algebro-geometric in nature, as we now explain.

\subsubsection{Algebro-geometric evidence}
This conjecture has a simple algebro-geometric description. Let $\varphi: \Sigma_{g'}\to \Sigma_g$ be a finite cover branched at $n$ points, and let $\mathscr{M}_\varphi$ be the moduli stack of complex structures on $\Sigma_{g'}, \Sigma_g$ compatible with $\varphi$, so that we have a diagram
\begin{equation}\label{equation:versal-family}
\xymatrix{
\mathscr{X}\ar[r]^p \ar[rd]_\pi & \mathscr{C}\ar[d]^q\ar[r]   \ar@{}[dr] | {\square}& \mathscr{C}_{g,n}\ar[d]  \\
& \mathscr{M}_{\varphi}\ar[r] & \mathscr{M}_{g,n}
}\end{equation}
where for each $m\in \mathscr{M}_\varphi$, the fiber of $$p_m: \mathscr{X}_m\to \mathscr{C}_m$$ over $m$ is a holomorphic map from a genus $g'$ curve to a genus $g$ curve which is topologically the same as $\varphi$. Here the map $\mathscr{M}_\varphi\to \mathscr{M}_{g,n}$ classifies the family of curves $\mathscr{C}/\mathscr{M}_\varphi$, so the square on the right is Cartesian.

The Putman-Wieland conjecture may be rephrased as follows:
\begin{conjecture}[Algebro-geometric rephrasing of Putman-Wieland]\label{conjecture:pw-ag}
	Suppose $g\geq 3$. Then for any \'etale map $f: \mathscr{M}\to \mathscr{M}_\varphi$, $$H^0(\mathscr{M}, f^*R^1\pi_*\mathbb{C})=0.$$
\end{conjecture}
Another way to say this is: the relative Jacobian of $\pi$ has no isotrivial isogeny factor.

This reformulation may now be attacked via Hodge-theoretic methods, and in particular, via an analysis of the derivative of the period map associated to the local system $R^1\pi_*\mathbb{C}$. In fact, the following is more or less immediate from \autoref{theorem:vanishing}:
\begin{theorem}[{\cite[Corollary 7.2.3]{landesman2022canonical}}]\label{theorem:PW-big-g}
Let $H$ be a finite group such that each irreducible representation of $H$ has dimension less than $g$. Then the Putman-Wieland conjecture is true for any $H$-cover $\Sigma_{g'}\to \Sigma_g$.	
\end{theorem}
In fact, a more or less identical argument shows that the Putman-Wieland conjecture is true for covers of $\Sigma_g$ of degree $d<g$. See \cite{landesman2023introduction} for a leisurely introduction to the algebro-geometric aspects of the Putman-Wieland conjecture, and \cite{landesman2023applications} for some closely related results.

Related methods allow us to prove the Putman-Wieland conjecture in the large $n$, as opposed to large $g$, regime. 
\begin{theorem}[Special case of {\cite[Theorem 1.15]{landesman2024big}}]
	Let $\varphi: \Sigma_{g'}\to \Sigma_g$ be an $H$-cover branched over $n$ points, with $H$ a simple group. Let $r$ be the maximal dimension of an irreducible representation of $H$, and suppose $$n>\frac{3r^2}{\sqrt{g+1}}+8r.$$ Then the Putman-Wieland conjecture is true for $\varphi$.
\end{theorem}
Note that this result applies even if $g=0$!

See also \cite{markovic2024second, klukowski2024tangle} for some results towards the Putman-Wieland conjecture of an algebro/differential-geometric nature.
\subsubsection{Big monodromy}
It seems natural to conjecture that much more than the Putman-Wieland conjecture is true:
\begin{conjecture}\label{conjecture:big-monodromy}
	Let $H$ be a finite group, $g\geq 3$, and $\varphi: \Sigma_{g'}\to \Sigma_g$ a Galois $H$-cover. With notation as in diagram \eqref{equation:versal-family}:
	\begin{enumerate}
		\item  the identity component of the Zariski closure of the monodromy group of the local system $R^1\pi_*\mathbb{C}$ is the derived subgroup of the centralizer of $H$ in $\on{Sp}_{2g'}(\mathbb{C})$, and
		\item the monodromy of $R^1\pi_*\mathbb{C}$ is an arithmetic subgroup of its Zariski-closure.
	\end{enumerate}
\end{conjecture}
The monodromy representations associated to the local systems $R^1\pi_*\mathbb{C}$ are known as \emph{higher Prym representations} and have been studied in special cases for a long time; for example, if $\varphi$ is an unramified $\mathbb{Z}/2\mathbb{Z}$-cover, they correspond precisely to classical Prym varieties.
\begin{remark}
It is not hard to see that the algebraic group identified in \autoref{conjecture:big-monodromy}(1) is the largest possible---it must preserve the symplectic form on the cohomology of our family of curves, commute with the $H$-action, and it must be semisimple (this last is always the case for local systems of geometric origin). Thus \autoref{conjecture:big-monodromy} is an instance of a standard slogan in algebraic geometry:
\begin{slogan}\label{slogan:big-monodromy}
	 Monodromy groups should be as big as possible.	
\end{slogan}
\end{remark}

There is a fair amount of evidence for \autoref{conjecture:big-monodromy}. For example, the papers \cite{looijenga1997prym, grunewald2015arithmetic, looijenga2021arithmetic} all prove special cases under various topological hypotheses. The main purpose of the paper \cite{landesman2024big} is to prove \autoref{conjecture:big-monodromy}(1) for $g$ large:
\begin{theorem}\label{theorem:big-monodromy}
	Let $\varphi: \Sigma_{g'}\to \Sigma_g$ be a Galois $H$-cover branched at $n$ points. Let $r$ be the maximum dimension of an irreducible complex representation of $H$. If either
	\begin{enumerate}
	\item $n=0$ and $g\geq 2r+2$, or 
	\item $n>0$ and $g>\max(2r+1, r^2)$,
	\end{enumerate}
then \autoref{conjecture:big-monodromy}(1) holds for $\varphi$.
\end{theorem}
\subsection{Big monodromy and Riemann-Hilbert problems}\label{subsection:big-monodromy-RH} We now explain some of the ideas that go into the proof of \autoref{theorem:big-monodromy} and \autoref{theorem:PW-big-g}, and some of the conjectures they inspire.
\subsubsection{Vector bundles and monodromy}\label{subsubsection:vb-and-monodromy} We consider the following situation. Let $$q: \mathscr{C}\to \mathscr{M}$$ be a family of $n$-punctured curves of genus $g$ over a smooth base $\mathscr{M}$, with the associated map $\mathscr{M}\to \mathscr{M}_{g,n}$ dominant \'etale. Let $\mathbb{U}$ be a local system on $\mathscr{C}$. We would like to understand the monodromy of the local system $R^1q_*\mathbb{U}$. For example, one might take $\mathbb{U}=p_*\mathbb{C}$ in the notation of \eqref{equation:versal-family}, in which case understanding this monodromy (and in particular, the monodromy on the weight one part $W^1R^1q_*\mathbb{U}$) amounts to \autoref{conjecture:big-monodromy}.

If $\mathbb{U}$ carries a complex variation of Hodge structure (say of weight zero)---and in particular if $\mathbb{U}$ is unitary---we can study this monodromy via the derivative of the period map associated to the complex variation of Hodge structure $W^1R^1q_*\mathbb{U}$, as we now explain. Explicitly, if $q$ factors as $$\xymatrix{ 
\mathscr{C} \ar@{^(->}[r]^\iota \ar[rd]_q & \overline{\mathscr{C}} \ar[d]^{\overline q} \\
& \mathscr{M}
}$$
with $\overline{q}$ a smooth relative compactification of $q$, then $$W^1R^1q_*\mathbb{U}=R^1\overline{q}_*\iota_*\mathbb{U}.$$

For the rest of this section we assume $q$ is itself proper, so $W^1R^1q_*\mathbb{U}=R^1q_*\mathbb{U}$, for notational simplicity. We also assume that $\mathbb{U}$ is unitary. Fix $m\in \mathscr{M}$ a point, and set $C=\mathscr{C}_m$. Let $(\mathscr{E},\nabla)=(\mathbb{U}|_C\otimes \mathscr{O}_C, \on{id}\otimes d)$ be the flat bundle on $C$ associated to $\mathbb{U}|_C$ by the Riemann-Hilbert correspondence. The Hodge filtration on $(R^1q_*\mathbb{U})_m=H^1(C, \mathbb{U}|_C)$ has terms given by $$F^1H^1(C, \mathbb{U}|_C)=H^0(C, \mathscr{E}\otimes\Omega^1_C), H^1(C, \mathbb{U}|_C)/F^1=H^1(C, \mathscr{E}).$$

Thus the derivative of the period map associated to the variation of Hodge structure on $R^1q_*\mathbb{U}$ at $m$ is a map $$dP_m: T_m\mathscr{M}\to \on{Hom}(H^0(C, \mathscr{E}\otimes\omega_C), H^1(C, \mathscr{E}))$$ which can be made explicit as follows. As the classifying map $\mathscr{M}\to\mathscr{M}_{g,n}$ is assumed to be \'etale, we may identify $T^*_m\mathscr{M}$ with the space of quadratic differentials on $C$, namely $H^0(C, \omega_C^{\otimes 2})$. By Serre duality, $H^1(C, \mathscr{E})$ is dual to $H^0(C, \mathscr{E}^\vee\otimes \omega_C)$. Hence $dP_m$ is adjoint to a map $$H^0(C, \mathscr{E}\otimes \omega_C)\otimes H^0(C, \mathscr{E}^\vee\otimes \omega_C)\to H^0(C,\omega_C^{\otimes 2}),$$ namely the natural pairing induced by taking sections $s_1\in H^0(C, \mathscr{E}\otimes \omega_C), s_2\in H^0(C, \mathscr{E}^\vee\otimes \omega_C)$ to $$s_1\otimes s_2\in H^0(\mathscr{E}\otimes \mathscr{E}^\vee\otimes \omega_C^{\otimes 2}),$$ and then composing with the map to $H^0(C, \omega_C^{\otimes 2})$ induced by the natural (trace) pairing $$\mathscr{E}\otimes \mathscr{E}^\vee\otimes \omega_C^{\otimes 2}\to  \omega_C^{\otimes 2}.$$ (See \cite[Appendix A]{landesman2022canonical} for a proof.)

Suppose that the monodromy representation of $\pi_1(\mathscr{M},m)$ on $H^1(C, \mathbb{U}|_C)$ admits a non-zero invariant vector, or in other words, that $R^1q_*\mathbb{U}$ admits a constant sub-variation of Hodge structure (here we are using the Theorem of the Fixed Part). Then, possibly after replacing $\mathbb{U}$ with its complex conjugate $\overline{\mathbb{U}}$, the map $$H^0(C, \mathscr{E}\otimes \omega_C)\to \on{Hom}(H^0(C, \mathscr{E}^\vee\otimes \omega_C), H^0(C, \omega_C^{\otimes 2}))$$ will have non-zero kernel. Put another way, there exists a section $\eta\in H^0(C, \mathscr{E}\otimes \omega_C)$ such that the induced (non-zero!) map $$-\cup \eta: \mathscr{E}^\vee\otimes \omega_C\to \omega_C^{\otimes 2}$$ induces the zero map on global sections. In particular $\mathscr{E}^\vee\otimes \omega_C$ is not generated by global sections at the generic point of $C$.

Thus we have shown:
\begin{proposition}\label{proposition:ggg-vanishing}
	Suppose that for a general fiber $C$ of $q$ the vector bundle $(\mathbb{U}|_C\otimes \mathscr{O}_C)^\vee\otimes \omega_C$ is generically generated by global sections. Then  for all $f: \mathscr{M}'\to \mathscr{M}$ dominant \'etale, $$H^0(\mathscr{M}, f^*R^1q_*\mathbb{U})=0.$$
\end{proposition}
In fact this is more or less the idea of the proof of \autoref{theorem:vanishing} and \autoref{theorem:PW-big-g}. It may be instructive to compare the statement to that of \autoref{conjecture:pw-ag}.
\subsubsection{Generic global generation}
We have just seen that the global generation properties of vector bundles on curves impact the monodromy of certain variations of Hodge structure on their moduli. We make this precise as follows:
\begin{conjecture}[ggg conjecture]\label{conjecture:ggg}
	Let $g\geq 3$, and let $$\rho: \pi_1(\Sigma_{g,n})\to U(r)$$ be a unitary representation. Fix a generic complex structure $X$ on $\Sigma_{g,n}$, with $\overline{X}$ the corresponding compact Riemann surface, and $D=\overline{X}\setminus X$, and let $\mathscr{E}_\star$ be the corresponding parabolic bundle. Then $\widehat{\mathscr{E}}_0\otimes \omega_{\overline{X}}(D)$ is generically generated by global sections.
\end{conjecture}
We refer to this conjecture as the ggg conjecture (for generically globally generated).

Here $\mathscr{E}_\star$ is the parabolic bundle associated to $\rho$ under the Mehta-Seshadri correspondence \cite{mehta-seshadri}. See e.g.~\cite[\S3]{landesman2024big} for the notation on parabolic bundles we are using. The following is the special case where $n=0$:
\begin{conjecture}
	Let $g\geq 3$, and let $$\rho: \pi_1(\Sigma_{g})\to U(r)$$ be a unitary representation. Fix a generic complex structure $X$ on $\Sigma_{g}$, and let $(\mathscr{E}, \nabla)$ be the corresponding flat bundle. Then $\mathscr{E}\otimes \omega_{X}$ is generically generated by global sections.
\end{conjecture}

Taking $\rho$ to have finite monodromy, these conjectures imply the Putman-Wieland conjecture (\autoref{conjecture:pw}), by \autoref{proposition:ggg-vanishing}. More generally, they would give some evidence for \autoref{conjecture:mod-rigidity}:
\begin{proposition}
	Suppose $g\geq 3$ and assume \autoref{conjecture:ggg}. Let $$\rho: \on{PMod}_{g,n+1}\to U(r)$$ be a representation whose restriction to the point-pushing subgroup $\pi_1(\Sigma_{g,n})\subset \on{PMod}_{g,n+1}$ is irreducible.  Then $\rho$ is cohomologically rigid.
\end{proposition}
\begin{proof}[Proof sketch, assuming $n=0$]
	Let $\mathbb{U}$ be the local system on $\mathscr{M}_{g,n+1}$ corresponding to $\rho$. It is not hard to see that it has finite determinant. Let $\pi: \mathscr{M}_{g,n+1}\to \mathscr{M}_g$ be the forgetful map; by assumption, the restriction of $\mathbb{U}$ to a fiber of $\pi$ is irreducible, hence $\pi_*\on{ad}(\mathbb{U})=0$. Thus $$H^1(\on{Mod}_{g,1}, \on{ad}(\mathbb{U}))=H^0(\mathscr{M}_{g}, R^1\pi_*\on{ad}(\mathbb{U})),$$ which vanishes by \autoref{proposition:ggg-vanishing}.
\end{proof}
It would be very interesting to formulate analogues of \autoref{conjecture:ggg} for non-unitary local systems. The unitary case is already of some interest, though; for example, the above rigidity result would provide some evidence for the famous conjecture that mapping class groups have Kazhdan's property T, which in particular implies that unitary representations are rigid.
\subsubsection{Global generation and big monodromy}
\autoref{proposition:ggg-vanishing} shows that global generation properties of the vector bundles under consideration are related to cohomological vanishing of the sort considered in the Putman-Wieland conjecture, \autoref{conjecture:pw-ag}. What about big monodromy, of the form considered in \autoref{theorem:big-monodromy}? We first explain an algebro-geometric variant of \autoref{theorem:big-monodromy}, which follows from the proof of that theorem (see \cite{landesman2024big}):
\begin{theorem}\label{theorem:big-monodromy-local-system}
	Notation as in \autoref{subsubsection:vb-and-monodromy}, with $q$ proper; let $C$ be a general fiber of $q$, of genus $g\geq 4$. Suppose that $\mathbb{U}$ has finite monodromy, the monodromy representation $\rho$ of $\mathbb{U}|_C$ is irreducible, and $\mathbb{U}|_C\otimes \omega_C, \mathbb{U}^\vee|_C\otimes \omega_C$ are globally generated. Let $G$ be the Zariski-closure of the image of the monodromy representation $$\pi_1(\mathscr{M})\to GL(H^1(C, \mathbb{U}|_C)).$$ 
	\begin{enumerate}
	\item If $\rho$ is orthogonally self-dual, then $G=\on{Sp}(H^1(C, \mathbb{U}|_C))$.
	\item If $\rho$ is symplectically self-dual, then $G=\on{SO}(H^1(C, \mathbb{U}|_C))$.
	\item If $\rho$ is not self-dual, then $G=\on{SL}(H^1(C, \mathbb{U}|_C))\times H$ for some finite group of scalars $H$. 
	\end{enumerate}
\end{theorem}
\begin{remark}
	A self-dual irreducible representation $\rho$ of a group $G$ has $\dim (\rho\otimes\rho)^G=1$. As $\rho\otimes\rho=\on{Sym}^2\rho\oplus \wedge^2\rho,$ we thus have that either $\dim (\on{Sym}^2\rho)^G=1$ or $\dim (\wedge^2\rho)^G=1$. In the former case, we say $\rho$ is orthogonally self-dual, and in the latter we say it is symplectically self-dual. As the cup product is alternating, the monodromy representations $H^1(C, \mathbb{U}|_C)$ appearing in \autoref{theorem:big-monodromy-local-system} are, by Poincar\'e duality, symplectically self-dual if $\rho$ is orthogonally self-dual, and orthogonally self-dual if $\rho$ is symplectically self-dual. The groups $\on{Sp}, \on{SO}$ appearing in the statement of \autoref{theorem:big-monodromy-local-system} are the group of automorphisms of $H^1(C, \mathbb{U}|_C)$ with trivial determinant that preserve the Poincar\'e duality pairing.
\end{remark}

The proofs of \autoref{theorem:big-monodromy}, and its generalization to non-proper curves (see \cite[Theorems 1.3 and 1.9]{landesman2024big}), rely on this result;  they proceed by verifying the global generation hypothesis. We briefly sketch the approach, and then make some related conjectures on global generation.

The key idea is that the global generation assumption allows us to functorially recover $\rho$ from the derivative of the period map associated to $R^1q_*\mathbb{U}$. Indeed, we have:
\begin{proposition}[Twisted generic Torelli, proof of {\cite[Theorem 6.2]{landesman2024big}}]\label{proposition:twisted-Torelli}
	Let $C$ be a smooth proper curve and $\mathbb{U}$ a unitary local system on $C$, with $\mathscr{E}=\mathbb{U}\otimes \mathscr{O}_C$ the associated vector bundle. Suppose that $\mathscr{E}\otimes \omega_C, \mathscr{E}^\vee\otimes\omega_C$ are globally generated. Then the image of the composition $$H^0(C, \mathscr{E}\otimes \omega_C)\otimes \mathscr{O}_C\overset{dP\otimes \on{id}}{\longrightarrow} H^1(C, \mathscr{E})\otimes H^0(C, \omega_C^{\otimes 2})\otimes \mathscr{O}_C\overset{\on{id}\otimes \on{ev}}{\longrightarrow} H^1(C, \mathscr{E})\otimes \omega_C^{\otimes 2}$$ is canonically isomorphic to $\mathscr{E}\otimes \omega_C$, where here $dP$ is adjoint to the derivative of the period map discussed in \autoref{subsubsection:vb-and-monodromy}.
\end{proposition}
We view this as a ``generic Torelli theorem with coefficients" --- compare e.g.~to the proof of the generic Torelli theorem in \cite{harris1985introduction}, or the proof of generic Torelli for hypersurfaces \cite{donagi1983generic}.

We now sketch the proof of \autoref{theorem:big-monodromy-local-system}, taking \autoref{proposition:twisted-Torelli} as input.
\begin{proof}[Sketch of proof of \autoref{theorem:big-monodromy-local-system}]
	We begin by showing that the complex variation of Hodge structure $R^1q_*\mathbb{U}$ appearing in \autoref{theorem:big-monodromy-local-system} is irreducible.
	
	By \cite[Proposition 4.9]{landesman2024big}, the vector bundle $\mathscr{E}=\mathbb{U}|_C\otimes \mathscr{O}_C$ satisfies $\mathscr{E}\otimes \omega_C, \mathscr{E}^\vee\otimes\omega_C$ are globally generated, for a fiber $C$ of $q$ over a general point $m\in \mathscr{M}$.\footnote{The proof is a somewhat involved deformation theory argument, about which we say nothing more here.} Hence by our twisted Torelli theorem \autoref{proposition:twisted-Torelli}, $\mathscr{E}\otimes \omega_C$, and hence $\mathscr{E}$ can be functorially recovered from the infinitesimal variation of Hodge structure associated to $R^1q_*\mathbb{U}$ at $m$. 
	
	In particular, if the complex variation of Hodge structure $R^1q_*\mathbb{U}$ is a non-trivial direct sum, the same is true for $\mathscr{E}$. But the Narasimhan-Seshadri correspondence \cite{narasimhan1965stable} implies that $\mathscr{E}$ is stable, hence irreducible.
	
	A slight enhancement of this argument shows that the identity component of the Zariski-closure of the monodromy group of $R^1q_*\mathbb{U}$ is in fact a simple group, acting irreducibly \cite[Theorem 6.7]{landesman2024big}. Now \cite[Theorem 0.5.1(b)]{zarkhin1985weights} (which Zarhin attributes to Deligne) implies that this group is either $\on{SO}, \on{Sp},$ or $\on{SL}$, acting via a minuscule representation. We rule out the non-standard minuscule representations using a further analysis of the infinitesimal variation of Hodge structure associated to $R^1q_*\mathbb{U}$, this time along so-called Schiffer variations \cite[\S7]{landesman2024big}.
\end{proof}
We make the following optimistic strengthening of \autoref{conjecture:ggg}:
\begin{conjecture}\label{conjecture:gg}
	Let $g\geq 3$, and let $$\rho: \pi_1(\Sigma_{g})\to U(r)$$ be a unitary representation. Fix a generic complex structure $X$ on $\Sigma_{g}$, and let $(\mathscr{E}, \nabla)$ be the corresponding flat bundle. Then $\mathscr{E}\otimes \omega_{X}$ is generated by global sections.
\end{conjecture}
By the proof of \autoref{theorem:big-monodromy-local-system} (and its strengthening \autoref{theorem:big-monodromy}), we have:
\begin{proposition}
	Assume \autoref{conjecture:gg}. Then \autoref{conjecture:big-monodromy}(1) is true for unramified covers $\Sigma_{g'}\to \Sigma_g$, with $g\geq 4$.
\end{proposition}
It would be of some interest to find (and prove) an analogue of \autoref{conjecture:gg} (for parabolic bundles, along the lines of \autoref{conjecture:ggg}) which would imply \autoref{conjecture:big-monodromy}(1) in full generality, even for ramified covers.

See also \cite[\S10]{landesman2024big} for some further discussion of these and related questions.
\subsubsection{Riemann-Hilbert problems}\label{subsubsection:Riemann-Hilbert}
There is another point of view on \autoref{conjecture:ggg} and its variants, that has a long history---that of \emph{Riemann-Hilbert problems}.  That is, given a monodromy representation with corresponding flat bundle $(\mathscr{E},\nabla)$, what can one say about the properties of the vector bundle $\mathscr{E}$? For example, Hilbert's 21st problem \cite{hilbert1900mathematical} asked when a given representation $\rho$ of $\pi_1(\mathbb{CP}^1\setminus\{x_1, \cdots, x_n\})$ can be obtained as the monodromy of a Fuchsian ODE, or in modern terms, when there exists a connection on the \emph{trivial} bundle $\mathscr{O}_{\mathbb{CP}^1}^r$  with regular singularities at $x_1, \cdots, x_n$ and monodromy $\rho$ (see \autoref{example:mic-p1}). The analogous question in higher genus (where instead of asking that $\mathscr{E}$ be trivial, we ask that it be semistable) was answered by Esnault-Viehweg and Gabber \cite{esnault1998semistable}.

Here we consider a variant of this type of question, where one fixes a flat bundle with regular singularities $(\mathscr{E}, \nabla)$ on a marked Riemann surface $(X,D)$, and then ask how $\mathscr{E}$ behaves under isomonodromic deformation, i.e.~when one perturbs the complex structure on $(X,D)$. We have already seen a version of these questions in \autoref{theorem:stability}. The general expectation is that $\mathscr{E}$ should behave ``as generically as possible" after a general isomonodromic deformation, but this is in fact not always the case.
\begin{theorem}{{\cite[Theorem 1.3.2]{landesman2022geometric}}}
	Let $X$ be a smooth proper curve of genus at least $2$. There exist flat vector bundles $(\mathscr{E},\nabla)$ on $X$ with irreducible monodromy such that no isomonodromic deformation has semistable underlying vector bundle.
\end{theorem}
\begin{remark}
Note that this theorem contradicts some other results in the literature, e.g.~the main results of \cite{biswas2016isomonodromic, biswas2020isomonodromic, biswas2021isomonodromic}. See \cite[Remark 5.1.9]{landesman2022geometric} for a discussion. 
\end{remark}

In fact the examples arise from the Kodaira-Parshin trick, discussed in \autoref{example:kp-trick}; the point is that, as we have seen before, non-unitary variations of Hodge structure never have semistable underlying vector bundle. Nonetheless, the following seems plausible (as stability is a generic property):
\begin{conjecture}\label{conjecture:unitary-stability}
	Let $(X,D)$ be a marked curve of genus $g$ at least $2$, and $(\mathscr{E},\nabla)$ a flat bundle on $X$ with regular singularities along $D$, and irreducible unitary monodromy. Then after a general isomonodromic deformation, $\mathscr{E}$ is (semi)stable.
\end{conjecture}
\autoref{conjecture:unitary-stability} is immediate when $D$ is empty, by the Narasimhan-Seshadri correspondence \cite{narasimhan1965stable}, and the semistable case follows when $g$ is large compared to the rank of $\mathscr{E}$, from \autoref{theorem:stability}. In \cite[Theorem 1.3.4]{landesman2022geometric}, Landesman and I prove that general isomonodromic deformations of flat bundles are in general not ``too far" from being semistable---that is, we bound their Harder-Narasimhan polygon.

It is also natural to expect that the vector bundles underlying isomonodromic deformations of a fixed non-trivial irreducible unitary connection behave generically in a \emph{cohomological} sense. For example, when $D$ is empty, such $\mathscr{E}$ is stable of slope zero, and so for $\mathscr{L}$ a line bundle of degree $1$ on $X$, one expects from the Riemann-Roch theorem that $$H^0(X, \mathscr{E}\otimes \mathscr{L}^{\otimes s})=0$$ for $s\leq g-1$. For example, it seems reasonable to conjecture:
\begin{conjecture}[Horizontal generic vanishing] \label{conjecture:horizontal-generic-vanishing}
	Let $g\geq 3$, and let $(X,D)$ be a marked smooth projective curve of genus $g\geq 3$. Let $(\mathscr{E}, \nabla)$ be a flat bundle on $X$ with irreducible, non-trivial unitary monodromy, and regular singularities along $D$, whose residue matrices have eigenvalues with real parts in $[0,1)$.\footnote{These are the bundles appearing in the Mehta-Seshadri correspondence \cite{mehta-seshadri}; they are also known as \emph{Deligne canonical extensions}.} There exists some non-decreasing function $f$, with $f(3)=1$, such that after  isomonodromic deformation to a general nearby curve $X'$, we have:  
	\begin{enumerate}
		\item (weak form) $H^0(X', \mathscr{E}(Z))=0$ for a general effective divisor $Z$ on $X'$ of degree $d\leq f(g)$.
		\item (strong form) $H^0(X', \mathscr{E}(Z))=0$ for all effective divisors $Z$ on $X'$ of degree $d\leq f(g)$.
	\end{enumerate}
\end{conjecture}
This conjecture is closely related to \autoref{conjecture:ggg}. For simplicity of notation we assume $D=\emptyset$. By considering the short exact sequence $$0\to \mathscr{E}^\vee(-p)\otimes \omega_X\to \mathscr{E}^\vee\otimes \omega_X\to \mathscr{E}^\vee\otimes \omega_X|_p\to 0,$$ and using that $H^1(\mathscr{E}^\vee\otimes \omega_X)=0$ for $\mathscr{E}$ as in \autoref{conjecture:horizontal-generic-vanishing}, we see that $\mathscr{E}^\vee\otimes \omega_X$ is generated by global sections at $p$ if and only if $H^1(X, \mathscr{E}^\vee(-p)\otimes \omega_X)=0$, or equivalently, by Serre duality, if $H^0(X, \mathscr{E}(p))=0$. Thus when $D=\emptyset$, \autoref{conjecture:horizontal-generic-vanishing}(1) implies \autoref{conjecture:ggg}, and \autoref{conjecture:horizontal-generic-vanishing}(2) implies \autoref{conjecture:gg}. (In fact \autoref{conjecture:horizontal-generic-vanishing} implies \autoref{conjecture:ggg} even for $D$ non-empty, but we omit the proof to shield the reader from more parabolic bundle notation.)

Our primary evidence for \autoref{conjecture:horizontal-generic-vanishing} comes from the proof of \cite[Proposition 4.9]{landesman2024big}, which one can use to prove \autoref{conjecture:horizontal-generic-vanishing} when $g$ is large compared to the rank of $\mathscr{E}$. It would be useful for someone to extract the precise function $f(g,r)$ one can obtain from the proof of that result, and to write a careful proof of the implication. The statement given there shows that \autoref{conjecture:horizontal-generic-vanishing}(2) is true when $g\geq 2+2\on{rk}(\mathscr{E})$ and $f(g)=1$.

We regard \autoref{conjecture:horizontal-generic-vanishing} (and in particular the special case discussed in the previous paragraph) as an analogue of Green-Lazarsfeld's generic vanishing theorem \cite{green1987deformation} and its variants (see e.g.~\cite{arapura1997geometry, pareschi2011gv} for some analogues in higher rank). While those theorems analyze the cohomological behavior of generic flat bundles on a fixed variety, \autoref{conjecture:horizontal-generic-vanishing} aims to understand the cohomological behavior of flat bundles with fixed monodromy on a variety (curve) with general complex structure.
\subsubsection{Prill's problem}
We briefly remark on the connection between \autoref{conjecture:horizontal-generic-vanishing} and another question in classical algebraic geometry: Prill's problem \cite[p.~268, Chapter VI, Exercise D]{acgh}.
\begin{question}[Prill's problem]
Let $X,Y$ be  smooth projective curves of genus at least $2$ over the complex numbers, and let $f: Y\to X$ be a non-constant morphism. Can a general fiber of $f$ move in a pencil? That is, can $$\dim H^0(Y, \mathscr{O}_Y(f^{-1}(x)))\geq 2$$ for general $x\in X$?
\end{question}
The answer, we believe, was expected to be ``no." Indeed, by the Riemann-Hurwitz formula we have $$\deg f^{-1}(y)\leq g(Y);$$ and a generic effective divisor on a curve of genus $g$, of degree at most $g$ does not move in a pencil.

By the projection formula we have $f_*\mathscr{O}_Y(f^{-1}(x))=(f_*\mathscr{O}_Y)(x)$. Hence setting $\mathscr{E}=f_*\mathscr{O}_Y/\mathscr{O}_X$, we have that $\dim H^0(Y, \mathscr{O}_Y(f^{-1}(x)))\geq 2$ if and only if $$\dim H^0(X, \mathscr{E}(x))\neq 0.$$ As $\mathscr{E}$ carries a unitary flat connection with regular singularities at the branch points of $f$ and residues with real parts in $[0,1)$ (with monodromy the non-trivial summand of the permutation representation associated to the deck transformation group of $f$), we expect this group to be zero for generic $X, x$, by \autoref{conjecture:horizontal-generic-vanishing}(1), when the genus of $X$ is at least $3$.

In fact this bound on the genus is necessary---Landesman and I observed in \cite{landesman2024prill} that \emph{every} genus $2$ curve $X$ admits a finite degree $36$ \'etale cover $f:Y\to X$ such that every fiber of $f$ moves in a pencil---thus the answer to Prill's problem is ``yes" when the genus of $X$ is $2$. In fact the construction is the same as that used by Markovi\'c  in his disproof of the Putman-Wieland conjecture in genus $2$ \cite{markovic2022unramified} (see \autoref{remark:markovic-pw}), due to Bogomolov and Tschinkel \cite{bogomolov2002unramified}. In \cite[Lemma 5.5]{landesman2023applications}, we observe that \emph{any} counterexample $\varphi: \Sigma_{g'}\to \Sigma_g$ to the Putman-Wieland conjecture, \autoref{conjecture:pw}, gives rise to an example of a cover of a \emph{general curve of genus $g$} such that Prill's problem has a positive answer.

\subsection{Non-abelian big monodromy}
Having digressed somewhat into abelian questions (e.g.~the monodromy of certain local systems), we finally return to where we started: the action of the mapping class group and its subgroups on the space of conjugacy classes of rank $r$ representations of $\pi_1(\Sigma_{g,n,r})$, namely $Y(g,n,r)$. 

As in \autoref{subsection:non-abelian-cohomology}, we now view this action as the analogue of the monodromy representation associated to $R^1\pi_*\mathbb{C}$, where $$\pi: \mathscr{C}_{g,n}\to \mathscr{M}_{g,n}$$ is the map from the universal $n$-punctured curve of genus $g$ to the moduli space of genus $g$ curves with $n$ marked points. (Indeed, the fiber of $R^1\pi_*\mathbb{C}^\times$ is precisely $Y(g,n,1)$, and the monodromy action on this fiber identifies with the natural mapping class group action on $Y(g,n,1)$.) More generally, given a family of smooth $n$-punctured curves of genus $g$, $p: \mathscr{C}\to \mathscr{M}$, we obtain an action of $\pi_1(\mathscr{M})$ on $Y(g,n,r)$. Explicitly, $p$ induces a classifying map $\mathscr{M}\to \mathscr{M}_{g,n}$, and the induced action of $\pi_1(\mathscr{M})$ on $Y(g,n,r)$ is given by the composition $$\pi_1(\mathscr{M})\to \pi_1(\mathscr{M}_{g,n})\simeq\on{PMod}_{g,n}\to \on{Aut}(Y(g,n,r)).$$
If we think of this action as a ``non-abelian" monodromy representation, we are naturally led to try to understand the extent to which \autoref{slogan:big-monodromy} (that monodromy groups should be as big as possible) holds in this case. In this section, we make a modest attempt at making sense of this slogan in the non-abelian setting.

For $g\geq 1$, one sense in which the local system $R^1\pi_*\mathbb{C}^\times$ has big monodromy is that for any dominant map $f: \mathscr{M}\to \mathscr{M}_{g,n}$, the group $$H^0(\mathscr{M}, f^*R^1\pi_*\mathbb{C}^\times)$$ is finite. In other words, there are only finitely many points of $Y(g,n,1)$ fixed by $\pi_1(\mathscr{M})$. \autoref{theorem:canonical-reps} tells us that the same is true in general for the action of $\pi_1(\mathscr{M})$ on $Y(g,n,r)$, when $g\geq r^2$, so \autoref{theorem:canonical-reps} may be thought of as a ``non-abelian big monodromy" statement. And \autoref{conjecture:finite-fixed-points} predicts the analogous statement holds true whenever $g\geq 3$.

There are a number of other ways one might make sense of non-abelian big monodromy. Goldman and others have studied the \emph{ergodicity} of the $\on{Mod}_{g,n}$-action on the space of \emph{unitary} representations of $\pi_1(\Sigma_{g,n})$ in a series of beautiful papers, for example \cite{goldman1997ergodic, goldmanmapping, pickrell2002ergodicity, goldman2021mapping}, and this action was studied from the view of topological density in a number of other papers, e.g.~ \cite{previte2000topological, previte2002topological, cantat-loray}. Closer to our point of view is the work of Katzarkov, Pantev, and Simpson \cite{katzarkov2003density}, who give two possible interpretations of big monodromy in this setting:
\begin{enumerate}
\item No invariant function (NIF): there are no meromorphic functions on $M_B(\Sigma_{g,n}, r)$ invariant under the action of $\pi_1(\mathscr{M})$.
\item Big orbit (BO): There exists a point of $M_B(\Sigma_{g,n}, r)$ whose orbit under $\pi_1(\mathscr{M})$ is dense in the Zariski topology on $M_B(\Sigma_{g,n}, r)$.
\end{enumerate}
They show that both notions hold true for the $\on{Mod}_{g}$-action on $M_B(\Sigma_{g}, r)$ for $g$ large and $r$ odd \cite[Theorem A]{katzarkov2003density} and (again for $r$ odd) for the $\pi_1(B)$-action on $M_B(\Sigma_{g}, r)$ when $B$ is the base of a Lefschetz pencil of sufficiently high degree in a fixed algebraic surface \cite[Theorem B]{katzarkov2003density}.

They conjecture the following:
\begin{conjecture}[{\cite[Conjecture 4.2]{katzarkov2003density}}]\label{conjecture:katzarkov-nabm}
	Let $\mathscr{C}\to \mathscr{M}$ be any non-isotrivial smooth proper family of genus $g\geq 2$ curves. Then the induced $\pi_1(\mathscr{M})$-action on $M_B(\Sigma_g, r)$ satisfies NIF and BO.
\end{conjecture}
As far as we know, little progress has been made on this conjecture since \cite{katzarkov2003density}.
\subsubsection{Invariant subvarieties}\label{subsection:invariant-subvarieties}
In \autoref{section:matrices}, \autoref{section:canonical}, and \autoref{section:p-curvature}, we studied the finite orbits of $\pi_1(\mathscr{M})$ on $Y(g,n,r)$. What about higher-dimensional invariant subvarieties? A natural (imprecise) expectation in the case $\mathscr{M}=\mathscr{M}_{g,n}$, analogous to \autoref{conjecture:mcg-finite-geometric-origin}, is that for $g\geq 3$, any such subvariety should all be motivic, in the sense of \autoref{subsubsection:motivic-subvarieties}. 
\begin{question}[Imprecise]
	Let $Z\subset Y(g,n,r)$ be a maximal irreducible subvariety stable under the action of a finite index subgroup of $\on{Mod}_{g,n}$. Then is $Z$ ``of geometric origin" for any complex structure on $\Sigma_{g,n}$?
\end{question}
The primary evidence we have for a positive answer is the theorem of Corlette-Simpson \cite{corlette-simpson} on rank $2$ local systems, on which we heavily relied in \autoref{section:matrices}. We recall it now:
\begin{theorem}[{\cite[Theorem 1]{corlette-simpson}, \cite[Theorem A]{loray-etc}}] \label{theorem:corlette-simpson}
	Let $X$ be a smooth quasi-projective complex variety, and $$\rho: \pi_1(X)\to \on{SL}_2(\mathbb{C})$$ a Zariski-dense representation. Either $\rho$ is rigid and of geometric origin, or there exists a map $f: X\to C$, for $C$ some Deligne-Mumford curve, such that $\rho$ is pulled back along $f$.
\end{theorem}
Combining this with the proof of \autoref{proposition:AG-interpretation-higher-genus}, we have the following:
\begin{proposition}
Let $\mathscr{C}\to S$ be a family of $n$-punctured curves of genus $g$, with $S$ a smooth variety. Let $s\in  S$ be a point, and set $C=\mathscr{C}_s$. Then letting $M_B(C, \on{SL}_2)^{\text{non-deg}}$ be the subvariety of $M_B(C, 2)$ consisting of Zariski-dense $\on{SL}_2$-local systems on $C$, every irreducible component $Z$ of $$(M_B(C_b, \on{SL}_2)^{\text{non-deg}})^{\pi_1(S,s)}$$ is motivic in the following sense. Either
\begin{enumerate}
	\item $Z$ is a point, and corresponds to a local system of geometric origin, or
	\item $Z$ consists of local systems pulled back from a fixed Deligne-Mumford curve.
\end{enumerate}
\end{proposition}
Note that the two conditions in the result above are not mutually exclusive.

Finally, we give a conjectural description of invariant subvarieties along the lines of our non-linear analogue of the $p$-curvature conjecture, \autoref{conjecture:precisification}.
\begin{conjecture}\label{conjecture:p-curvature-invariant-subvarieties}
	Let $\mathscr{X}\to S$ be a smooth proper morphism over a finitely-generated integral $\mathbb{Z}$-algebra $R$, $s\in S$ an $R$-point, and $Z\subset \mathscr{M}_{dR}(\mathscr{X}/S, r)_s$ a closed substack. Then $Z(\mathbb{C})$ is invariant under a finite index subgroup of $\pi_1(S, s)$ if its formal isomonodromic deformation has an integral model.
	\end{conjecture}
This is meant to be the higher-dimensional analogue of \autoref{conjecture:precisification}; it specializes to that statement if $Z$ is a point. It is arguably the non-abelian analogue of \cite[Conjecture 9.2]{katz1982conjecture}, which aims to characterize the identity component of the Zariski-closure of the monodromy group of an ODE in terms of its $p$-curvatures. This latter conjecture is in fact equivalent to the classical $p$-curvature conjecture, \autoref{conjecture:gk-p-curvature}, by \cite[Theorem 10.2]{katz1982conjecture}. On the other hand, we do not know how to reduce \autoref{conjecture:p-curvature-invariant-subvarieties} to \autoref{conjecture:precisification}.
\subsubsection{Geometric subgroups of the mapping class group}
We conclude with a brief discussion of some questions that seem broadly relevant to the analysis of the ``non-abelian monodromy" of $\pi_1(\mathscr{M})$ on $Y(g,n,r)$ associated to an arbitrary family of $n$-punctured curves of genus $g$, $q: \mathscr{C}\to\mathscr{M}$, with $\mathscr{M}$ smooth. As before, this action factors through the map of fundamental groups $$\pi_1(\mathscr{M})\to \pi_1(\mathscr{M}_{g,n})=\on{Mod}_{g,n},$$ induced by the classifying map $\mathscr{M}\to \mathscr{M}_{g,n}$. We call the image of such a map a \emph{geometric} subgroup of the mapping class group $\on{Mod}_{g,n}$. Thus a natural question becomes:
\begin{question}
	What are the geometric subgroups of $\on{Mod}_{g,n}$?
\end{question}
There are some evident restrictions on such subgroups. For example, by the Torelli theorem, non-trivial geometric subgroups of $\on{Mod}_{g,n}$ cannot be contained in the Torelli group. Indeed, if $\mathbb{V}$ is \emph{any} variation of Hodge structure on $\mathscr{M}_{g,n}$ with quasi-finite period map, then $\mathbb{V}$ yields an analogous restriction: the restriction of $\mathbb{V}$ to any geometric subgroup of $\on{Mod}_{g,n}$ must have infinite monodromy. Moreover any variation of Hodge structure whatsoever on $\mathscr{M}_{g,n}$ must have semisimple monodromy when restricted to a geometric subgroup (as variations of Hodge structure are always semisimple).

It seems natural to ask for a geometric analogue of this last observation, which would be useful in approaching \autoref{conjecture:katzarkov-nabm}. Let $\gamma\subset \Sigma_{g,n}$ be a simple closed curve, and let $\Gamma_\gamma\subset \on{Mod}_{g,n}$ be the centralizer of the Dehn twist about $\gamma$. 
\begin{question}
	Fix a simple closed curve $\gamma\subset \Sigma_{g,n}$. Can an infinite geometric subgroup of $\on{Mod}_{g,n}$ be conjugate to a subgroup of $\Gamma_\gamma$?
\end{question}
Here we view the subgroups $\Gamma_\gamma$ as analogues of parabolic subgroups of $\on{GL}_r(\mathbb{C})$; they are the fundamental groups of punctured neighborhoods of boundary divisors in the Deligne-Mumford compactification of $\mathscr{M}_{g,n}$.
\bibliographystyle{alpha}
\bibliography{bibliography-mmm.bib}

\end{document}